%% file: adaptive_defeaturing.tex
\documentclass{article}

\usepackage{graphicx}
\usepackage{amssymb}
\usepackage[pdftex]{hyperref}
\usepackage{xcolor}
\hypersetup{
	colorlinks,
	linkcolor={red!50!black},
	citecolor={blue!90!black},
	urlcolor={blue!80!black}
}
\usepackage{cite}
\usepackage[hmargin={2.5cm,2.5cm}, vmargin={3cm,3cm}, dvips]{geometry}

\usepackage{fancyhdr}
\newcommand{\spn}[1]{\mathrm{span}\left\{#1\right\}}

\newcommand{\support}[1]{\mathrm{supp}\left(#1\right)}

\usepackage{amsmath} 
\usepackage{amssymb} 
\usepackage{amsthm} 
\usepackage{array}
\usepackage{placeins}
\usepackage{enumerate}
\usepackage{esint}
\usepackage{graphics}
\usepackage{graphicx}
\usepackage{pifont}
\usepackage{tikz}
\usetikzlibrary{3d}
\usetikzlibrary{arrows}
\usetikzlibrary{patterns}

\usetikzlibrary{shapes,positioning,shadows,trees}
\usetikzlibrary{arrows.meta,chains}

\usepackage{pgfplots}
\pgfplotsset{compat=1.15}
\usepackage{caption}
\usepackage{subcaption}
\usepackage{listings}
\usepackage[shortlabels]{enumitem}
\usepackage{multirow}
\usepackage{multicol}
\usepackage{wrapfig}

\usepackage{afterpage}
\usepackage{mathrsfs}
\usepackage{mathtools}
\usepackage{tabularx}

\newtheoremstyle{droit}
{}
{}
{\upshape}
{}
{\bfseries}
{}
{ }
{}
\newtheoremstyle{italique}
{}
{}
{\itshape}
{}
{\bfseries}
{}
{ }
{}
\theoremstyle{italique}
\newtheorem{theorem}{Theorem}[section]

\newtheorem{lemma}[theorem]{Lemma}

\theoremstyle{droit}

\newtheorem{remark}[theorem]{Remark}

\newtheorem{definition}[theorem]{Definition}
\newtheorem{assumption}[theorem]{Assumption}

\usepackage{mathtools}

\newcommand{\RN}[1]{%
	\textup{\uppercase\expandafter{\romannumeral#1}}%
}

\makeatletter
\newcommand{\vast}{\bBigg@{4}}
\newcommand{\Vast}{\bBigg@{5}}
\makeatother
\usepackage{geometry}
\usetikzlibrary{shapes,backgrounds}

\newcommand{\setminussign}{{\mathrm r}}
\newcommand{\intersign}{{\mathrm s}}

\newcommand{\trim}{\text{cut}}
\newcommand{\untr}{\text{act}}
\newcommand{\hQFmin}{h_{F}^\mathrm{min}}
\newcommand{\hQSmin}{h_{S}^\mathrm{min}}
\newcommand{\deltaQSmin}{\delta_{S}^\mathrm{min}}

\usepackage{algorithm,algcompatible}
\usepackage{algpseudocode}
\algnewcommand{\lst}{\texttt{lst}}
\algnewcommand{\slst}{\texttt{slst}}
\algnewcommand{\SEND}{\textbf{send}}
\newsavebox{\algleft}
\newsavebox{\algright}

\definecolor{darkgreen}{rgb}{0,0.4,0} 
\definecolor{darkbrown}{rgb}{0.5, 0.396, 0.09}

\definecolor{c1}{rgb}{0.0, 0.4196078431372549, 0.6431372549019608}
\definecolor{c2}{rgb}{1.0, 0.5019607843137255, 0.054901960784313725}
\definecolor{c3}{rgb}{0.6705882352941176, 0.6705882352941176,
	0.6705882352941176} \definecolor{c}{rgb}{0.34901960784313724, 0.34901960784313724, 0.34901960784313724}
\definecolor{c4}{rgb}{0.37254901960784315, 0.6196078431372549,
	0.8196078431372549} 
\definecolor{c5}{rgb}{0.5372549019607843, 0.5372549019607843,
	0.5372549019607843} 
\definecolor{c6}{rgb}{1.0, 0.7372549019607844, 0.4745098039215686}
\definecolor{c7}{rgb}{0.8117647058823529, 0.8117647058823529,
	0.8117647058823529}

\newsavebox{\imagebox}

\usepackage{tikz-3dplot}
\tikzset{declare function={
		vcrity(\ph,\th)=atan2(sin(\th)*sin(\ph),min(cos(\ph),-1/sqrt(2))*cos(\th));
		vcritz(\ph,\th)=\ph;
	},pics/ycylinder/.style={code={
			\tikzset{3d/cylinder/.cd,#1}
			\def\pv##1{\pgfkeysvalueof{/tikz/3d/cylinder/##1}}
			\pgfmathsetmacro{\vmin}{vcrity(\tdplotmainphi,\tdplotmaintheta)}
			\pgfmathsetmacro{\vmax}{\vmin-180}
			\path[3d/cylinder/mantle]
			let \p1=($(0,1,0)-(0,0,0)$),\n1={atan2(\y1,\x1)} in
			[shading angle=\n1]
			plot[variable=\t,domain=\vmin:\vmax,smooth]
			({\pv{r}*cos(\t)},0,{\pv{r}*sin(\t)})
			-- 
			plot[variable=\t,domain=\vmax:\vmin,smooth]
			({\pv{r}*cos(\t)},\pv{h},{\pv{r}*sin(\t)})
			--cycle;
			\pgfmathtruncatemacro{\itest}{sign(cos(\tdplotmainphi))}
			\ifnum\itest=-1
			\path[3d/cylinder/top] plot[variable=\t,domain=0:360,smooth cycle]
			({\pv{r}*cos(\t)},\pv{h},{\pv{r}*sin(\t)}) ;
			\fi
			\ifnum\itest=1
			\path[3d/cylinder/top] plot[variable=\t,domain=0:360,smooth cycle]
			({\pv{r}*cos(\t)},0,{\pv{r}*sin(\t)}) ;
			\fi
	}},3d/.cd,cylinder/.cd,r/.initial=1,h/.initial=1,
	mantle/.style={draw},top/.style={draw}}

\makeatletter
\def\customrevertcolormap#1{%
	\pgfplotsarraycopy{pgfpl@cm@#1}\to{custom@COPY}%
	\c@pgf@counta=0
	\c@pgf@countb=\pgfplotsarraysizeof{custom@COPY}\relax
	\c@pgf@countd=\c@pgf@countb
	\advance\c@pgf@countd by-1 %
	\pgfutil@loop
	\ifnum\c@pgf@counta<\c@pgf@countb
	\pgfplotsarrayselect{\c@pgf@counta}\of{custom@COPY}\to\pgfplots@loc@TMPa
	\pgfplotsarrayletentry\c@pgf@countd\of{pgfpl@cm@#1}=\pgfplots@loc@TMPa
	\advance\c@pgf@counta by1 %
	\advance\c@pgf@countd by-1 %
	\pgfutil@repeat
}%

\makeatother
\usepgfplotslibrary{colorbrewer}

\newcommand{\logLogSlopeReverseTriangle}[5]
{
	
	\pgfplotsextra
	{
		\pgfkeysgetvalue{/pgfplots/xmin}{\xmin}
		\pgfkeysgetvalue{/pgfplots/xmax}{\xmax}
		\pgfkeysgetvalue{/pgfplots/ymin}{\ymin}
		\pgfkeysgetvalue{/pgfplots/ymax}{\ymax}
		
		\pgfmathsetmacro{\xArel}{#1}
		\pgfmathsetmacro{\yArel}{#3}
		\pgfmathsetmacro{\xBrel}{#1+#2}
		\pgfmathsetmacro{\yBrel}{\yArel}
		\pgfmathsetmacro{\xCrel}{\xArel}
		
		\pgfmathsetmacro{\lnxB}{\xmin*(1-(#1-#2))+\xmax*(#1-#2)} 
		\pgfmathsetmacro{\lnxA}{\xmin*(1-#1)+\xmax*#1} 
		\pgfmathsetmacro{\lnyA}{\ymin*(1-#3)+\ymax*#3} 
		\pgfmathsetmacro{\lnyC}{\lnyA+#4*(\lnxA-\lnxB)}
		\pgfmathsetmacro{\yCrel}{\lnyC-\ymin)/(\ymax-\ymin)}
		
		\coordinate (A) at (rel axis cs:\xArel,\yArel);
		\coordinate (B) at (rel axis cs:\xBrel,\yBrel);
		\coordinate (C) at (rel axis cs:\xCrel,\yCrel);
		
		\draw[#5]   (A)-- node[pos=0.5,anchor=north] {\scriptsize 1}
		(B)-- 
		(C)-- node[pos=0.5,anchor=east] {\scriptsize #4}
		cycle;
	}
}

\makeatletter
\newtoks\pgf@ps@toks
\newcount\c@pgf@ps

\def\pgf@ps@sp{ }

\def\pgf@ps@esettoks#1{\edef\pgf@ps@tmp{#1}\pgf@ps@toks\expandafter{\pgf@ps@tmp}}

\def\pgf@ps@repop#1#2{%
	\c@pgf@countb=#2\relax%
	\def\pgf@ps@op{#1}%
	\def\pgf@ps@ops{}\pgf@ps@@repop}
\def\pgf@ps@@repop{%
	\ifnum\c@pgf@countb<1\relax%
	\else%
	\edef\pgf@ps@ops{\pgf@ps@op\pgf@ps@ops}%
	\advance\c@pgf@countb by-1\relax%
	\expandafter\pgf@ps@@repop%
	\fi%
}

\def\pgf@ps@generate@ps{%
	\c@pgf@counta=\pgf@ps@ncol\relax%
	\c@pgf@countb=\c@pgf@counta%
	\advance\c@pgf@countb by-1\relax%
	\pgf@ps@esettoks{ \noexpand\pgf@ps@interp{col@\the\c@pgf@counta}{col@\the\c@pgf@countb} }%
	\pgfmathloop
	\ifnum\c@pgf@counta<2\relax%
	\else%
	\c@pgf@countb=-\c@pgf@counta%
	\advance\c@pgf@countb by\pgf@ps@ncol\relax%
	\advance\c@pgf@counta by-1\relax%
	\pgf@ps@repop{pop\pgf@ps@sp}{\c@pgf@countb}%
	\c@pgf@countb=\c@pgf@counta%
	\advance\c@pgf@countb by-1\relax%
	\ifnum\c@pgf@countb=0\relax%
	\c@pgf@countb=\pgf@ps@ncol\relax%
	\fi%
	\pgf@ps@esettoks{ \the\c@pgf@counta\pgf@ps@sp eq { \pgf@ps@ops \noexpand\pgf@ps@interp{col@\the\c@pgf@counta}{col@\the\c@pgf@countb} }{ \the\pgf@ps@toks } ifelse}%
	\repeatpgfmathloop%
	\c@pgf@counta=\pgf@ps@ncol\relax%
	\advance\c@pgf@counta by-2\relax%
	\pgf@ps@repop{dup\pgf@ps@sp}{\c@pgf@counta}%
	\pgf@ps@esettoks{ \pgf@ps@ops \the\pgf@ps@toks }%
}

\def\pgf@ps@colorstorgb#1{%
	\c@pgf@ps=1\relax%
	\pgfutil@for\pgf@ps@:={#1}\do{%
		\pgf@ps@coltorgb{\pgf@ps@}{col@\the\c@pgf@ps}%
		\advance\c@pgf@ps by1\relax}%
}

\def\pgf@ps@coltorgb#1#2{%
	\edef\pgf@ps@marshal{\noexpand\pgfshadecolortorgb{#1}}%
	\expandafter\pgf@ps@marshal\expandafter{\csname#2\endcsname}%
}

\def\pgf@ps@rgb#1{\csname#1\endcsname}
\def\pgf@ps@interp#1#2{%
	\pgf@ps@rgb{#1red} mul exch \pgf@ps@rgb{#2red} mul add
	5 1 roll
	\pgf@ps@rgb{#1green} mul exch \pgf@ps@rgb{#2green} mul add
	3 1 roll
	\pgf@ps@rgb{#1blue} mul exch \pgf@ps@rgb{#2blue} mul add
}

\def\pgfdeclarecolorwheelshading#1#2#3{%
	\pgf@ps@getcols{#3}%
	\pgfmathparse{mod(#2+360/\pgf@ps@ncol-90,360)}%
	\pgf@x=\pgfmathresult pt\relax%
	\ifdim\pgf@x<0pt\relax%
	\advance\pgf@x by360pt\relax%
	\fi%
	\edef\pgf@ps@rot{\pgfmath@tonumber{\pgf@x}}%
	\pgf@ps@generate@ps%
	\pgf@ps@esettoks{%
		\noexpand\pgfdeclarefunctionalshading[#3]{#1}%
		{\noexpand\pgfpoint{-50bp}{-50bp}}{\noexpand\pgfpoint{50bp}{50bp}}%
		{\noexpand\pgf@ps@colorstorgb{#3}}%
		{%
			2 copy abs exch abs add 0.0001 ge { atan } { pop } ifelse
			\pgf@ps@rot\pgf@ps@sp add dup 360 ge { -360 add } { } ifelse
			360 \pgf@ps@ncol\pgf@ps@sp 
			div div dup floor dup 3 1 roll neg add dup neg 1 add exch
			2 copy 2 copy 7 -1 roll 1 add
			\the\pgf@ps@toks}}%
	\edef\pgf@ps@marshal{\the\pgf@ps@toks}%
	\pgf@ps@marshal}

\def\pgf@ps@getcols#1{%
	\c@pgf@ps=0\relax%
	\pgfutil@for\pgf@ps@:={#1}\do{\advance\c@pgf@ps by1}%
	\edef\pgf@ps@ncol{\the\c@pgf@ps}%
}
\pgfdeclarecolorwheelshading{rainbow}{-45}{c4,c,c2}

\newcommand{\changes}[1]{{#1}}

\begin{document}



\title{Adaptive analysis-aware defeaturing: the case of Neumann boundary conditions}


\author{{A. Buffa$^{1}$, O. Chanon$^2$, R. V\'azquez$^{3}$}\\ \\\vspace{-0.2cm}
	\footnotesize{$^1$ MNS, Institute of Mathematics, \'Ecole Polytechnique F\'ed\'erale de Lausanne, Switzerland}\\
	\footnotesize{Istituto di Matematica Applicata e Tecnologie Informatiche `E. Magenes' (CNR), Pavia, Italy}\\
	\footnotesize{$^2$ ABB Corporate Research Center, ABB Schweiz AG, Baden-D\"attwil, Switzerland; ondine.chanon@ch.abb.com}\\\vspace{-0.2cm}
	\footnotesize{$^3$ Departamento de Matemática Aplicada and Centro de de Investigaci\'on y Tecnolog\'ia Matem\'atica de Galicia (CITMAga)}\\\vspace{-0.2cm}
	\footnotesize{Universidade de Santiago de Compostela, Santiago de Compostela, Spain}
}

\maketitle
\vspace{-0.8cm}
\noindent\rule{\linewidth}{0.4pt}
\thispagestyle{fancy}
\begin{abstract}
	Removing geometrical details from a complex domain is a classical operation in computer aided design for simulation and manufacturing. This procedure simplifies the meshing process, and it enables faster simulations with less memory requirements. However, depending on the partial differential equation that one wants to solve, removing some important geometrical features may greatly impact the solution accuracy. 
	Unfortunately, the effect of geometrical simplification on the accuracy of the problem solution is often neglected or its evaluation is based on engineering expertise, only due to the lack of reliable tools. It is therefore important to have a better understanding of the effect of geometrical model simplification, also called defeaturing, to improve our control on the simulation accuracy along the design and analysis phases.
	In this work, we consider as a model problem the Poisson equation \changes{on a geometry with Neumann features, we consider some finite element discretization of it,} and we build an adaptive strategy that is twofold. Firstly, it is able to perform geometrical refinements, that is, to choose at each iteration step which geometrical feature is important to obtain an accurate solution. \changes{Secondly, it performs standard mesh refinements; since the geometry changes at each iteration, the algorithm is designed to be used with an immersed method}. To drive this adaptive strategy, we introduce an \textit{a posteriori} estimator of the energy error between the exact solution defined in the exact fully-featured geometry, and the numerical approximation of the solution defined in the defeatured geometry. The reliability of the estimator is proven for very general \changes{(potentially trimmed multipatch) geometric configurations, and in particular for isogeometric analysis with hierarchical B-splines}. Finally, numerical experiments are performed to validate the presented theory and to illustrate the capabilities of the proposed adaptive strategy. 
\end{abstract}


\section{Introduction} 
The interoperability between the design of complex objects and the numerical resolution of partial differential equations (PDEs) on those objects has been a major challenge since the creation of finite element (FE) methods. Isogeometric analysis (IGA) was introduced in the seminal work \cite{igabasis} to simplify this issue: the main idea of IGA consists in employing the basis functions used to describe geometries in computer aided design (CAD) software for the numerical analysis of the problem, namely B-splines or variants thereof. This new paradigm has open the road to an extensive amount of research; the interested reader is referred to \cite{igabook} for a review of the method, and to \cite{igaappli} for a review of a wide range of real world applications on which the method has shown its strong capabilities. The mathematical foundations of IGA have been developed in \cite{igahref,igaanalysis}, and numerous implementations have now been developed, see \cite{nguyen2015isogeometric,geopdes} for instance. 

Moreover, different extensions of the original IGA method have been developed in order to deal with geometries of increased complexity: the main developments include for instance multipatch domains \cite{multipatch,bracco2020isogeometric} or geometries obtained by Boolean operations such as intersections (trimming) \cite{trimming,trimmingbletzinger,antolinvreps,weimarrusigantolin,antolin2021quadrature} or unions \cite{antolinwei,unionkargaran}. The related engineering literature includes for instance the FE cell method with IGA on complex geometries \cite{rankcomplexgeom,rank2012geometric}, and the analysis of shell structures \cite{bletzinger1,bletzinger2,coradello}. 
However, B-spline basis functions have a tensor-product structure that hinders the capability of efficiently capturing localized properties of the PDE solution in small areas of the computational domain. In order to overcome this issue, the construction of locally refined splines and their use in adaptivity has been a very active area of research. The developments in this direction have recently been reviewed in \cite{reviewadaptiveiga}. In particular, we are interested here in using hierarchical B-splines (HB-splines) \cite{forsey1988hierarchical,kraft1997adaptive,vuong} and their variant called truncated hierarchical B-splines (THB-splines) \cite{giannelli2016thb,thbgiannelli,giannelli2014strongly}, in an adaptive IGA framework \cite{buffagiannelli1,trimmingest,trimmedshells}.

However, dealing with very complex geometries remains challenging, and even the most recent and most efficient methods may come at a prohibitive cost. The only geometric description of complex domains may already require a very large number of degrees of freedom, but not all of them are necessary to perform an accurate analysis. This is where analysis-aware defeaturing handily comes into play: if one is able to determine the geometrical features of a complex geometry that have the least influence on the accuracy of the PDE solution at hand, then one can simplify the geometric description of the domain by defeaturing, leading to an easier, cheaper, and still accurate analysis. By doing so, the meshing step is also simplified, as it is otherwise often unfeasible. Moreover, adding and removing geometrical features to a design in order to meet engineering requirements is a typical process in simulation-based design for manufacturing. It is thus important to consider the impact of such geometrical changes during the analysis phase of the problem. 

Defeaturing has been approached first and foremost using subjective \textit{a priori} criteria relying on the expertise of engineers, on geometrical considerations, or on some knowledge about the mechanical problem at hand (laws of conservation, constitutive equations, etc.), see \cite{surveymodelsimpl, fineremondinileon2000, foucault2004mechanical}. However, the need of \textit{a posteriori} estimators soon appeared to be essential for the automation of simulation-based design processes in order to evaluate the error induced by defeaturing, as they only use the results of the analysis in the defeatured geometry. Therefore, many different \textit{a posteriori} criteria can be found in the literature: the interested reader is referred to \cite{ferrandesgiannini2009,tsa1,tsa2,gopalakrishnan2007formal,gopalakrishnansuresh2008,turevsky2009efficient,TANG2013413,ligaomartin2011,ligao2011,ligaozhang2012,ligaomartin2013}. 
Nevertheless, very few of those works come with a sound theoretical theory, and most of them rely on some heuristics. \changes{A noteworthy early defeaturing and coarsening approach, known as composite finite elements, was devised by Hackbusch and Sauter in \cite{hackbusch1997composite,hackbusch1997composite2}. This method operates under the assumption that defeaturing occurs primarily due to the inherent limitations of the mesh in accurately representing geometric features.} Very close works include the study of heterogeneous and perforated material \cite{odenvemaganti2000,vegamanti2004,carstensensauter2004}, and the study of the error induced by boundary approximation in numerical methods \cite{repinsautersmolianski,dorfler1998adaptive}. \changes{However and in contrast, our approach aims to address defeaturing as a distinct process, independent of any discretization technique. By doing so, we are able to separate the error contributions coming from defeaturing and from the numerical approximation error.}

The present work strongly relies on the previous paper \cite{paper1defeaturing} by the authors, where a precise mathematical framework for analysis-aware defeaturing is introduced in the context of Poisson's equation. In that paper, an efficient and reliable \textit{a posteriori} estimator of the defeaturing error in energy norm is derived, based on the variational formulations of the exact and defeatured problems. It has then been generalized in \cite{paper3multifeature} to geometries with multiple features, in the context of linear elasticity and Stokes equations. 

In the following, we are interested in the numerical approximation of the defeatured problem presented in \cite{paper1defeaturing}. We first introduce an \textit{a posteriori} estimator of the overall error between the exact solution of the PDE defined in the exact domain, and the numerical approximation of the solution of the corresponding PDE defined in the defeatured domain. The proposed estimator \changes{can be adapted to any FE method used for the discretization of the given problem, and it} is able to control both the defeaturing error and the numerical error contributions of the overall numerical defeatured error. This allows us to design an adaptive algorithm which is twofold. On the one hand, the algorithm is able to perform some standard mesh refinement on the hierarchical mesh to reduce the discretization error. On the other hand, starting from a fully defeatured geometry on which the solution can be computed, the adaptive strategy is also able to perform some geometric refinement to reduce the error due to defeaturing. That is, at each iteration step, the features that most affect the solution accuracy are marked to be added to the geometrical model at the next iteration. We then concentrate \changes{in particular} on the THB-spline based IGA numerical approximation of the elliptic problem at hand. We detail the proposed adaptive strategy in this framework, and give a mathematical proof of the reliability of the error estimator when IGA is used.

\changes{Our approach is similar to the one of Burman et al. in \cite{burman2019posteriori}, but we aim at setting a framework in which defeaturing is considered independently from the numerical setting. That is, the adaptive method that we propose does not depend on the chosen numerical method, and it allows to locally determine where the simplification of the geometry, which does not necessarily come from the meshing process, affects the solution accuracy.}

In the present article, we first define the considered defeatured problem and its numerical approximation \changes{with a FE method of choice,} in Section~\ref{sec:pbstatement}. Then in Section~\ref{sec:adaptive}, we design a combined mesh and geometric adaptive strategy in the context of analysis-aware defeaturing. Subsequently, in Section~\ref{sec:iga}, we review the basic concepts of IGA with THB-splines, \changes{as a natural discretization method to deal with geometric defeaturing}. We then prove in Section~\ref{sec:combinedest} the reliability of the \textit{a posteriori} estimator of the discrete defeaturing error that drives the proposed adaptive strategy, under some reasonable assumptions. \changes{All the details are given in the case in which the chosen numerical method is IGA, followed by a discussion on the straight-forward adaptation of the proof to the standard $C^0$-FE method.} We then present in Section~\ref{sec:gen} the extension of IGA to trimmed and multipatch geometries as this allows to describe complex geometrical models, and we use these spline technologies to explain in details the proposed adaptive strategy when IGA is used. 
This strategy is illustrated in Section~\ref{sec:numexp} by various numerical experiments that simultaneously validate and extend the presented theory. To finish, conclusions are drawn in Section~\ref{sec:ccl}, and some technical lemmas used throughout the paper are given in Appendix~\ref{sec:appendix}. \\

\textit{Notation}\\
Let $n\in\{2,3\}$, let $\omega$ be an open $k$-dimensional manifold in $\mathbb{R}^n$, $k\leq n$, and let $\varphi\subset \partial \omega$. Moreover, let $y$ and $z$ be any functions, let $A$ be a functional subspace of the Sobolev space $H^1(\omega)$, let $1\leq p \leq \infty$, and let $s\in\mathbb R$. In this article, we will use the notation summarized in Table~\ref{tbl:notation}, \changes{and the definitions of Sobolev norms and semi-norms correspond to the ones in~\cite[Definition~1.3.2.1]{grisvard}.}
\begin{table}[h!]
\begin{tabularx}{\textwidth}{l|X}
	Notation & Definition \\ \hline
	$|\omega|$ & $k$-dimensional measure of $\omega$. \\
	$\overline{\omega}$ & Closure of $\omega$.\\
	$\mathrm{int}(\omega)$ & Interior of $\omega$.\\
	$\mathrm{conn}(\omega)$ & Set of connected components of $\omega$.\\
	$\mathrm{diam}(\varphi)$ & Manifold diameter $=\max_{\xi, \eta \in \varphi} \rho(\xi, \eta)$, where $\rho(\xi, \eta)$ is the infimum of lengths of continuous piecewise $C^1$ paths between $\xi$ and $\eta$ in $\partial \omega \changes{\, \supset \varphi}$.  \\
	$\overline{z}^\omega$ & Average of $z$ over $\omega$. \\
	$\support{z}$ & Open support of $z$.\\
	$\|z\|_{0,\omega}$ & Norm of $z$ in $L^2(\omega)$.\\
	$\|z\|_{s,\omega}$ & Norm of $z$ in the Sobolev space $H^s(\omega)$ of order $s$.\\
	$|z|_{s,\omega}$ & Semi-norm of $z$ in $H^s(\omega)$.\\
	$\mathrm{tr}_\varphi(z)$ & Trace of $z$ on $\varphi\subset\partial \omega$.\\
	$\mathrm{tr}_\varphi(A)$ & Trace space on $\varphi$ of the functions in $A$, i.e. $\left\{ \mathrm{tr}_\varphi(z) : z\in A \right\}$.\\
	$H^1_{y,\varphi}(\omega)$ & Set of functions $z\in H^1(\omega)$ such that $\mathrm{tr}_\varphi(z) = y$.\\
	$\#S$ & Cardinality of the set $S$.
\end{tabularx}
\caption{Notation used throughout the article.}\label{tbl:notation}
\end{table}

In the remaining part of this article, the symbol $\lesssim$ will be used to mean any inequality which is independent of the size of the features, of the mesh size $h$, and of the number of hierarchical levels (see Section \ref{ss:hbs}). However, those inequalities may depend \changes{on the dimension $n$, on the problem data (geometry, boundary conditions and functional data), on the degree (and regularity) of the discretization method}, and on the shape of the features and of the mesh elements. Moreover, we will write A $\simeq B$ whenever $A \lesssim B$ and $B \lesssim A$. 

\section{Defeaturing problem statement} \label{sec:pbstatement}
\input{pbstatement}

\section{Adaptive analysis-aware defeaturing strategy} \label{sec:adaptive}
\input{adaptivedefeat}

\section{A short review of isogeometric analysis and spline technologies} \label{sec:iga}
\input{hierarchicaliga}

\section{Reliability of the discrete defeaturing error estimator}\label{sec:combinedest}
\input{combinedest} 

\section{The adaptive strategy on complex spline geometries}\label{sec:gen}
\input{gen2}
\input{refineiga}

\section{Numerical  experiments} \label{sec:numexp}
\input{numexamples2}

\section{Conclusions} \label{sec:ccl}
In this work, we have designed an adaptive analysis-aware defeaturing strategy that both performs standard mesh refinement and geometric refinement. 
The latter means that features are added to the simplified geometry when their absence is responsible for most of the solution accuracy loss. 
To steer this adaptive strategy, we have introduced a novel \textit{a posteriori} error estimator of the energy norm of the discrete defeaturing error, i.e., the error between the exact solution $u$ computed in the exact domain $\Omega$, and the numerical solution $u_0^h$ computed in the defeatured domain $\Omega_0$. 

Taking Poisson's equation as model problem, we have then considered IGA with THB-splines as a well-suited numerical method to perform the proposed algorithm. In this framework, and building upon the results of \cite{paper1defeaturing,paper3multifeature} obtained in continuous spaces, we have demonstrated the reliability of the proposed discrete defeaturing error estimator for a wide range of geometries in $\mathbb R^n$, $n=2,3$, containing $N_f\geq 1$ general complex features. To perform the adaptive strategy, we have considered trimmed and multipatch domains as mesh-preserving methods in the IGA context. 

The proposed estimator and the derived adaptive strategy have been tested on a wide range of numerical experiments. In all of them, the estimator acts as an excellent approximation of the discrete defeaturing error, and it correctly drives the adaptive strategy. In particular, it is able to correctly weight the impact of defeaturing with respect to the numerical approximation of the defeatured solution. \changes{However, the proposed estimator relies on the heuristic evaluation of the constants $\alpha_D$ and $\alpha_N$. To avoid this limitation, one could use \textit{a posteriori} error estimators based on equilibrated fluxes instead of residual ones. This is the subject of a subsequent work, see \cite{defeaturingequilibratedflux}.}

\appendix
\section{Appendix} \label{sec:appendix}
\input{appendix}

\section*{Acknowledgment}
The authors gratefully acknowledge the support of the European Research Council, via the ERC AdG project CHANGE n.694515. 
R. V\'azquez and O. Chanon also thank the support of the Swiss National Science Foundation via the project HOGAEMS n.200021\_188589, and the project n.P500PT\_210974.


\addcontentsline{toc}{chapter}{Bibliography}
\bibliography{bib2}
\bibliographystyle{ieeetr}








\end{document}


%% file: pbstatement.tex
In this section, we precisely define the defeatured problem, following \cite{paper1defeaturing} and \cite{paper3multifeature}. We first define it in Sobolev spaces for a geometry containing a single feature, then we discretize the problem using a FE space, and we finally generalize this setting to a geometry containing an arbitrary number of distinct features. 

\subsection{Defeaturing model problem} \label{ss:continuouspbstatement}
Let $\Omega\subset \mathbb R^n$ be an open Lipschitz domain, let $\mathbf n$ be the unitary outward normal to $\Omega$, and let $\partial \Omega = \overline{\Gamma_D}\cup \overline{\Gamma_N}$ with $\Gamma_D\cap\Gamma_N = \emptyset$, $|\Gamma_D| > 0$. Then, let $g_D\in H^{\frac{3}{2}}(\Gamma_D)$, $g\in H^{\frac{1}{2}}(\Gamma_N)$ and $f\in L^2\left(\Omega\right)$, and assume we aim at solving the following Poisson's problem in the exact geometry $\Omega$: find $u\in H^1(\Omega)$, the weak solution of 
\begin{align} \label{eq:originalpb}
\begin{cases}
-\Delta u = f &\text{ in } \Omega \\
u = g_D &\text{ on } \Gamma_D \\
\displaystyle\frac{\partial u}{\partial \mathbf{n}} = g &\text{ on } \Gamma_N,\vspace{0.1cm}
\end{cases}
\end{align}
that is, $u\in H^1_{g_D,\Gamma_D}(\Omega)$ satisfies for all $v\in H^1_{0,\Gamma_D}(\Omega)$, 
\begin{equation} \label{eq:weakoriginalpb}
\int_\Omega \nabla u \cdot \nabla v \,\mathrm dx = \int_\Omega fv \,\mathrm dx + \int_{\Gamma_N} g v \,\mathrm ds.
\end{equation}

Moreover, we assume that $\Omega$ is a complicated domain, meaning that it contains a geometrical detail of smaller scale called feature and denoted by $F$. More precisely, we suppose that $F= \text{int}\left(\overline{F_\mathrm n} \cup \overline{F_\mathrm p}\right)$, where, if we let $$\Omega_\star := \Omega\setminus\overline{F_\mathrm p},$$
then $F_\mathrm p\subset \Omega$ and $\overline{F_\mathrm n}\cap \overline{\Omega_\star} \subset \partial \Omega_\star$, as illustrated in Figure \ref{fig:G0ppty1}. That is, we say that $F$ contains a negative component $F_\mathrm n$ in which material is removed, and a positive component $F_\mathrm p$ in which material is in excess. 
In particular, if $F_\mathrm p = \emptyset$ and $F = F_\mathrm n$, then $F$ is said to be negative, while if $F_\mathrm n = \emptyset$ and $F = F_\mathrm p$, then $F$ is said to be positive. We also assume that $F_\mathrm n$, $F_\mathrm p$ and $\Omega_\star$ are open Lipschitz domains. 

The defeatured problem is formulated on a defeatured domain $\Omega_0$ where features are removed: holes are filled with material, and protrusions are cut out of the computational domain. More precisely, the defeatured geometry $\Omega_0 \subset \mathbb R^n$ reads:
\begin{equation} \label{eq:defomega0}
\Omega_0 :=\text{int}\left(\overline{\Omega_\star}\cup \overline{F_\mathrm n}\right).
\end{equation}
We assume $\Omega_0$ to be an open Lipschitz domain. Note that in general, $\Omega_\star \subset \left(\Omega \cap \Omega_0\right)$, but if $F$ is completely negative or positive, then the two sets are equal.

\begin{figure}
	\centering
	\begin{subfigure}[t]{0.4\textwidth}
		\begin{center}
			\begin{tikzpicture}[scale=4.5]
			\fill[c3,opacity=0.2] (2.7,1) -- (3,1) -- (3,0.5) -- (2,0.5) -- (2,1) -- (2.3,1) -- (2.3,1.3) -- (2.4,1.3) -- (2.6,1.1) -- (2.6,1) -- (2.4,1) -- (2.4,0.8) -- (2.7,0.8) --cycle;
			\draw[thick] (2.7,1) -- (3,1) -- (3,0.5) -- (2,0.5) -- (2,1) -- (2.3,1) ;
			\draw[c1,thick,dashed] (2.4,1) -- (2.4,0.8) -- (2.7,0.8) -- (2.7,1);
			\draw[c1,thick] (2.4,1) -- (2.6,1) -- (2.6,1.1) -- (2.4,1.3) -- (2.3,1.3) -- (2.3,1);
			\draw (2.23,0.7) node{$\Omega$} ;
			\draw[c1,thick] (2.3,1.15) node[left]{$\gamma_\mathrm p$} ;
			\draw[c1,thick] (2.55,0.8) node[below]{$\gamma_\mathrm n$} ;
			\end{tikzpicture}
			\caption{Domain $\Omega$ with a complex feature.}
		\end{center}
	\end{subfigure}
	~
	\begin{subfigure}[t]{0.4\textwidth}
		\begin{center}
			\begin{tikzpicture}[scale=4.5]
			\fill[c3, opacity=0.4] (2.3,1) rectangle (2.6,1.3);
			\fill[c3, opacity=0.1] (2,0.5) rectangle (3,1);
			\draw[thick] (2,0.5) -- (2,1) ;
			\draw[thick] (3,0.5) -- (3,1) ;
			\draw[thick] (2,0.5) -- (3,0.5);
			\draw[thick] (2,1) -- (2.3,1) ;
			\draw[thick] (2.7,1) -- (3,1) ;
			\draw[c2] (2.6,1.2) node[right]{$\tilde \gamma$};
			\draw (2.46,1.16)node{$\tilde F_\mathrm p$};
			\draw[c2,thick] (2.6,1.1) -- (2.6,1.3) -- (2.4,1.3);
			\draw[c1,dashed] (2.4,1) -- (2.4,0.8) ;
			\draw[c1,dashed] (2.7,1) -- (2.7,0.8) ;
			\draw[c1,dashed] (2.4,0.8) -- (2.7,0.8) ;
			\draw[c1,dashed] (2.6,1.1) -- (2.4,1.3);
			\draw (2.23,0.7) node{$\Omega_0$} ;
			\draw[c2,thick,densely dotted] (2.3,1) -- (2.4,1);
			\draw[c2] (2.33,1) node[below]{$\gamma_{0,\mathrm p}$};
			\draw[c2] (2.55,1) node[below]{$\gamma_{0,\mathrm n}$};
			\draw[c2, thick] (2.4,1) -- (2.6,1);
			\draw[c2, thick] (2.6,1) -- (2.7,1);
			\draw[c4,thick] (2.4,1.0025) -- (2.6,1.0025) -- (2.6,1.1);
			\draw[c4,thick] (2.4,1.3) -- (2.3,1.3) -- (2.3,1);
			\draw[c2,thick] (2.4,0.9975) -- (2.6,0.9975);
			\draw[c2,thick] (2.6,0.9975) -- (2.7,1);
			\draw[c4] (2.65,1.05) node{$\gamma_\intersign$};
			\end{tikzpicture}
			\caption{Defeatured domain $\Omega_0$ and simplified positive component $\tilde F_\mathrm p$.} \label{fig:boundingbox}
		\end{center}
	\end{subfigure}
	~
	\begin{subfigure}[t]{0.4\textwidth}
		\begin{center}
			\begin{tikzpicture}[scale=4.5]
			\draw[white] (2,1.3)--(2,1.35);
			\fill[c3, opacity=0.1] (2,0.5) rectangle (3,1);
			\fill[c3, opacity=0.55] (2.4,0.8) rectangle (2.7,1.00);
			\fill[c3, opacity=0.4] (2.3,1) -- (2.6,1) -- (2.6,1.1) -- (2.4,1.3) -- (2.3,1.3); 
			\fill[c3, opacity=0.8] (2.6,1.1) -- (2.6,1.3) -- (2.4,1.3);
			\draw[c2,thick] (2.6,1.1) -- (2.6,1.3) -- (2.4,1.3);
			\draw[c1] (2.5,1.14) node{$\gamma_\setminussign$};
			\draw (2.54,1.24)node{$G_\mathrm p$};
			\draw (2.55,0.9)node{$F_\mathrm n$};
			\draw (2.4,1.1)node{$F_\mathrm p$};
			\draw[c2,thick,densely dotted] (2.3,1) -- (2.4,1);
			\draw[thick] (2,0.5) -- (2,1) ;
			\draw[thick] (3,0.5) -- (3,1) ;
			\draw[thick] (2,0.5) -- (3,0.5);
			\draw[thick] (2,1) -- (2.3,1) ;
			\draw[thick] (2.7,1) -- (3,1) ;
			\draw[c1,thick,dashed] (2.4,1) -- (2.4,0.8) ;
			\draw[c1,thick,dashed] (2.7,1) -- (2.7,0.8) ;
			\draw[c1,thick,dashed] (2.4,0.8) -- (2.7,0.8) ;
			\draw[c1,thick,dashed] (2.6,1.1) -- (2.4,1.3);
			\draw[c4,thick] (2.4,1.0025) -- (2.6,1.0025) -- (2.6,1.1);
			\draw[c4,thick] (2.4,1.3) -- (2.3,1.3) -- (2.3,1);
			\draw[c4] (2.25,1.15) node{$\gamma_\intersign$};
			\draw (2.23,0.7) node{$\Omega_\star$} ;
			\draw[c2,thick] (2.4,0.9975) -- (2.6,0.9975);
			\draw[c2,thick] (2.6,0.9975) -- (2.7,1);
			\end{tikzpicture}
			\caption{Domains $\Omega_\star$, $F_\mathrm n$, $F_\mathrm p$ and $G_\mathrm p$.}
		\end{center}
	\end{subfigure}
	~
	\begin{subfigure}[t]{0.4\textwidth}
		\begin{center}
			\begin{tikzpicture}[scale=4.5]
			\fill[c3,opacity=0.25] (2.7,1) -- (3,1) -- (3,0.5) -- (2,0.5) -- (2,1) -- (2.3,1) -- (2.3,1.3) -- (2.6,1.3) -- (2.6,1) -- (2.4,1) -- (2.4,0.8) -- (2.7,0.8) --cycle;
			\draw[c3,thick] (2.6,1) -- (2.6,1.3) -- (2.3,1.3) -- (2.3,1.25);
			\draw[c2,densely dotted] (2.3,1) -- (2.4,1);
			\draw[thick] (2,0.5) -- (2,1) ;
			\draw[thick] (3,0.5) -- (3,1) ;
			\draw[thick] (2,0.5) -- (3,0.5);
			\draw[thick] (2,1) -- (2.3,1) ;
			\draw[thick] (2.7,1) -- (3,1) ;
			\draw[c1,dashed] (2.4,1) -- (2.4,0.8) ;
			\draw[c1,dashed] (2.7,1) -- (2.7,0.8) ;
			\draw[c1,dashed] (2.4,0.8) -- (2.7,0.8) ;
			\draw[c1,dashed] (2.6,1.1) -- (2.4,1.3);
			\draw[c4,thick] (2.4,1.0025) -- (2.6,1.0025) -- (2.6,1.1);
			\draw[c4,thick] (2.4,1.3) -- (2.3,1.3) -- (2.3,1);
			\draw[c2,thick] (2.4,0.9975) -- (2.6,0.9975);
			\draw[c2,thick] (2.6,0.9975) -- (2.7,1);
			\draw (2.23,0.7) node{$\tilde\Omega$} ;
			\draw[c2,thick] (2.6,1.1) -- (2.6,1.3) -- (2.4,1.3);
			\end{tikzpicture}
			\caption{Domain $\tilde \Omega$.}
		\end{center}
	\end{subfigure}
\caption{Illustration of the notation on a domain with a complex feature. \changes{The different domains are distinguished by different gray intensities.}} \label{fig:G0ppty1}
\end{figure}
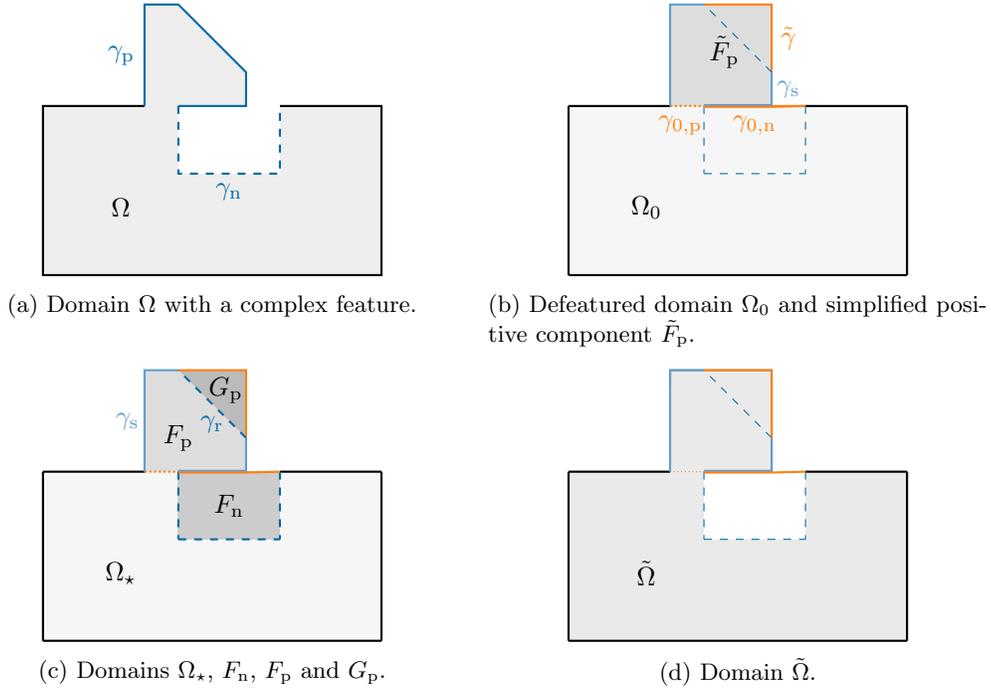

%
In this work, the analysis is performed under the following assumption: 
\begin{assumption}\label{as:neumann}
	A Neumann boundary condition is imposed on the boundary of the features, that is,
	$$\Gamma_D \cap \left(\partial F_\mathrm n \cup\partial F_\mathrm p\right)= \emptyset.$$
\end{assumption}
\begin{remark}
	\changes{Dealing with Dirichlet boundary conditions on the features' boundaries is a problem which goes beyond the scope of this paper, as difficulties arise from the necessity to work with negative Sobolev norms in the computation of a defeaturing error estimator. The interested reader is however referred for instance to \cite{repinsautersmolianski,carstensensauter2004,repinsauterbook} for some different insight on this subject.}
\end{remark}

To introduce the corresponding defeatured problem, let $\mathbf n_0$ and $\mathbf n_F$ be the unitary outward normal vectors to $\Omega_0$ and \changes{to either $F_\mathrm p$ or $F_\mathrm n$}, respectively. 
We remark that $\mathbf n_F$ may not be uniquely defined if the outward normal to $F_\mathrm n$ is of opposite sign of the outward normal to $F_\mathrm p$, but we allow this abuse of notation since the context will always make it clear. Furthermore, and as illustrated in Figure \ref{fig:G0ppty1}, let
\begin{align*}
\gamma_0 := \text{int}\left(\overline{\gamma_{0,\mathrm n}} \cup \overline{\gamma_{0,\mathrm p}}\right) \subset \partial \Omega_0 &\quad \text{with} \quad \gamma_{0,\mathrm n} := \partial F_\mathrm n \setminus \partial \Omega_\star, \quad \gamma_{0,\mathrm p}:= \partial F_\mathrm p \setminus \partial \Omega, \\
\gamma :=\text{int}\left(\overline{\gamma_{\mathrm n}} \cup \overline{\gamma_{\mathrm p}}\right) \subset \partial \Omega &\quad \text{with} \quad \gamma_\mathrm n:= \partial F_\mathrm n\setminus \overline{\gamma_{0,\mathrm n}}, \quad \gamma_\mathrm p:= \partial F_\mathrm p\setminus \overline{\gamma_{0,\mathrm p}}.
\end{align*} 
so that $\partial F_\mathrm n = \overline{\gamma_\mathrm n} \cup \overline{\gamma_{0,\mathrm n}}$ with $\gamma_\mathrm n\cap \gamma_{0,\mathrm n}=\emptyset$, and $\partial F_\mathrm p = \overline{\gamma_\mathrm p} \cup \overline{\gamma_{0,\mathrm p}}$ with $\gamma_\mathrm p\cap \gamma_{0,\mathrm p}=\emptyset$. 
Since from Assumption~\ref{as:neumann}, $\Gamma_D \cap \left(\partial F_\mathrm n \cup\partial F_\mathrm p\right)= \emptyset$, then note that $\gamma \subset \Gamma_N$. 

Moreover, consider any $L^2$-extension of the restriction $f\vert_{\Omega_\star}$ in the negative component $F_\mathrm n$ of $F$, that we still write $f\in L^2(\Omega_0)$ by abuse of notation. Then instead of (\ref{eq:originalpb}), we solve the following defeatured (or simplified) problem: after choosing $g_{0}\in H^{\frac{1}{2}}(\gamma_0)$, find the weak solution $u_0\in H^1(\Omega_0)$ of 
\begin{align} \label{eq:simplpb}
\begin{cases}
-\Delta u_0 = f &\text{ in } \Omega_0 \\
u_0 = g_D &\text{ on } \Gamma_D \\ \vspace{1mm}
\displaystyle\frac{\partial u_0}{\partial \mathbf{n}_0}  = g &\text{ on } \Gamma_N\setminus {\gamma}\\ 
\displaystyle\frac{\partial u_0}{\partial \mathbf{n}_0}  = g_{0} &\text{ on } \gamma_0,\vspace{0.1cm}
\end{cases}
\end{align}
that is, $u_0\in H^1_{g_D,\Gamma_D}(\Omega_0)$ satisfies for all $v\in H^1_{0,\Gamma_D}(\Omega_0)$, 
\begin{equation} \label{eq:weaksimplpb}
\int_{\Omega_0} \nabla u_0 \cdot \nabla v \,\mathrm dx = \int_{\Omega_0} fv \,\mathrm dx + \int_{\Gamma_N\setminus \gamma} g v \,\mathrm ds + \int_{\gamma_0} g_{0} v \,\mathrm ds.
\end{equation}

Because of the simplification of the domain $\Omega$ into $\Omega_0$ by defeaturing, a geometric error called defeaturing error is introduced, and it is important to be able to control this error. That is, in this work, we are interested in controlling the energy norm of the defeaturing error ``$u-u_0$" in $\Omega$.\\

However, since $u_0$ is not defined everywhere in $\Omega$ (it is not defined in $F_\mathrm p\subset \Omega$), then as in \cite{paper1defeaturing}, we need to solve an extension problem in a Lipschitz domain $\tilde{F}_\mathrm p\subset \mathbb{R}^n$ that contains $F_\mathrm p$ and such that $\gamma_{0,\mathrm p} \subset \left(\partial \tilde{F}_\mathrm p \cap \partial F_\mathrm p\right)$. That is, let $\tilde F_\mathrm p$ be a suitable (simple) domain extension of $F_\mathrm p$, which can be for instance the bounding box of $F_\mathrm p$ if the boundary of the latter contains $\gamma_{0, \mathrm p}$ (see Figure \ref{fig:boundingbox}). To simplify the following exposition, let us assume that $\tilde F_\mathrm p \cap \Omega_\star = \emptyset$, even if this hypothesis could easily be removed. Then, let 
\begin{equation} \label{eq:Gp}
G_\mathrm p := \tilde F_\mathrm p \setminus \overline{F_\mathrm p}, \qquad 
\tilde\Omega :=\text{int}\left(\overline{\Omega}\cup \overline{G_\mathrm p}\right) = \text{int}\left(\overline{\Omega_\star}\cup \overline{\tilde F_\mathrm p}\right),
\end{equation}
and let us also assume that $G_\mathrm p$ is Lipschitz. Note that $G_\mathrm p$ can be seen as a negative feature of $F_\mathrm p$ whose simplified geometry is $\tilde F_\mathrm p$. 

Now, let us define the extension problem that we solve in $\tilde F_\mathrm p$. To do so, consider any $L^2$-extension of $f$ in $\tilde F_\mathrm p$, that we still write $f$ by abuse of notation, let $\tilde{\mathbf{n}}$ be the unitary outward normal of $\tilde F_\mathrm p$, let $\tilde \gamma := \partial \tilde F_\mathrm p \setminus \partial F_\mathrm p$, and let $\gamma_\mathrm p$ be decomposed as $\gamma_\mathrm p = \text{int}(\overline{\gamma_\intersign}\cup\overline{\gamma_\setminussign})$, where $\gamma_\intersign$ and $\gamma_\setminussign$ are open, $\gamma_\intersign$ is the part of $\gamma_\mathrm p$ that is shared with $\partial \tilde F_\mathrm p$ while $\gamma_\setminussign$ is the remaining part of $\gamma_\mathrm p$, that is, the part that does not belong to $\partial \tilde F_\mathrm p$. 
Then we solve the following extension problem in $\tilde F_\mathrm p$: after choosing $\tilde g\in H^{\frac{1}{2}}(\tilde \gamma)$, find $\tilde{u}_0\in H^1\left(\tilde{F}_\mathrm p\right)$, the weak solution of 
\begin{align} \label{eq:featurepb}
\begin{cases}
-\Delta \tilde{u}_0 = f &\text{ in } \tilde{F}_\mathrm p \\ \vspace{1mm}
\tilde{u}_0 = u_0 & \text{ on } \gamma_{0,\mathrm p} \\ \vspace{1mm}
\displaystyle\frac{\partial \tilde u_0}{\partial \tilde{\mathbf{n}}}  = \tilde g & \text{ on } \tilde \gamma \\
\displaystyle\frac{\partial \tilde u_0}{\partial \tilde{\mathbf{n}}}  = g & \text{ on } \gamma_\intersign.
\end{cases}
\end{align}
That is, $\tilde{u}_0\in H^1_{u_0,\gamma_{0,\mathrm p}}\left(\tilde F_\mathrm p\right)$ satisfies for all $v\in H^1_{0,\gamma_{0,\mathrm p}}\left(\tilde{F}_\mathrm p\right)$, 
\begin{equation} \label{eq:weakfeaturepb}
\int_{\tilde{F}_\mathrm p} \nabla \tilde{u}_0 \cdot \nabla v \,\mathrm dx = \int_{\tilde{F}_\mathrm p} fv \,\mathrm dx + \int_{\tilde \gamma} \tilde{g} v \,\mathrm ds + \int_{\gamma_\intersign} {g} v \,\mathrm ds.
\end{equation}

Finally, let $u_\mathrm d\in H^1_{g_D,\Gamma_D}\left(\Omega\right)$ be the extended defeatured solution, that is,
\begin{equation}\label{eq:defud}
u_\mathrm d = u_0\vert_{\Omega_\star} \text{ in } \Omega_\star=\Omega\setminus\overline{F_\mathrm p} \quad \text{ and } \quad u_\mathrm d = \tilde{u}_0\vert_{F_\mathrm p} \text{ in } F_\mathrm p.
\end{equation}
Now, both the exact solution $u$ and the defeatured solution $u_\mathrm d$ are defined in the same exact domain $\Omega$, and thus we can precisely define the defeaturing error as $\left\|\nabla\left(u-u_\mathrm d\right)\right\|_{0,\Omega}$. \\

To ease the notation in the remaining part of the article, let
\begin{equation} \label{eq:neumannbd}
\Gamma_N^0 := \left(\Gamma_N\setminus\gamma\right)\cup\gamma_0 \quad \text{ and } \quad \tilde \Gamma_N := \gamma_\intersign \cup \tilde \gamma
\end{equation}
be the Neumann boundaries of $\Omega_0$ and of $\tilde F_\mathrm p$, respectively, and let 
\begin{equation} \label{eq:defSigma}
\Sigma := \left\{\gamma_\mathrm n, \gamma_{0,\mathrm p}, \gamma_\setminussign\right\}.
\end{equation}

\subsection{Finite element formulation of the defeatured problem} \label{ss:fepbstatement}
In practical applications, the differential problems are not solved analytically but numerically. In this work, we employ the Galerkin method to discretize the defeatured problem. More precisely, let $\mathcal Q_0$ be a finite element mesh defined on $\Omega_0$, and let $V^h(\Omega_0)\subset H^1(\Omega_0)$ be a FE space built from $\mathcal Q_0$. 
We suppose for simplicity that the Dirichlet boundary $\Gamma_D$ is the union of full element edges (if $n=2$) or faces (if $n=3$), that the Dirichlet data $g_D$ is the trace of a discrete function in $V^h(\Omega)$ which we will still write $g_D$ by abuse of notation, and that $g_D \in \text{tr}_{\Gamma_D}\left(V^h(\Omega_0)\right)$, where $\text{tr}_{\Gamma_D}\left(V^h(\Omega_0)\right)$ is the trace space of the discrete functions of $V^h(\Omega_0)$ on $\Gamma_D$. 

Then, let $\tilde{\mathcal Q}$ be a FE mesh defined on $\tilde F_\mathrm p$, and let $V^h\big(\tilde F_\mathrm p\big)\subset H^1\big(\tilde F_\mathrm p\big)$ be a FE space built from $\tilde{\mathcal Q}$ that satisfies the following compatibility assumption. 
\begin{assumption} \label{as:discrspacescompat}
	Using the notation introduced in Table \ref{tbl:notation}, $V^h(\Omega_0)$ and $V^h\big(\tilde F_\mathrm p\big)$ satisfy
	\begin{equation*}
		\text{tr}_{\gamma_{0,\mathrm p}}\big(V^h\left(\Omega_0\right)\big) \subset \text{tr}_{\gamma_{0,\mathrm p}}\Big(V^h\big(\tilde F_\mathrm p\big)\Big).
	\end{equation*}
	\changes{That is, the trace of the discrete solution, computed in $\Omega_0$, can be exactly represented in the discrete space in the feature, and therefore the solutions in $\Omega_0$ and $\tilde{F}_p$ can be glued conformingly in $H^1$.}
\end{assumption}

We are now able to define $u_\mathrm d^h$ as the discretized counter-part of $u_\mathrm d$ defined in~(\ref{eq:defud}). More precisely, it is the continuous function corresponding to the Galerkin approximation of problem~(\ref{eq:weaksimplpb}) using the discrete space $V^h(\Omega_0)$, and of problems~(\ref{eq:weakfeaturepb}) using the discrete spaces $V^h\big(\tilde F_\mathrm p^k\big)$ for all $k=1,\ldots,N_f$. 
Then, we define the discrete defeaturing error (or overall error) as $\left\|\nabla\left(u-u_\mathrm d^h\right)\right\|_{0,\Omega}$. 
In the following Section \ref{sec:adaptive}, we introduce an \textit{a posteriori} estimator of this error. 

\begin{remark}
	The only required compatibility between $V^h(\Omega_0)$ and $V^h\big(\tilde F_\mathrm p\big)$ is Assumption \ref{as:discrspacescompat}. \changes{This assumption gives the guarantee that tr$_{\gamma_{0,\mathrm p}}\left(u_0^h\right) \in \text{tr}_{\gamma_{0,\mathrm p}}\Big(V^h\big(\tilde F_\mathrm p\big)\Big)$, where $u_0^h$ is the discretized counter-part of $u_0$.} Moreover, note that $\mathcal Q_0$ and $\tilde{\mathcal Q}$ are possibly overlapping meshes, but both should have $\gamma_{0,\mathrm p}$ as part of their boundary. 
\end{remark}

\begin{remark}
	Note that the aim is never to solve the original problem (\ref{eq:originalpb}) in the exact geometry $\Omega$. Indeed, we assume that one needs to remove the features of $\Omega$ since solving a PDE in $\Omega$ is either too costly or even unfeasible (for instance, it could be impossible to mesh $\Omega$). Therefore in principle, the original problem (\ref{eq:originalpb}) is never solved in the discrete setting. 
\end{remark}

\subsection{Generalization to multiple features} \label{ss:multifeature}
In this section, we generalize the previous setting to an exact geometry $\Omega$ containing $N_f\geq1$ distinct features, following \cite{paper3multifeature}. Note that the notation introduced here is compatible and generalizing the one used in the previous sections. So let the exact geometry $\Omega\subset\mathbb{R}^n$ be an open Lipschitz domain with $N_f\in \mathbb{N}$ distinct complex Lipschitz geometrical features $\mathcal F:= \left\{F^k\right\}_{k=1}^{N_f}$. That is, for all $k=1,\ldots,N_f$, $F^k$ is an open domain which is composed of a (not necessarily connected) negative component $F_\mathrm n^k$ and a (not necessarily connected) positive component $F^k_\mathrm p$ that can have a non-empty intersection. More precisely, $F^k = \text{int}\left(\overline{F_\mathrm p^k} \cup \overline{F_\mathrm n^k}\right)$, where $F_\mathrm n^k$ and $F_\mathrm p^k$ are open Lipschitz domains such that if we let 
$$F_\mathrm p := \text{int}\left(\bigcup_{k=1}^{N_f} \overline{F_\mathrm p^k}\right), \quad F_\mathrm n := \text{int}\left(\bigcup_{k=1}^{N_f} \overline{F_\mathrm n^k}\right), \quad \Omega_\star := \Omega \setminus \overline{F_\mathrm p}, $$
then $F_\mathrm p \subset \Omega, \quad \left(\overline{F_\mathrm n} \cap \overline{\Omega_\star} \right) \subset \partial \Omega_\star.$
In this setting, the defeatured geometry is defined as in (\ref{eq:defomega0}) by 
\begin{equation} \label{eq:defomega0multi}
\Omega_0 := \text{int}\left( \overline{\Omega_\star} \cup \overline{F_\mathrm n} \right) \subset \mathbb{R}^n,
\end{equation}
and we assume that $\Omega_0$ is an open Lipschitz domain. 
Let us make the following separability assumption on the features, already used in \cite{paper3multifeature}:
\begin{assumption} \label{def:separated}
	The features $\mathcal{F}$ are separated, that is, 
	\begin{itemize}
		\item for every $k,\ell=1,\ldots,N_f$, $k\neq\ell$, $\overline{F^k}\cap \overline{F^\ell} = \emptyset$, 
		\item there exist sub-domains $\Omega^k\subset \Omega$, $k=1,\ldots,N_f$ such that \begin{itemize}[$\diamond$]
			\item $F_\mathrm p^k \subset \Omega^k,\quad \left(\gamma_\mathrm n^k \cup \gamma_\setminussign^k\right) \subset \partial \Omega^k, \quad \gamma_{0,\mathrm p}^k \subset \partial (\Omega^k \cap \Omega_0)$, 
			\item $\left|\Omega^k\right| \simeq |\Omega|$, i.e., the measure of $\Omega^k$ is comparable with the measure of $\Omega$, not with the measure of the feature $F^k$, 
			\item \changes{let $N_s$ be the maximum number of superposed sub-domains $\Omega^k$, i.e., $$N_s := \displaystyle\max_{J\subset \left\{1,\ldots,N_f\right\}} \left(\#J : \bigcap_{k\in J}\Omega^k \neq \emptyset \right).$$ Then $N_s$ is bounded independently from the number of features $N_f$, and if $N_f$ is large, $N_s \ll N_f$.}
		\end{itemize}
	\end{itemize}
\end{assumption}
\begin{remark}\begin{itemize}
	\item It is always possible to satisfy the first condition of Assumption \ref{def:separated} by changing the numbering of the features. Indeed, if there are $k,\ell = 1,\ldots,N_f$ such that $\overline{F^k}\cap\overline{F^\ell}\neq \emptyset$, then \hbox{$F^{k,\ell}:= \textrm{int}\left( \overline{F^k} \cup \overline{F^\ell}\right)$} can be considered as a single feature that replaces the two features \hbox{$F^k$ and $F^\ell$}. 
	\item The second condition of Assumption \ref{def:separated} means that one cannot have an increasingly large number of features that are arbitrarily close to one another. \changes{That is, one cannot have fractal-like structures or geometries with boundaries which are complicated everywhere. The interested reader is referred to \cite{hiptmairshapeapprox,HEYDAROV2022102157,shapederpaper} for first results in this different framework.}
	\item If $N_f = 1$, one can take $\Omega^1 := \Omega$. 
\end{itemize}
\end{remark}


Then, let us use the same notations as in Sections \ref{ss:continuouspbstatement} and \ref{ss:fepbstatement}, where an upper index $k$ is added to a quantity referring to feature $F^k$. For instance, for all $k=1,\ldots,N_f$, we write 
\vspace{-0.1cm}\begin{align}
&\gamma_\mathrm n^k := \partial F_\mathrm n^k \setminus \partial \Omega_\star, \qquad \gamma_\mathrm n := \bigcup_{k=1}^{N_f} \gamma_\mathrm n^k, \qquad \tilde \Gamma_N^k := \gamma_\intersign^k \cup \tilde \gamma^k, \qquad\tilde \Gamma_N := \bigcup_{k=1}^{N_f} \tilde \Gamma_N^k, \label{eq:GammaNktilde}
\end{align}
and similarly for $\gamma_{0,\mathrm p}$ and $\gamma_\setminussign$. In particular, we generalize the definition of $\Sigma$ from (\ref{eq:defSigma}) as follows:
\begin{align}
\Sigma^k &:= \left\{\gamma_\mathrm n^k, \gamma_{0,\mathrm p}^k, \gamma_\setminussign^k\right\}, \quad \text{ for } k=1,\ldots,N_f, \nonumber \\
\Sigma &:= \{\sigma \in \Sigma^k : k=1,\ldots,N_f\}. \label{eq:sigmamulti}
\end{align}
Let $\tilde{\mathbf{n}}^k$ be the unitary outward normal of $\tilde F_\mathrm p^k$, and to ease the notation, let us write $\mathbf n^k := \mathbf n_{F^k}$ the unitary outward normal of $F^k$, for all $k=1,\ldots,N_f$. Let us also write $\tilde F_\mathrm p^k$ the domain extension of the positive component $F_\mathrm p^k$ of feature $F^k$ and 
\vspace{-0.2cm}\begin{equation} \label{eq:defGpk}
G_\mathrm p^k := \tilde F_\mathrm p^k \setminus \overline{F_\mathrm p^k}.
\end{equation}
As in the single feature case, note that $G_\mathrm p^k$ can be seen as a negative feature of $F_\mathrm p^k$ whose simplified geometry is $\tilde F_\mathrm p^k$. 
To simplify the following exposition and even if this hypothesis could easily be removed, let us make the following assumption. 
\begin{assumption}
	$\tilde F_\mathrm p^k \cap \tilde F_\mathrm p^\ell = \emptyset$ for all $k,\ell = 1,\ldots,N_f$ such that $k\neq \ell$, and $\tilde F_\mathrm p\cap \Omega_\star = \emptyset$ with $\tilde F_\mathrm p := \displaystyle\bigcup_{k=1}^{N_f} \tilde F_\mathrm p^k$. 
\end{assumption}
\noindent Then define 
$\tilde\Omega :=\text{int}\left(\overline{\Omega_\star}\cup \overline{\tilde F_\mathrm p}\right)$. 
Moreover, let $\tilde{\mathcal Q}^k$ be a FE mesh defined on $\tilde F_\mathrm p^k$ for all $k=1,\ldots,N_f$, and let
\begin{align}
\vspace{-0.2cm}\mathcal{Q} := \mathcal Q_0 \cup \tilde{\mathcal Q}\quad \text{ with }\quad \tilde{\mathcal Q}:=\displaystyle\bigcup_{k=1}^{N_f} \tilde{\mathcal Q}^k. \label{eq:defQhmulti}
\end{align}
In this multi-feature setting, the FE spaces $V^h(\Omega_0)$ and $V^h\big(\tilde F_\mathrm p^k\big)$ respectively built from $\mathcal Q_0$ and $\tilde{\mathcal Q}^k$, need to satisfy Assumption \ref{as:discrspacescompat} for each $k=1,\ldots,N_f$. 

Moreover, we first solve the Galerkin formulation of the defeatured problem~(\ref{eq:weaksimplpb}) in $\Omega_0$ to obtain $u_0^h$. Then, for each feature $F^k$ whose positive component is non-empty, we solve the Galerkin formulation of the extension problem~(\ref{eq:weakfeaturepb}) in $\tilde F_\mathrm p^k$ to obtain $\tilde u_0^{k,h} \in V^h\big(\tilde F_\mathrm p^k\big)\cap H_{u_0^h,\gamma_{0,\mathrm p}^k}^1\big(\tilde F_\mathrm p^k\big).$
To reduce the notation in the multi-feature context, we will write
\begin{equation} \label{eq:notationukh}
	u_k^h \equiv \tilde u_0^{k,h}.
\end{equation}
Note that these extensions $u_k^h$ can be computed separately (and in parallel) for each feature $F^k$.
Finally, the discrete extended defeatured solution $u_\mathrm d^h$ is defined as:
\begin{equation}\label{eq:defudhmultifeat}
u_\mathrm d^h = u_0^h\big\vert_{\Omega_\star} \text{ in } \Omega_\star=\Omega\setminus\overline{F_\mathrm p} \quad \text{ and } \quad u_\mathrm d^h = {u}_k^h\big\vert_{F_\mathrm p^k} \text{ in } F_\mathrm p^k \,\,\text{ for } k=1,\ldots,N_f,
\end{equation}
and the discrete defeaturing error (or overall error) is defined as $\left\|\nabla\left(u-u_\mathrm d^h\right)\right\|_{0,\Omega}$. 

%% file: adaptivedefeat.tex
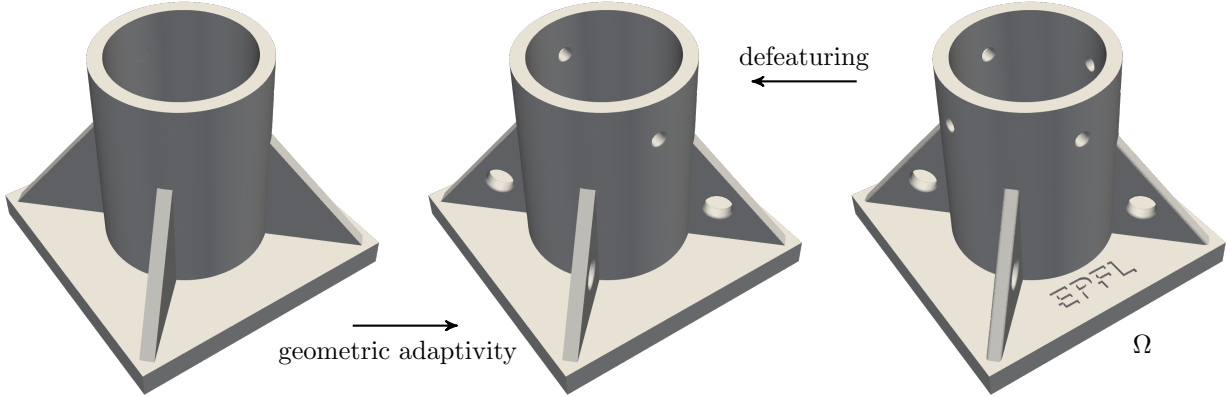
\begin{figure}
	\begin{center}
		\begin{tikzpicture}
		\draw node{\includegraphics[scale=0.135,trim=450 0 400 20, clip]{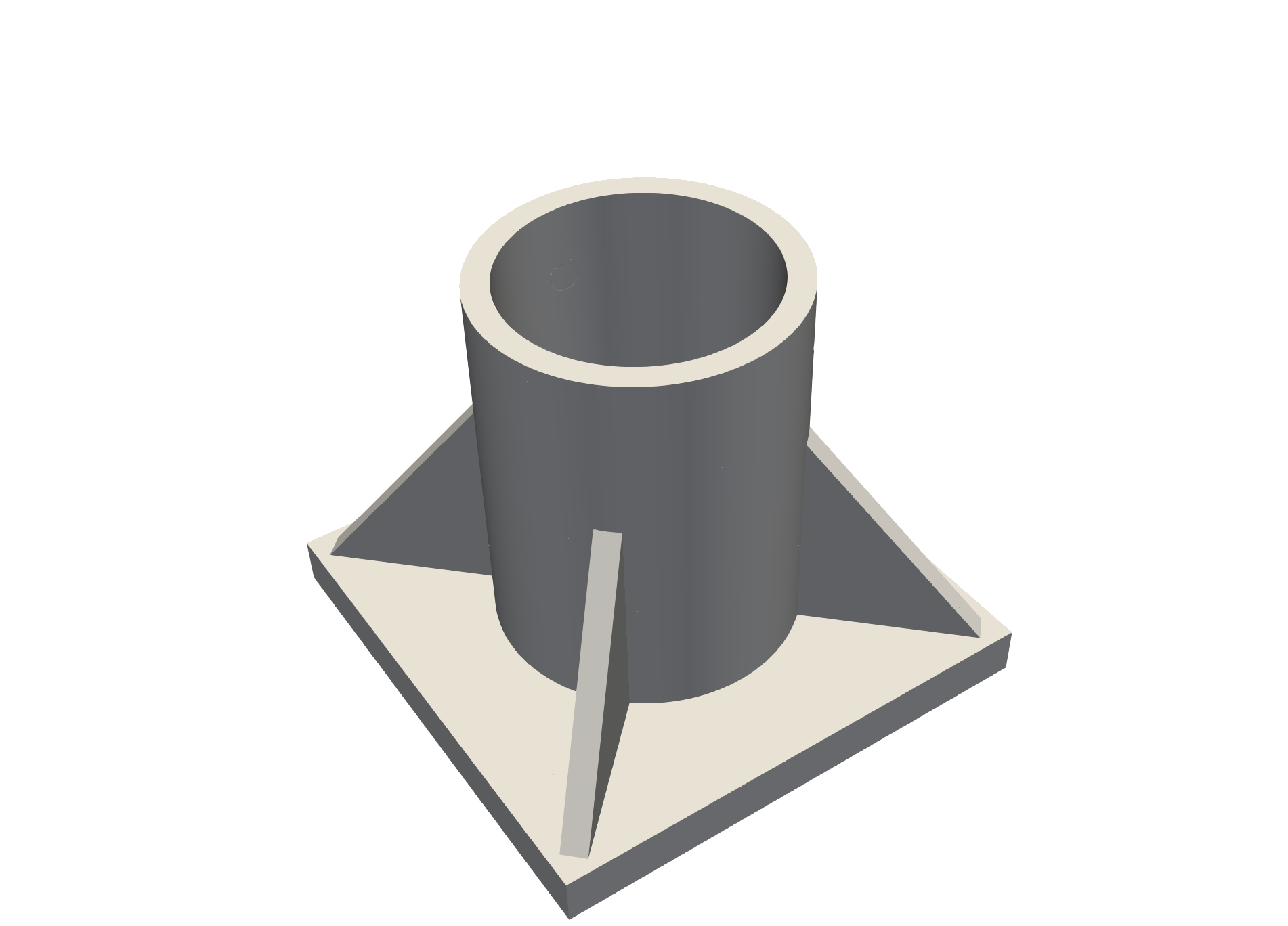}\hspace{5mm} 
			\includegraphics[scale=0.135,trim=450 0 400 20, clip]{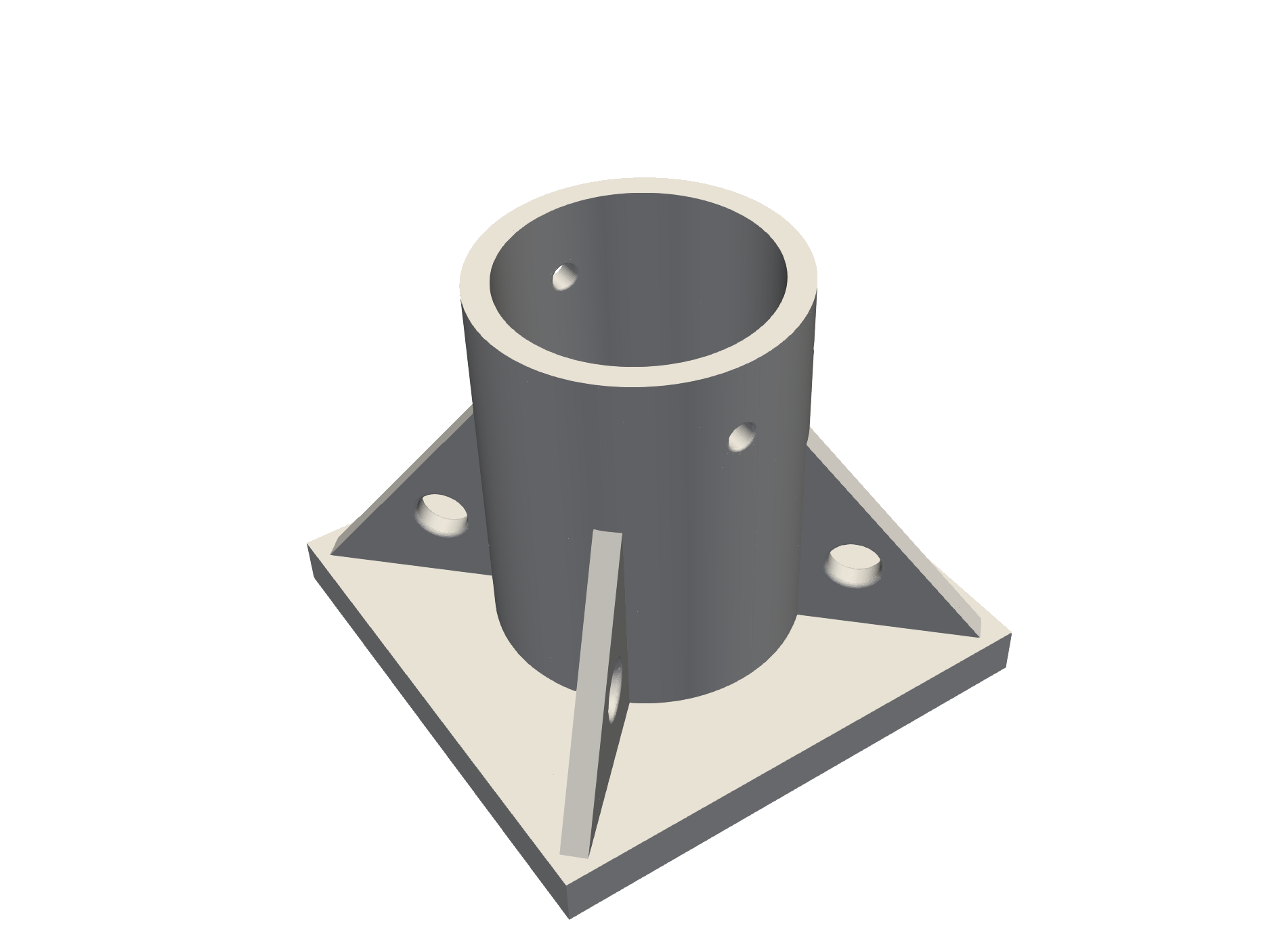}\hspace{5mm} 
			\includegraphics[scale=0.135,trim=450 0 400 20, clip]{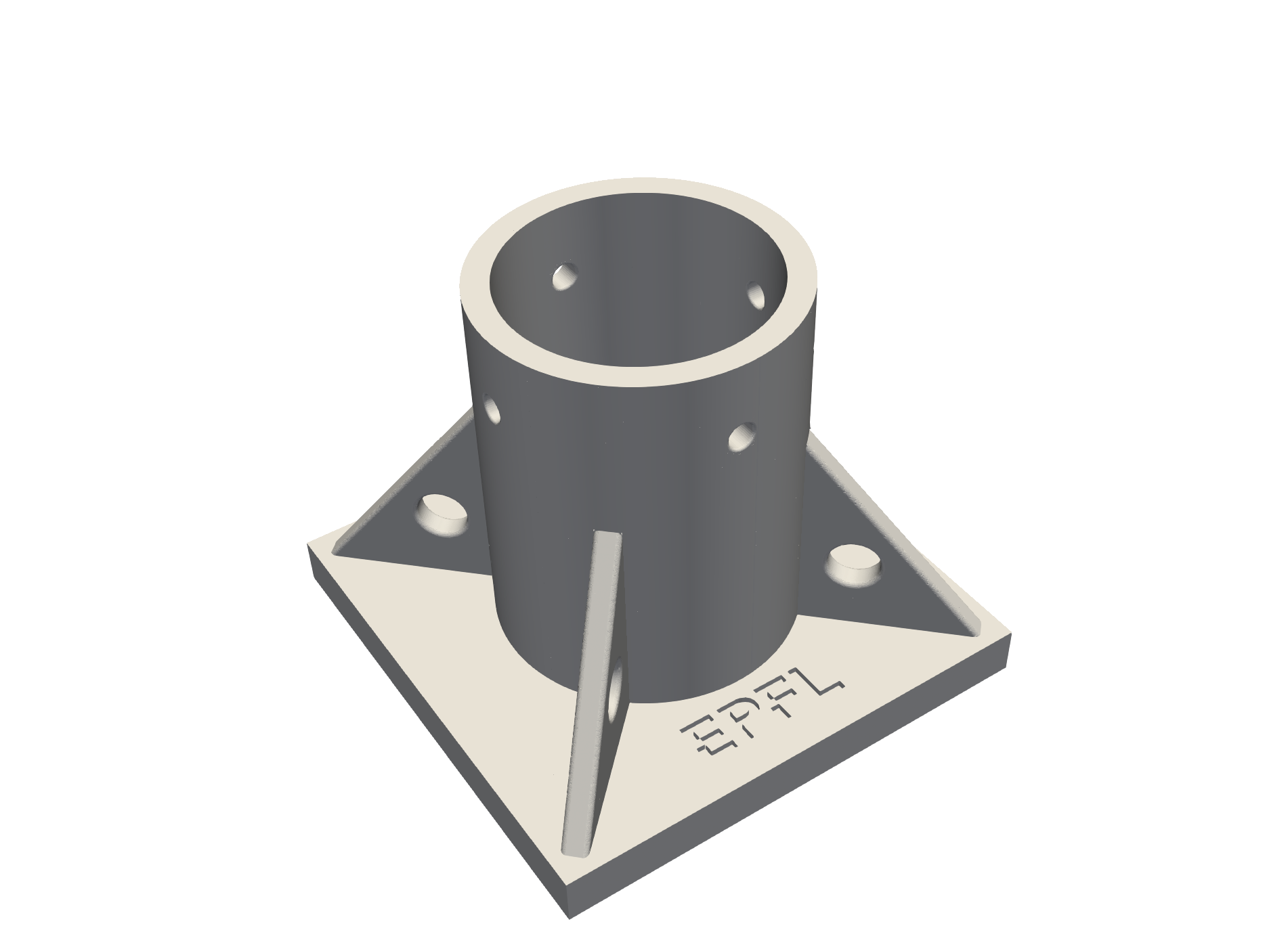}}; 
		\draw[->,
		>=stealth',
		auto,thick] (-3.5,-2.25) -- (-2.1,-2.25) ;
		\draw (-2.9,-2.6) node{{geometric} adaptivity} ;
		\draw[->,
		>=stealth',
		auto,thick] (3.2,1) -- (1.8,1) ;
		\draw (2.5,1.3) node{defeaturing} ;
		\draw (7,-2.5) node{$\Omega$};
		\end{tikzpicture}
		\caption{Illustration of defeaturing and {geometric} adaptivity.} \label{fig:defeatadaptive}
	\end{center}
\end{figure}

In this section, we aim at defining an adaptive analysis-aware defeaturing method for the Poisson problem defined in Section \ref{sec:pbstatement}, in a geometry $\Omega$ containing $N_f\geq 1$ distinct complex features. In particular, starting from a very coarse mesh $\mathcal Q$ and from a fully defeatured geometry $\Omega_0$, we want to precisely define a strategy that determines:
\begin{itemize}
	\item when and where the mesh $\mathcal Q$ needs to be refined (standard $h$-refinement, or numerical adaptivity), 
	\item when and which geometrical features that have been removed by defeaturing need to be reinserted in the geometrical model (geometric adaptivity, see also \cite{paper3multifeature}).
\end{itemize}
Note that the word \textit{defeaturing} may be misleading when thinking of an adaptive strategy: the geometry $\Omega_0$ on which the problem is actually solved is (partially) defeatured, but the adaptive algorithm chooses which features need to be \textit{added} to be able to solve the differential problem up to a given accuracy. This is illustrated in Figure \ref{fig:defeatadaptive}. \\

{Closely following the framework of adaptive finite elements \cite{nochettoprimer} for elliptic PDEs, we recall the four main building blocks of adaptivity, composing one iteration of the iterative process:
	\FloatBarrier
	\begin{figure}[h!]
		\begin{center}
			\begin{tikzpicture}[
			start chain = going right,
			node distance=7mm,
			block/.style={shape=rectangle, draw,
				inner sep=1mm, align=center,
				minimum height=5mm, minimum width=15mm, on chain}] 
			\node[block] (n1) {SOLVE};
			\node[block] (n2) {ESTIMATE};
			\node[block] (n3) {MARK};
			\node[block] (n4) {REFINE};
			
			\draw[->] (n1.east) --  + (0,0mm) -> (n2.west);
			\draw[->] (n2.east) --  + (0,0mm) -> (n3.west);
			\draw[->] (n3.east) --  + (0,0mm) -> (n4.west);
			\draw[<-] (n1.south) --  + (0,-5mm) -| (n4.south);
			\end{tikzpicture}
		\end{center}
	\end{figure}
	\FloatBarrier}
\noindent In the following, we elaborate on each of these blocks in the context of analysis-aware defeaturing, letting $i\in\mathbb N$ be the current iteration index. 
In particular, in Section \ref{sec:estimate}, we propose an \textit{a posteriori} estimator of the discrete defeaturing error $\left\|\nabla\left(u-u_\mathrm d^h\right)\right\|_{0,\Omega}$. Its reliability will be proven in Section \ref{sec:combinedest} in the context of IGA with THB-splines under reasonable assumptions. 

To begin the adaptive process, let $\Omega_0^{(0)}$ be the fully defeatured geometry defined as in (\ref{eq:defomega0multi}), i.e., the domain in which all features of $\Omega$ are removed. Since some features will be reinserted during the adaptive process, we denote $\Omega_0^{(i)}$ the simplified geometry at the $i$-th iteration, and in general, we use the super index $(i)$ to refer to objects at the same iteration. However, to alleviate the notation, we will drop the super index when it is clear from the context. In particular, we will often write $\Omega_0\equiv \Omega_0^{(i)}$. 

\subsection{Solve} \label{sec:solve}
Using suitable FE spaces, we first solve the Galerkin formulation of problem~(\ref{eq:weaksimplpb}) defined in the (partially) defeatured geometry $\Omega_0$. Then, we solve the Galerkin formulation of the local extension problem~(\ref{eq:weakfeaturepb}) for each feature having a non-empty positive component. We thus obtain the discrete defeatured solution $u_\mathrm d^h\equiv u_\mathrm d^{h,(i)}$ defined in (\ref{eq:defudhmultifeat}), as an approximation of the exact solution $u$ of (\ref{eq:originalpb}) at iteration $i$.

\subsection{Estimate} \label{sec:estimate}
In order to define the proposed \textit{a posteriori} estimator $\mathscr{E}\big(u_\mathrm d^h\big)$ of the discrete defeaturing error $\left\|\nabla\left(u-u_\mathrm d^h\right)\right\|_{0,\Omega}$, let us first introduce some further notation. In the following and by abuse of notation, we use $u_\mathrm d^h$ instead of $u_0^h$ or $u_k^h$ for some $k=1,\ldots,N_f$ where $N_f \equiv N_f^{(i)}$, whenever the context is clear. For each $\sigma\in\Sigma$, let $d_\sigma$ be the continuous defeaturing error term, and let $d_\sigma^h$ be its discrete counterpart. That is, for all $k=1,\ldots,N_f$ and all $\sigma \in \Sigma^k$, 
\begin{equation} \label{eq:defdsigmah}
d_\sigma := \begin{cases}\vspace{1mm}
g + \displaystyle\frac{\partial u_\mathrm d}{\partial \mathbf n^k} & \text{if } \sigma = \gamma_\mathrm n^k \\\vspace{1mm}
g - \displaystyle\frac{\partial u_\mathrm d}{\partial \mathbf n^k} & \text{if } \sigma = \gamma_\setminussign^k \\
-g_0 - \displaystyle\frac{\partial u_\mathrm d}{\partial \mathbf n^k} & \text{if } \sigma = \gamma_{0,\mathrm p}^k.
\end{cases}
\text{ and } \quad d_\sigma^h := \begin{cases}\vspace{1mm}
g + \displaystyle\frac{\partial u^h_\mathrm d}{\partial \mathbf n^k} & \text{if } \sigma = \gamma_\mathrm n^k \\\vspace{1mm}
g - \displaystyle\frac{\partial u^h_\mathrm d}{\partial \mathbf n^k} & \text{if } \sigma = \gamma_\setminussign^k \\
-g_0 - \displaystyle\frac{\partial u^h_\mathrm d}{\partial \mathbf n^k} & \text{if } \sigma = \gamma_{0,\mathrm p}^k.
\end{cases}
\end{equation}

\begin{remark}
	\changes{In the problem statement of Section~\ref{ss:continuouspbstatement}, we have considered $g_D\in H^{\frac{3}{2}}(\Gamma_D)$, $g\in H^{\frac{1}{2}}(\Gamma_N)$ and $f\in L^2\left(\Omega\right)$, which is a higher regularity than what is required for Poisson's problem to be well defined. This is justified by the fact that we need the quantities in~(\ref{eq:defdsigmah}) to be in $L^2(\sigma)$ to define the defeaturing error estimator from \cite{paper1defeaturing}.}
\end{remark}
\begin{remark}\label{rmk:compatcond}
	Note that for all $\sigma \in \Sigma$, the average value of $d_\sigma$ over $\sigma$ is a computable quantity as it is independent of (the unknown function) $u_\mathrm d$. Indeed, using the single feature notation for simplicity, recalling the definition of $G_\mathrm p$ from (\ref{eq:Gp}) and as already observed in \cite[Remark~5.3]{paper1defeaturing}, 
	\begin{align}
	\overline{d_{\gamma_\mathrm n}}^{\gamma_\mathrm n} &= \frac{1}{\left|\gamma_\mathrm n\right|} \left( \int_{\gamma_\mathrm n} g\,\mathrm ds - \int_{\gamma_{0,\mathrm n}} g_0\,\mathrm ds - \int_{F_\mathrm n} f\,\mathrm dx\right), \nonumber \nonumber \\
	\overline{d_{\gamma_\setminussign}}^{\gamma_\setminussign} &= \frac{1}{\left|\gamma_\setminussign\right|} \left( \int_{\gamma_\setminussign} g\,\mathrm ds - \int_{\tilde \gamma} \tilde g\,\mathrm ds - \int_{G_\mathrm p} f\,\mathrm dx\right), \nonumber \\ 
	\overline{d_{\gamma_{0,\mathrm p}}}^{\gamma_{0,\mathrm p}} &= \frac{1}{\left|\gamma_{0,\mathrm p}\right|} \left( \int_{\gamma_{0,\mathrm p}} g_0\,\mathrm ds - \int_{\gamma_\mathrm p} g\,\mathrm ds - \int_{F_\mathrm p} f\,\mathrm dx \right).\label{eq:compatcond}
	\end{align}
	That is, for all $\sigma\in\Sigma$, $\overline{d_\sigma}^\sigma$ accounts for and only depends on the choice of defeaturing data, namely, the Neumann boundary data $g_0$ on $\gamma_0$ and $\tilde g$ on $\tilde \gamma$, and the right hand side extension $f$ in $F_\mathrm n$ and in $G_\mathrm p$. In other words, these quantities evaluate the (lack of) conservation of the continuous solution flux going through $F_\mathrm n$, $G_\mathrm p$ and $F_\mathrm p$. \changes{Note that one cannot always choose the data so that these quantities equal to zero: $f$ could be the gravity force for instance so that in practice, one usually naturally extends the data of the exact problem to the defeatured problem. This is also done to avoid some local meshing of the defeatured geometry to implement specific right hand side or boundary conditions.}
\end{remark}

If we define $\eta\in\mathbb{R}$, $\eta>0$ as the unique solution of $\eta = -\log(\eta)$, then for all $\sigma\in\Sigma$, let 
\begin{align} \label{eq:cgamma}
c_{\sigma} := \begin{cases}
\max\big(\hspace{-0.05cm}-\log\left(|\sigma|\right), \eta \big)^\frac{1}{2} & \text{ if } n = 2 \\
1 & \text{ if } n = 3.
\end{cases}
\end{align}
We can now define the total error estimator as follows:
\begin{align}
\mathscr{E}\big(u_\mathrm d^h\big) &= \left[\alpha_D^2\mathscr{E}_D\big(u_\mathrm d^h\big)^2 + \alpha_N^2\mathscr{E}_N\big(u_\mathrm d^h\big)^2 \right]^\frac{1}{2}, \label{eq:overallestimator}
\end{align}
where $\alpha_D>0$ and $\alpha_N>0$ are parameters to be tuned, 
\begin{equation*}
\mathscr{E}_D\big(u_\mathrm d^h\big)^2 :=\sum_{\sigma \in \Sigma} |\sigma|^{\frac{1}{n-1}}\left\| d_\sigma^h - \overline{d_\sigma^h}^\sigma \right\|^2_{0,\sigma} + \mathscr{E}_C^{\,\,2}
\end{equation*}
accounts for the defeaturing error as in \cite{paper1defeaturing}, with
\begin{equation*}
\mathscr{E}_C^{\,\,2} :=\sum_{\sigma \in \Sigma} c_\sigma^2 |\sigma|^{\frac{n}{n-1}}\left| \overline{d_\sigma}^\sigma \right|^2
\end{equation*}
that accounts for the data compatibility conditions (see Remark \ref{rmk:compatcond}), and $\mathscr{E}_N\big(u_\mathrm d^h\big)$ accounts for the numerical error that depends on the chosen FE method. More precisely,
$$\mathscr{E}_N\big(u_\mathrm d^h\big)^2 := \mathscr{E}_{N}^0\big(u_0^h\big)^2 + \sum_{k=1}^{N_f} \mathscr{E}_N^{k}\big(u_k^h\big)^2,$$
where $\mathscr{E}_N^{0}\big(u_0^h\big)$ is the estimator for the numerical error of $u_0^h$ in $\Omega_0$ corresponding to the discretized problem (\ref{eq:weaksimplpb}), and for all $k=1,\ldots,N_f$, $ \mathscr{E}_N^{k}\big(u_k^h\big)$ is the estimator for the numerical error of $u_k^h$ in $\tilde F_\mathrm p^k$ corresponding to the discretized problem (\ref{eq:weakfeaturepb}). 

\begin{remark} \label{rmk:nodecouple}
	\begin{itemize}
		\item The numerical error contribution $\mathscr E_N\big(u_0^h\big)$ to the estimator is computed on the {same} meshes used for the discretization, in $\Omega_0$ and in $\tilde F_\mathrm p$. Instead, the defeaturing error contribution $\mathscr E_D\big(u_0^h\big)$ to the estimator considers boundary integrals which are not necessarily union of edges or faces of the mesh, i.e., the pieces of boundaries $\sigma\in\Sigma$ are {in general not} fitted by the mesh.
		\item If $u_\mathrm d = u$ or in other words, without defeaturing error, the proposed estimator corresponds to the standard numerical residual error estimator between $u_0$ and $u_0^h$ in $\Omega_0$, and between $u_k$ and $u_k^h$ in $\tilde F_\mathrm p^k$ for all $k=1,\ldots,N_f$, that is,
		$\mathscr{E}\big(u_\mathrm d^h\big) = \mathscr{E}_N\big(u_\mathrm d^h\big)$. In this case, the proposed adaptive strategy coincides with what can be found in the standard adaptive FE literature, see e.g. \cite{nochetto2009theory}. 
		\item If $u_0^h \equiv u_0$ in $\Omega_0$ and $u_k\equiv u_k^h$ in $\tilde F_\mathrm p^k$ for all $k=1,\ldots,N_f$, i.e., without numerical error, then \hbox{$\mathscr{E}_N\big(u_\mathrm d^h\big)=0$}, and thus we recover the defeaturing error estimator introduced in \cite{paper1defeaturing} and generalized in \cite{paper3multifeature} to \hbox{geometries} with multiple features.
		\item Under the exact data compatibility condition $\overline{d_\sigma}^\sigma = 0$ for all $\sigma \in \Sigma$ (see (\ref{eq:compatcond})), the term $\mathscr{E}_C$ vanishes. If $\mathscr E_C$ is large with respect to the other terms of $\mathscr{E}\big(u_\mathrm d^h\big)$, the defeaturing data should be chosen more carefully, i.e., the Neumann boundary data $g_0$ on $\gamma_0$, $\tilde g$ on $\tilde \gamma$, and the right hand side extension $f$ in $F_\mathrm n$ and in each $G_\mathrm p^k$, $k=1,\ldots,N_f$.
		\item \changes{The exact expression of $\mathcal E_N\big(u_\mathrm d^h\big)$ depends on the chosen FE method. The details are given in Sections~\ref{sec:combinedest} and~\ref{sec:gen} for the special case in which the numerical method of choice is IGA.}
		\item \changes{Optimal values for the parameters $\alpha_N$ and $\alpha_D$ in~(\ref{eq:overallestimator}) require the (potentially heuristic) knowledge of the relative value of the efficiency indices corresponding to the estimator $\mathcal E_N\big(u_\mathrm d^h\big)$ of the numerical error, and to the estimator $\mathcal E_D(u_\mathrm d)$ of the defeaturing error. Based on various numerical experiments (see also Section~\ref{sec:numexp}), a good rule of thumb when choosing the estimator $\mathcal{E}_N$ as the residual estimator consists in taking $\alpha_N = 1$ and $\alpha_D = 4$.}
	\end{itemize}
\end{remark}

The discrete defeaturing error estimator $\mathscr E\big(u_\mathrm d^h\big)$ introduced in (\ref{eq:overallestimator}) can easily be decomposed into local contributions. Indeed, $\mathscr{E}_D\big(u_\mathrm d^h\big)$ (and $\mathscr{E}_C$) can readily be decomposed into single feature contributions as follows:
\begin{align}
\mathscr{E}_C^{\,\,2}&=  \sum_{k=1}^{N_f} \left(\mathscr{E}_C^k\right)^2, \quad 
\text{with } \quad \left(\mathscr{E}_C^k\right)^2:=\sum_{\sigma \in \Sigma^k} c_\sigma^2 |\sigma|^{\frac{n}{n-1}}\left| \overline{d_\sigma}^\sigma \right|^2, \quad k=1,\ldots,N_f, \nonumber \\
\mathscr{E}_D\big(u_\mathrm d^h\big)^2 &= \sum_{k=1}^{N_f} \mathscr{E}_D^k\big(u_\mathrm d^h\big)^2, \quad
\text{with } \quad \mathscr{E}_D^k\big(u_\mathrm d^h\big)^2:=\sum_{\sigma \in \Sigma^k} |\sigma|^{\frac{1}{n-1}}\left\| d_\sigma^h - \overline{d_\sigma^h}^\sigma \right\|^2_{0,\sigma} + \left(\mathscr{E}_C^k\right)^2, \quad k=1,\ldots,N_f. \label{eq:localED}
\end{align}
Similarly, the numerical error estimator $\mathscr{E}_N\big(u_\mathrm d^h\big)$ can in general be decomposed into single mesh element contributions $\mathscr{E}_N^K\big(u_\mathrm d^h\big)$ for $K\in\mathcal Q$, in the form
\begin{equation*}
	\mathscr{E}_N\big(u_\mathrm d^h\big)^2 = \sum_{K\in \mathcal Q} \mathscr{E}_N^K\big(u_\mathrm d^h\big)^2.
\end{equation*}
These local decomposition are necessary to guide the adaptive refinement, both in terms of the mesh and in terms of the geometrical model.

\subsection{Mark} \label{sec:mark}
Using a maximum strategy and recalling that $N_f \equiv N_f^{(i)}$ at the current iteration $i$, we select and mark elements $\mathcal M\subset \mathcal Q$ to be refined, and features 
$$\left\{ F^{k} \right\}_{k\in I_\mathrm m}\subset \mathcal F \quad \text{ with } I_\mathrm m \subset \left\{1,\ldots,N_f^{(i)}\right\}$$
to be added to the (partially) defeatured geometry $\Omega_0\equiv\Omega_0^{(i)}$.
That is, after choosing a marking parameter $0<\theta\leq 1$, the marked elements $K_\mathrm m\in \mathcal M$ and the marked features $F^{k_\mathrm m}$ for $k_\mathrm m \in I_\mathrm m$ verify
\begin{align}
\alpha_N\, \mathscr{E}_N^{K_\mathrm m}\big(u_\mathrm d^h\big) &\geq \theta \max\left( \alpha_N \max_{K\in \mathcal Q} \mathscr{E}_N^K\big(u_\mathrm d^h\big), \,\alpha_D\max_{k=1,\ldots,N_f} \mathscr{E}_D^k\big(u_\mathrm d^h\big) \right), \label{eq:markelements}\\
\alpha_D\, \mathscr{E}_D^{k_\mathrm m}\big(u_\mathrm d^h\big)  &\geq \theta \max\left( \alpha_N \max_{K\in \mathcal Q} \mathscr{E}_N^K\big(u_\mathrm d^h\big), \,\alpha_D\max_{k=1,\ldots,N_f} \mathscr{E}_D^k\big(u_\mathrm d^h\big) \right). \label{eq:markfeatures}
\end{align}
In other words, the set of marked elements and the set of selected features are the ones giving the most substantial contribution to the overall error estimator. The smaller is $\theta$, the more elements and features are selected. Note also that the larger is $\alpha_N$ with respect to $\alpha_D$ in (\ref{eq:overallestimator}), the more importance is given to $\mathscr E_N\big(u_\mathrm d^h\big)$ with respect to $\mathscr E_D\big(u_\mathrm d^h\big)$, and vice versa. 

\begin{remark}
	\changes{A Dörfler strategy, also called bulk-chasing strategy, could also be used in a similar manner. The maximum marking strategy has been used here for simplicity. }
\end{remark}


\subsection{Refine} \label{sec:refine}
On the one hand, based on the set $\mathcal M$ of marked elements, the mesh $\mathcal Q$ is refined thanks to an $h$-refinement procedure corresponding to the chosen FE method. During this refinement step, we need to make sure that Assumption \ref{as:discrspacescompat} remains satisfied. 

On the other hand, the defeatured geometry $\Omega_0^{(i)}$ is refined, meaning that the marked features $\left\{ F^{k} \right\}_{k\in I_\mathrm m}$ are inserted in the geometrical model. That is, the new partially defeatured geometrical model $\Omega_0^{(i+1)}$ at the next iteration is built as follows:
\begin{align}
\Omega_0^{\left(i+\frac{1}{2}\right)} &= \Omega_0^{(i)} \setminus \overline{\bigcup_{k\in I_\mathrm m} F_\mathrm  n^k}, \label{eq:Omega0ip12abs}\\
\Omega_0^{(i+1)} &= \mathrm{int}\left( \overline{\Omega_0^{\left(i+\frac{1}{2}\right)}} \cup \overline{\bigcup_{k\in I_\mathrm m} F_\mathrm p^k} \right). \label{eq:Omega0p1abs}
\end{align}
And thus in particular,
\begin{align*} 
F_\mathrm n^{(i+1)} &:= F_\mathrm n^{(i)} \setminus \overline{\bigcup_{k\in I_\mathrm m} {F_\mathrm n^k}}, \quad F_\mathrm p^{(i+1)} := F_\mathrm p^{(i)} \setminus \overline{\bigcup_{k\in I_\mathrm m} {F_\mathrm p^k}}, \\
\tilde F_\mathrm p^{(i+1)} &:= \tilde F_\mathrm p^{(i)} \setminus \overline{\bigcup_{k\in I_\mathrm m} {\tilde F_\mathrm p^k}}, \quad \Omega_\star^{(i+1)} := \Omega \setminus \overline{F_\mathrm p^{(i+1)}},
\end{align*}
and as in definition (\ref{eq:defomega0multi}),
\begin{equation*} 
\Omega_0^{(i+1)} = \text{int}\left( \overline{\Omega_\star^{(i+1)}} \cup \overline{F_\mathrm n^{(i+1)}} \right).
\end{equation*}
\changes{Note that in some iterations, it may happen that either the mesh or the geometry are refined, but not both.}

Once the mesh and the defeatured geometry have been refined, the modules SOLVE and ESTIMATE respectively presented in Section \ref{sec:solve} and Section \ref{sec:estimate} can be called again. To do so, we update $\Omega_0$ to be $\Omega_0^{(i+1)}$, we define $N_f^{(i+1)} := N_f^{(i)} - \#I_\mathrm m$, we update the set of features $\mathcal F$ to be $\mathcal F \setminus \left\{ F^{k} \right\}_{k\in I_\mathrm m}$, and we renumber the features from $1$ to \changes{$N_f^{(i+1)}$.} 
The adaptive loop is continued until a certain given tolerance on the error estimator $\mathscr{E}\big(u_\mathrm d^h\big)$ is reached.

\begin{remark} \label{rmk:noremesh}
	One does not want to remesh the geometrical model when features are added to it, as this would cancel the efforts made by standard $h$-refinement in the previous iterations. Therefore, in order to avoid remeshing when some features are added to the geometrical model, we design here a strategy to be used with mesh-preserving methods such as fictitious domain approaches or immersed methods, for which the computational domain is immersed in a background mesh \cite{hansbo2002unfitted,haslinger2009new,rank2012geometric,burman2015cutfem}. 
	This refinement step will be made clearer in the particular context of IGA in Section \ref{ss:refineiga}. 
\end{remark}

%% file: hierarchicaliga.tex
In this section, we shortly review the IGA method, following \cite{igaanalysis}. We refer to \cite{igabook} for more details about the method and its applications. We then quickly review the notion of HB-splines and their extension to THB-splines \cite{vuong,thb2}, that allow for local mesh refinement in IGA.

\subsection{Standard B-splines}\label{ss:sbs}
Consider two strictly positive integers, $p$ denoting the degree of the B-spline basis functions, and $N$ denoting their number also called number of {degrees of freedom}. Then, let $\Xi:=\left\{\xi_i\right\}_{i=1}^{N+p+1}$
be a knot vector, that is, a non-decreasing sequence of $N+p+1$ real values in the parametric domain $(0,1)$. Using Cox-de Boor formula \cite{coxdeboor}, it is possible to recursively define the corresponding univariate B-spline basis $$\hat{\mathcal{B}}_p = \left\{\hat B_{i,p}:(0,1)\to\mathbb{R}, \,i=1,\ldots,N\right\}.$$ 
Let $k_\xi\in \mathbb{N}\setminus\{0\}$ be the multiplicity of every $\xi\in\Xi$, then $\hat{\mathcal{B}}_p$ is $C^{p-k_\xi}$-continuous in every $\xi\in\Xi$ and $C^\infty$-continuous everywhere else. In the following, we will only use open knot vectors as used in standard CAD, i.e., the first and last knots of $\Xi$ have multiplicity $p+1$. This leads to an interpolatory B-spline basis at both ends of the parametric domain $(0,1)$. 

Multivariate B-spline basis functions are then easily defined as tensor-products of the previously introduced univariate B-splines. More precisely, let $\mathbf p = \left(p_1, \ldots, p_{n}\right)$ be a vector of polynomial degrees, let $\boldsymbol N = \left({N}_1, \ldots, {N}_{n}\right)$ be a vector of number of degrees of freedom in each space direction, let $\Xi^j:=\left\{\xi_i^j\right\}_{i=1}^{N_j+p_j+1}$ be the knot vector corresponding to the parametric direction $j$, and let $\hat B^j_{i,p_j}$ be the $i$-th basis function in the $j$-th direction, $j\in\{1,\ldots,n\}$.
Then the multivariate B-spline basis of degree $\mathbf p$ is given by $$\hat{\mathcal{B}}_\mathbf p := \left\{\hat B_{\mathbf i,\mathbf p} : \mathbf i\in \mathbf I \text{ and } \hat B_{\mathbf i,\mathbf p} := \prod_{j=1}^{n} \hat B_{i_j, p_j}^j : (0,1)^{n}\to\mathbb{R} \right\},$$
where $\mathbf i = \left(i_1, \ldots, i_{n}\right)$ is a multi-index denoting a position in the tensor-product structure, and $\mathbf I$ is the set of such indices.

For $j = 1,\ldots,n$, let us consider the set $Z^j\subset \Xi^j$ of non-repeated knots in the $j$-th direction, written $Z^j := \left\{\zeta_1^j, \ldots, \zeta^j_{M_j}\right\}$ with $M_j\in \mathbb{N}\setminus\{0,1\}$. The values of $Z^j$ are called {breakpoints}, and they form a rectangular grid in the parametric domain $(0,1)^n$,
$$\hat{\mathcal Q} := \left\{\hat K_{\boldsymbol m} := \displaystyle\bigtimes_{j=1}^{n} \left(\zeta_{m_j}^j, \zeta_{m_j+1}^j\right) : \boldsymbol m=\{m_1,\ldots,m_{n}\}, 1\leq m_j\leq M_j-1 \text{ for } j=1,\ldots,n\right\}.$$ 
$\hat{\mathcal Q}$ is called {parametric B\'ezier mesh}, and each $\hat K_{\boldsymbol m}$ is a {parametric element}.
Finally, the support extension of a parametric element $\hat K_{\boldsymbol m}\in\hat{\mathcal Q}$ is defined as
\begin{equation}\label{eq:suppext}
S_\text{ext}\big(\hat K_\mathbf m\big) := \bigtimes_{j=1}^{n}S_\text{ext}\left(\zeta_{m_j}^j, \zeta_{m_j+1}^j\right),
\end{equation}
where if $i$ is the index verifying $p_j+1\leq i \leq {N}_j$ and such that we can uniquely rewrite the interval $\left(\zeta_{m_j}^j, \zeta_{m_j+1}^j\right) = \left(\xi_i^j,\xi^j_{i+1}\right)$, then
$$S_\text{ext}\left(\zeta_{m_j}^j, \zeta_{m_j+1}^j\right) := \left(\xi_{i-p_j}^j, \xi_{i+p_j+1}^j\right).$$

For simplicity, we restrict ourselves to B-splines, but it is possible to generalize the previously introduced concepts to non-uniform B-splines (NURBS). We refer the interested reader to \cite{igabook} and \cite{nurbsbook}.

\subsection{Hierarchical B-splines}\label{ss:hbs}
The tensor-product structure of the multivariate basis functions introduced in the previous section does not allow for local refinement. To overcome this limitation, one of the most successful extensions is given by HB-splines \cite{vuong} and their truncated extension \cite{thbgiannelli}, that we briefly review in this section. \\

Let $\hat D=(0,1)^{n}$ and $L\in\mathbb N$, and consider a sequence $\hat{\mathcal{B}}^0, \hat{\mathcal{B}}^1, \ldots, \hat{\mathcal{B}}^L$ of B-spline bases defined on $\hat D$ such that 
\begin{equation}\label{eq:nestedbsp}
\spn{\hat{\mathcal{B}}^0}\subset \spn{\hat{\mathcal{B}}^1}\subset \ldots\subset \spn{\hat{\mathcal{B}}^L}.
\end{equation}
Moreover, let $\hat{\mathcal Q^\ell}$ be the mesh corresponding to $\mathcal B^\ell$, for all $\ell=0,\ldots,L$.
Then, let $\boldsymbol{\hat{\boldsymbol D}}^L = \left\{\hat D^0,\hat D^1,\ldots,\hat D^L\right\}$ be a hierarchy of nested sub-domains of $\hat D$ of depth $L$, that is, such that $$\hat D =: \hat D^0 \supseteq \hat D^1 \supseteq \ldots \supseteq \hat D^L := \emptyset,$$
and assume that $\hat D^\ell$ is a union of lower level elements, i.e.,
$$\hat D^\ell = \text{int}\left(\bigcup_{\hat K\in \hat{\mathcal K}^{\ell-1}} \overline{\hat K}\right), \quad \hat{\mathcal K}^{\ell-1} \subset \hat{\mathcal Q}^{\ell-1}, \quad \forall \ell = 1,\ldots,L.$$
We are now able to recursively define the HB-spline basis $\hat{\mathcal{H}}=\hat{\mathcal{H}}\left(\boldsymbol{\hat D}^L\right)$ as follows:
\begin{equation} \label{eq:defHBsplines}
\begin{cases}
\hat{\mathcal{H}}^0 := \hat{\mathcal{B}}^0; \\
\hat{\mathcal{H}}^{\ell+1} := \hat{\mathcal{H}}^{\ell+1}_\ell \cup \hat{\mathcal{H}}^{\ell+1}_{\ell+1}, \hspace{2mm} \ell = 0,\ldots,L-2;\\
\hat{\mathcal{H}} := \hat{\mathcal{H}}^{L-1},
\end{cases}
\end{equation}
where for all $\ell = 0,\ldots,L-1$,
\begin{equation} \label{eq:HllHllp1}
\hat{\mathcal{H}}^{\ell+1}_\ell := \left\{\hat B\in \hat{\mathcal{H}}^\ell : \hspace{1mm} \mathrm{supp}\big(\hat B\big) \not\subset \hat D^{\ell+1}\right\}, \quad \hat{\mathcal{H}}_{\ell+1}^{\ell+1} := \left\{\hat B\in\hat{\mathcal{B}}^{\ell+1} : \hspace{1mm} \mathrm{supp}\big(\hat B\big) \subset \hat D^{\ell+1} \right\},
\end{equation}
and as usual for HB-splines, the supports are considered to be open. 
In other words, HB-spline basis functions of level $\ell$ are the B-splines of level $\ell$ whose support is only constituted of elements of level equal or higher than $\ell$, and of at least one element of level $\ell$. 

Furthermore, the {parametric hierarchical mesh} associated to the hierarchy $\boldsymbol{\hat D}^L$ is defined as the union of the active elements of each level, that is,
\begin{equation} \label{eq:paramhiermesh}
\hat{\mathcal Q} := \bigcup_{\ell=0}^L \hat{\mathcal Q}^\ell_A, \quad \text{ with } \quad \hat{\mathcal Q}^\ell_A := \left\{ \hat K\in\hat{\mathcal Q}^\ell : \hat K\subseteq \hat{D}^\ell, \,\hat K\not\subseteq \hat D^{\ell+1}\right\}.
\end{equation}
For all $\hat K \in \hat{\mathcal Q}$ such that $\hat K \in \hat{\mathcal Q}^\ell$ for some $\ell=0,\ldots,L$, we write lev$\big(\hat K\big) = \ell$, and we call $\ell$ the level of $\hat{K}$. 

\subsection{Truncated hierarchical B-splines} \label{ss:thb}
Thanks to the nested property (\ref{eq:nestedbsp}) of B-spline spaces forming a hierarchical space, for all $\ell=1,\ldots,L$, it is possible to write $\hat{B}\in \hat{\mathcal B}^{\ell-1}$ with respect to the B-spline basis functions of level $\ell$ as follows:
$$ \hat{B} = \sum_{\hat{B}_i^{\ell}\in \hat{\mathcal B}^\ell} c_i \hat{B}_i^\ell.$$
Then, let us define the truncation operator $\text{trunc}^\ell$ with respect to level $\ell$ as
$$ \text{trunc}^\ell\hat{B} := \sum_{\hat{B}_i^{\ell}\in \hat{\mathcal B}^\ell\setminus \hat{\mathcal H}_\ell^\ell} c_i \hat{B}_i^\ell, \quad \forall \hat{B}\in \hat{\mathcal B}^{\ell-1},$$
where $\hat{\mathcal H}_\ell^\ell$ is defined in (\ref{eq:HllHllp1}). Note that $ \text{trunc}^L\hat{B} = \hat B$ for all $\hat B\in \hat{\mathcal B}^{L-1}$. 
If we recursively apply this truncation operator to the HB-splines of $\hat{\mathcal H}$ from (\ref{eq:defHBsplines}), we obtain a different basis, the THB-spline basis, spanning the same space as $\hat{\mathcal H}$ while extending its range of interesting properties: in particular, THB-spline basis functions have reduced support and satisfy the partition of unity property (see \cite{buffagiannelli1} for instance). More precisely, the THB-spline basis is defined by
\begin{equation*} 
\hat{\mathcal T} := \left\{ \text{trunc}^{L}\left( \ldots \left(\text{trunc}^{\ell+1}\hat{B}\right)\cdots\right) : \hat{B}\in \hat{\mathcal{B}}^{\ell}\cap\hat{\mathcal H}, \,\ell = 0,\ldots,L-1 \right\}.
\end{equation*}

Moreover, recall definition (\ref{eq:suppext}) of a parametric element support extension, and let us extend it to the hierarchical context. The multilevel support extension of a parametric element $\hat K\in\hat{\mathcal Q}^\ell$ with respect to level $k$, with $0\leq k\leq \ell\leq L$, is defined as follows:
$$S_\text{ext}\big(\hat K, k\big) := S_\text{ext}\big(\hat{K}'\big), \quad \text{with } \hat K'\in \hat{\mathcal Q}^k \text{ and } \hat{K} \subseteq \hat{K}'.
$$

We are now able to define the notion of $\mathcal T$-admissibility for the hierarchical mesh $\hat{\mathcal Q}$ defined in~(\ref{eq:paramhiermesh}), following \cite{reviewadaptiveiga} and \cite{buffagiannelli1}. 
To do so, let us consider the auxiliary domains $\hat \omega^0 := \hat D^0 = \hat D$, and for $\ell = 1,\ldots,L$, 
$$\hat \omega^\ell := \bigcup \left\{ \overline{\hat K} : \hat K\in \hat{\mathcal Q}^\ell, \,S_\text{ext}\big(\hat K, \ell\big)\subseteq \hat{D}^\ell \right\}.$$
In other words, domains $\hat \omega^\ell $ are the regions of $\hat D^\ell$ where all the active basis functions of level $\ell-1$ truncated with respect to level $\ell$ equal to zero. 
With these notions in hand, let us introduce the following definition.

\begin{definition} \label{def:admissibilityparam}
	The mesh $\hat{\mathcal Q}$ is said to be $\mathcal T$-admissible of class $\mu\in\{2,\ldots,L-1\}$ if 
	$$\hat D^\ell \subseteq \hat \omega^{\ell-\mu+1} \quad \text{ for }\ell = \mu, \mu+1, \ldots, L-1.$$
	From \cite[Proposition~9]{buffagiannelli1}, this \changes{implies} that the THB-spline basis functions in $\hat{\mathcal T}$ which take nonzero values over any element $\hat K\in \hat{\mathcal Q}$ belong to at most $\mu$ successive levels.
\end{definition}


\subsection{Isogeometric analysis: physical domain, mesh and discrete space} \label{ss:igarest}
In Sections \ref{ss:sbs}--\ref{ss:thb}, B-splines and their hierarchical variants have only been introduced in the parametric domain $\hat{D} := (0,1)^{n}$. More general spline domains $D$ can then be defined as linear combination of the (truncated hierarchical) B-spline basis functions $\hat B_{\mathbf i,\mathbf p}$ with some {control points} $\{\mathbf{P}_\mathbf i\}_{\mathbf i\in\mathbf I}\subset \mathbb{R}^n$ of the physical domain. That is, $D$ is the image of a mapping $\mathbf F:\hat{D}\to\mathbb R^n$ defined by $\mathbf{F}(\boldsymbol\xi) = \sum_{\mathbf i\in\mathbf I} \hat B_{\mathbf i,\mathbf p}(\boldsymbol\xi)\mathbf{P}_\mathbf i$ for all $\boldsymbol\xi \in \hat{D}$.

If $D$ is a (TH)B-spline geometry determined by (the refinement of) a mapping $\mathbf F:\hat D \to D$, we define the physical hierarchical mesh $\mathcal Q(D)$ corresponding to $\hat{\mathcal Q}$ as
\begin{equation} \label{eq:mesh}
\mathcal Q(D) := \left\{ K := \mathbf F\big(\hat K\big) : \hat K\in \hat{\mathcal Q} \right\},
\end{equation}
on which the following classical assumption is made (see e.g.,\cite{igaanalysis}). 
\begin{assumption}\label{as:isomap}
	The isogeometric mapping $\mathbf F : \hat{D} \to D$ is bi-Lipschitz, $\mathbf F\vert_{\overline{\hat K}} \in C^\infty\Big(\overline{\hat K}\Big)$ for every $\hat K \in \hat{\mathcal Q}$, and $\mathbf F^{-1}\vert_{\overline{K}} \in C^\infty\left(\overline{K}\right)$ for every $K \in \mathcal Q(D)$.
\end{assumption}
All definitions that were previously introduced in the parametric domain $\hat D$ are readily transferred to the physical domain $D$, thanks to the isogeometric mapping $\mathbf F$. In particular, Definition \ref{def:admissibilityparam} is extended to the physical mesh $\mathcal Q(D)$ as follows.
\begin{definition} \label{def:admissibilityphysical}
	The mesh $\mathcal Q(D)$ defined by (\ref{eq:mesh}) is said to be $\mathcal T$-admissible of class $\mu\in\{2,\ldots,L-1\}$ if the underlying parametric mesh $\hat{\mathcal Q}$ is $\mathcal T$-admissible of class $\mu$. 
\end{definition}


Finally, the isogeometric paradigm consists in using the same basis functions for the description of the computational domain and for the finite dimensional space on which one seeks the Galerkin solution of a PDE. That is, the numerical solution of a PDE defined in the THB-spline domain $D$, image of the isogeometric mapping $\mathbf F$, is sought in the finite dimensional space $V^h(D) := \spn{\mathcal T(D)}$, where
\begin{equation} \label{eq:Hdiscrspace}
\mathcal{T}(D) := \left\{B:=\hat B\circ \mathbf{F}^{-1} :  \hat B\in\hat{\mathcal{T}}\right\}.
\end{equation}

In the remaining part of this work, we consider the previously introduced THB-splines basis functions, while noting that everything could equivalently be done with their non-truncated counterpart instead.

%% file: combinedest.tex
In this section, we analyze the \textit{a posteriori} discrete defeaturing error estimator $\mathscr E\big(u_\mathrm d^h\big)$ introduced in~(\ref{eq:overallestimator}) on a given (fixed) defeatured geometrical model $\Omega_0$, from which $N_f\geq 1$ features are missing. 

\changes{We first make the complete analysis in the special case of IGA with THB-splines.} For simplicity in the analysis, and even if it is not strictly needed, we assume that the considered (TH)B-splines are at least $C^1$-continuous. The general case could be treated in a similar way following the classical theory of the standard adaptive FE method. \changes{In the last part of this section, we then discuss into details how to adapt the proof to demonstrate the reliability of the error estimator if a standard $C^0$-FE method is used instead of IGA.\\}

To define the component $\mathscr E_N\big(u_\mathrm d^h\big)$ estimating the numerical error in the IGA context, we first introduce the interior residuals $r$ and the boundary residuals $j$ as follows:
\begin{equation} \label{eq:notationcomplex}
	r:=\begin{cases}f+\Delta u_0^h \,\, \text{ in } \Omega_0\\
		f+\Delta u_k^h \,\, \text{ in } \tilde F_\mathrm p^k, \,\, \forall k=1,\ldots,N_f,
	\end{cases} \text{ and } \quad j:=\begin{cases}
		g-\displaystyle\frac{\partial u_\mathrm d^h}{\partial \mathbf n} & \text{ on } \Gamma_N\setminus\left(\gamma_\mathrm n\cup\gamma_\setminussign\right)\\
		g_0-\displaystyle\frac{\partial u_0^h}{\partial \mathbf n_0} & \text{ on } \gamma_0\\
		\tilde g-\displaystyle\frac{\partial u_k^h}{\partial \tilde{\mathbf n}^k} & \text{ on } \tilde\gamma^k, \,\,\forall k=1,\ldots,N_f.
	\end{cases}
\end{equation}
Furthermore, recalling definitions (\ref{eq:neumannbd}) of $\Gamma_N^0$ and $\tilde \Gamma_N$, let $\mathcal E_0$ be the set of edges (if $n=2$) or faces (if $n=3$) of $\mathcal Q_0$ that are part of $\Gamma_N^0$, let $\tilde{\mathcal E}^k$ be the set of edges or faces of $\tilde{\mathcal Q}^k$ that are part of $\tilde\Gamma_N^k$ for all $k=1,\ldots,N_f$, and let 
\begin{equation}
	\mathcal E := \mathcal E_0 \cup \tilde{\mathcal E}, \quad \tilde{\mathcal E} := \bigcup_{k=1}^{N_f} \tilde{\mathcal E}^k. \label{eq:edgesdef}
\end{equation}
In the sequel, edges are called faces even when $n = 2$.
For all $K\in \mathcal Q$, we denote $h_K:=\text{diam}(K)$ and $h:=\displaystyle\max_{K\in \mathcal Q} h_K$, and we denote $h_E := \text{diam}(E)$ for all $E\in \mathcal E$. Then, assuming that the mesh $\mathcal Q$ fits the boundary of the simplified domain $\Omega_0$, the numerical error estimator is given by 
\begin{align} \label{eq:ENiga}
	\mathscr{E}_N\big(u_\mathrm d^h\big)^2 & := \sum_{K\in \mathcal Q} h_K^2\|r\|^2_{0,K} + \sum_{E\in \mathcal E} h_{E}\|j\|^2_{0,E}.
\end{align}
Moreover, if we let $\mathcal E_K := \left\{ E\in\mathcal E : E\subset \partial K \right\}$ for all $K\in \mathcal Q$, then the corresponding local contribution is 
\begin{align}
	\mathscr{E}_N^K\big(u_\mathrm d^h\big)^2&:=h_K^2\left\| r\right\|_{0,K}^2 + \sum_{E\in \mathcal{E}_K} h_E \left\| j \right\|_{0,E}^2, \quad \forall K\in \mathcal Q. \label{eq:localEN}
\end{align}

In the following, we show that the proposed overall error estimator $\mathscr E\big(u_\mathrm d^h\big)$ is reliable in the case in which the mesh is fitted to the simplified geometry, and under reasonable assumptions. That is, we show that it is an upper bound for the discrete defeaturing error between the analytic solution $u$ of the exact problem~(\ref{eq:weakoriginalpb}) and the discrete numerical solution $u^{h}_{\mathrm d}$ of the defeatured problem introduced in (\ref{eq:defudhmultifeat}), in the energy norm:
\begin{equation}\label{eq:reliability}
\big\| \nabla\left(u-u_\mathrm d^h \right)\big\|_{0,\Omega} \lesssim \mathscr{E}\big(u_\mathrm d^h\big),
\end{equation}
where the hidden constant is independent of the mesh size $h$, the number $L$ of hierarchical levels of the mesh, the number $N_f$ of features, and their size. 
We first demonstrate it in the simplest single feature case ($N_f=1$) in which the only feature $F$ is negative, and then we use this result to derive and prove the reliability of the estimator when $F$ is a generic complex feature. Finally, we generalize this result to multi-feature geometries for which $N_f\geq1$. 

\begin{remark}
	\changes{Equation~(\ref{eq:reliability}) states the reliability of the proposed error estimator, but it does not state its efficiency. The main challenge to obtain an efficient estimator relies on the fact that the quantity of interest is the energy norm of the total error in the exact geometry $\Omega$, while the differential problem is discretized in the defeatured geometry $\Omega_0$ and in the simplified positive components of the features $\tilde F_\mathrm p$. Therefore, to have a reliable estimator of the total error in $\Omega$, one needs to be able to estimate the energy error due the numerical approximation of the problem in the subdomains $\Omega_\star\subset \Omega_0$ and $F_\mathrm p \subset \tilde F_\mathrm p$. Instead, $\mathcal{E}_N\big(u_\mathrm d^h\big)$ may overestimate the numerical error in $\Omega$ as it estimates the numerical error in the energy norm in the whole domain in which the problem is discretized, that is, in $\Omega_0$ and in $\tilde F_\mathrm p$. This problem is strongly related to goal-oriented error estimation, which however mostly considers linear quantities of interest.}
\end{remark}

\subsection{Isogeometric defeaturing problem setting} \label{ss:igadefeat}
Let us assume that the defeatured geometry $\Omega_0$ defined in (\ref{eq:defomega0multi}) is a THB-spline domain generated by a THB-spline basis $\mathcal{T}(\Omega_0)$ (see Section \ref{ss:thb}). Let $\mathcal Q_0 := \mathcal Q(\Omega_0)$ be the hierarchical mesh as defined in (\ref{eq:mesh}) on which the basis $\mathcal T(\Omega_0)$ is built, and let $V^h(\Omega_0)$ be the finite dimensional subspace of $H^1(\Omega_0)$ defined by
$$V^h(\Omega_0) := \spn{\mathcal T(\Omega_0)}.$$
Similarly, assume that for all $k=1,\ldots,N_f$, the positive component extension $\tilde F_\mathrm p^k$ of feature $F^k$ is a THB-spline domain generated by a THB-spline basis $\mathcal{T}\big(\tilde F_\mathrm p^k\big)$. Let $\tilde{\mathcal Q}^k := \mathcal Q\big(\tilde F_\mathrm p^k\big)$ be the hierarchical mesh as defined in (\ref{eq:mesh}) on which the basis $\mathcal{T}\big(\tilde F_\mathrm p^k\big)$ is built, and let $V^h\big(\tilde F_\mathrm p^k\big)$ be the finite dimensional subspace of $H^1\big(\tilde F_\mathrm p^k\big)$ defined by
$$V^h\big(\tilde F_\mathrm p^k\big) := \spn{\mathcal T\big(\tilde F_\mathrm p^k\big)}.$$
Note that in this section, we assume that $\mathcal Q_0$ and $\tilde{\mathcal Q}^k$ are fitted to the simplified geometries $\Omega_0$ and $\tilde F_\mathrm p^k$, respectively, and recall the definition of the global mesh from (\ref{eq:defQhmulti}):
\begin{align*}
\mathcal{Q} := \mathcal Q_0 \cup \tilde{\mathcal Q}\quad \text{ with }\quad \tilde{\mathcal Q}:=\displaystyle\bigcup_{k=1}^{N_f} \tilde{\mathcal Q}^k.
\end{align*}
Moreover, let us make the following shape regularity assumption. 
\begin{assumption}\label{as:shapereg}
	For all $k=1,\ldots,N_f$, the meshes $\mathcal Q_0$ and $\tilde{\mathcal Q}^k$ are shape regular, that is, for all $K\in \mathcal Q_0$ and all $K\in\tilde {\mathcal Q}^k$, $\displaystyle\frac{h_K}{\rho_K} \lesssim 1$, where $\rho_K$ denotes the radius of the largest ball inscribed in $K$. 
\end{assumption}
Under Assumption~\ref{as:shapereg}, we say that $\mathcal Q$ defined in (\ref{eq:defQhmulti}) is shape regular, by abuse of terminology. 
As a consequence of Assumption \ref{as:shapereg}, $|K|^\frac{1}{n} \simeq h_K\simeq h_E$ for all $K\in \mathcal Q$ and all $E\in \mathcal E$ with $E\subset \partial K$, where $\mathcal E$ is the set of Neumann boundary faces as defined in (\ref{eq:edgesdef}). Moreover, specific to IGA with THB-splines, we also make the following assumption on $\mathcal Q$:
\begin{assumption} \label{as:admissibility}
	$\mathcal Q$ is $\mathcal T$-admissible of class $\mu$ for some $\mu\in\mathbb{N}$, $\mu\geq 2$. That is, $\mathcal Q_0$ and $\tilde{\mathcal Q}^k$ are $\mathcal T$-admissible of class $\mu$ for all $k=1,\ldots,N_f$, according to Definition \ref{def:admissibilityphysical}. 
\end{assumption}
Finally, we recall that for all $k=1,\ldots,N_f$, the discrete spaces $V^h(\Omega_0)$ and $V^h\big(\tilde F_\mathrm p^k\big)$ should have compatible traces on $\gamma_{0,\mathrm p}^k$, following Assumption \ref{as:discrspacescompat}.
Then, referring to Sections \ref{ss:fepbstatement} and \ref{ss:multifeature} for the notation, we solve the Galerkin formulation of (\ref{eq:weaksimplpb}) in $\Omega_0$ to obtain the discretized defeatured solution $u_0^h$, followed by the Galerkin formulation of (\ref{eq:weakfeaturepb}) in $\tilde F_\mathrm p^k$ for all $k=1,\ldots,N_f$ to obtain the discretized defeatured solution extensions $u_k^h$. This allows us to define $u_\mathrm d^h$ as in (\ref{eq:defudhmultifeat}), $u_\mathrm d^h$ being the discrete defeatured solution approximating the exact solution $u\in H_{g_D,\Gamma_D}^1(\Omega)$. \\

In the subsequent analysis, we denote by
$$V^h_0(\Omega_0) := V^h(\Omega_0)\cap H_{0,\Gamma_D}^1(\Omega_0)\qquad \text{ and } \qquad V^h_0\big(\tilde F_\mathrm p\big) := V^h\big(\tilde F_\mathrm p\big)\cap H_{0,\gamma_{0,\mathrm p}}^1\big(\tilde F_\mathrm p\big),$$
and as previously discussed, we make the following assumption for simplicity.
\begin{assumption} \label{as:C1}
	$V^h(\Omega_0)\subset C^1(\Omega_0)$ and $V^h\big(\tilde F_\mathrm p^k\big)\subset C^1\big(\tilde F_\mathrm p^k\big)$ for all $k=1,\ldots,N_f$. 
\end{assumption}
Furthermore, as a consequence of Assumption \ref{as:admissibility}, it is possible to build Scott-Zhang-type operators 
\begin{align}
&I_0^h:H_{0,\Gamma_D}^1(\Omega_0)\to V^h_0(\Omega_0) \label{eq:scottzhang}\\
\text{ and } \quad &\tilde I^h_k:H_{0,\gamma_{0,\mathrm p}^k}^1\big(\tilde F_\mathrm p^k\big)\to V^h_0\big(\tilde F_\mathrm p^k\big), \quad \forall k=1,\ldots,N_f,\nonumber
\end{align}
having the following properties (see \cite{corrigendum},\cite[Section~6.1.3]{reviewadaptiveiga}): for all $v\in H_{0,\Gamma_D}^1(\Omega_0)$, for all $k=1,\ldots,N_f$ and all \hbox{$w\in H_{0,\gamma_{0,\mathrm p}^k}^1\big(\tilde F_\mathrm p^k\big)$}, 
\begin{align}
\sum_{K\in\mathcal Q_0} h_K^{-2} \left\|v-I_0^h(v)\right\|^2_{0,K} \lesssim \|\nabla v\|^2_{0,\Omega_0} \quad \text{ and }\quad & \sum_{K\in\mathcal Q_0} \left\|\nabla I_0^h(v)\right\|^2_{0,K} \lesssim \|\nabla v\|^2_{0,\Omega_0},  \label{eq:L2H1scottzhang}  \\
\sum_{K\in\tilde{\mathcal Q}^k} h_K^{-2} \left\|w-\tilde I_k^h(w)\right\|^2_{0,K} \lesssim \|\nabla w\|^2_{0,\tilde F_\mathrm p^k} \quad \text{ and }\quad & \sum_{K\in\tilde{\mathcal Q}^k} \left\|\nabla \tilde I_k^h(w)\right\|^2_{0,K} \lesssim \|\nabla w\|^2_{0,\tilde F_\mathrm p^k}. \label{eq:L2H1scottzhangFp} 
\end{align}
Note that the right equations imply that for all $v\in H_{0,\Gamma_D}^1(\Omega_0)$, for all $k=1,\ldots,N_f$ and all \hbox{$w\in H_{0,\gamma_{0,\mathrm p}^k}^1\big(\tilde F_\mathrm p^k\big)$}, 
\begin{align}
\sum_{K\in\mathcal Q_0} \left\|\nabla \big( v-I_0^h(v) \big)\right\|^2_{0,K} &\lesssim \|\nabla v\|^2_{0,\Omega_0}, \label{eq:H1H1scottzhang}\\
\sum_{K\in\tilde{\mathcal Q}^k} \left\|\nabla \big( w-\tilde I_k^h(w) \big)\right\|^2_{0,K} &\lesssim \|\nabla w\|^2_{0,\tilde F_\mathrm p^k}. \label{eq:H1H1scottzhangFp}
\end{align}

Finally, 
recalling the definition of the sub-domains $\Omega^k$ given by the separability Assumption \ref{def:separated}, let us make the following assumption on the geometry of the features.
\begin{assumption} \label{as:controlsize}
	For all $k=1,\ldots,N_f$, let $\Omega_\star^k$, $\Omega_0^k$ and $ \tilde \Omega^k$ be the sub-domains of, respectively, $\Omega_\star$, $\Omega_0$ and $\tilde \Omega$, relative to $\Omega^k$. More precisely,
	\begin{equation}
	\Omega_\star^k := \Omega^k\cap\Omega_\star, \quad \Omega_0^k := \text{int}\left(\overline{\Omega_\star^k} \cup \overline{F_\mathrm n^k}\right) \quad \text{and} \quad \tilde \Omega^k := \text{int}\left(\overline{\Omega^k}\cup \overline{G_\mathrm p^k}\right) = \text{int}\left(\overline{\Omega^k_\star} \cup \overline{\tilde F_\mathrm p^k}\right).
	\end{equation}
	Then for all $k=1,\ldots,N_f$, there exist generalized Stein extension operators 
	\begin{align} 
	&\mathsf{E}_{\Omega^k_\star\to\Omega^k_0}: H_{0,\Gamma_D\cap\partial \Omega^k}^1\big(\Omega^k_\star\big)\to H_{0,\Gamma_D\cap\partial \Omega^k}^1\big(\Omega^k_0\big), \label{eq:defSteinstar0} \\
	&\mathsf{E}_{\Omega^k\to \tilde \Omega^k}: H_{0,\Gamma_D\cap\partial \Omega^k}^1\big(\Omega^k\big)\to H_{0,\Gamma_D\cap\partial \Omega^k}^1\big(\tilde \Omega^k\big), \label{eq:defSteintostar} \\
	&\mathsf{E}_{\Omega^k_\star\to\tilde\Omega^k}: H^1_{0,\Gamma_D\cap\partial \Omega^k}(\Omega^k_\star)\to H^1_{0,\Gamma_{D}\cap\partial \Omega^k}\big(\tilde\Omega^k\big) \label{eq:defSteinstartilde}
	\end{align}
	which are bounded, that is, they satisfy the following properties: for all $w\in H_{0,\Gamma_D\cap\partial \Omega^k}^1\big(\Omega^k_\star\big)$ and all \hbox{\changes{$v\in H^1_{0,\Gamma_D\cap\partial \Omega^k}\big(\Omega^k\big)$}}, 
	\begin{align} 
	&\left\|\nabla \mathsf{E}_{\Omega^k_\star\to\Omega_0^k}(w)\right\|_{0,\Omega_0^k} \lesssim \left\|\nabla w\right\|_{0,\Omega^k_\star}, \label{eq:extensionppty0} \\
	&\left\|\nabla \mathsf{E}_{\Omega^k\to\tilde\Omega^k}(v)\right\|_{0,\tilde\Omega^k} \lesssim \left\|\nabla v\right\|_{0,\Omega^k}, \label{eq:extensionpptytotilde} \\
	&\left\|\nabla \mathsf{E}_{\Omega^k_\star\to\tilde\Omega^k}(w)\right\|_{0,\tilde\Omega^k} \lesssim \left\|\nabla w\right\|_{0,\Omega^k_\star}.\label{eq:extensionpptystartilde}
	\end{align}
\end{assumption}

Note that such operators are built for a large class of domains in \cite{sauterwarnke}, based on the Stein operator introduced in \cite{stein}. 

\subsection{Reliability of the discrete defeaturing error estimator: negative feature} \label{sec:estneg}
\input{negativeest}

\subsection{Reliability of the discrete defeaturing error estimator: complex feature} \label{sec:estcompl}
\input{complexest}

\subsection{Reliability of the discrete defeaturing error estimator: multiple features} \label{ss:multiest}
\input{multiest}
\changes{
\subsection{Reliability of the error estimator with the standard finite element method}
\input{femgeneralization}

}

%% file: negativeest.tex
In this section, we analyze the proposed estimator in the case of a single negative feature $F$ of $\Omega$, meaning that $F_\mathrm p = \emptyset$. Since we concentrate on the single feature case, we drop the upper index $k$ everywhere, and since the feature is negative, $\gamma = \gamma_\mathrm n$, $\gamma_0 = \gamma_{0,\mathrm n}$, $\mathcal Q = \mathcal Q_0$, and $u_\mathrm d^h = u_0^h\big\vert_\Omega$. Let us recall the definitions of the continuous and discrete defeaturing error terms $d_\gamma\in L^2(\gamma)$ and $d_{\gamma}^h\in L^2(\gamma)$ from (\ref{eq:defdsigmah}), and of the interior and boundary residuals $r\in L^2(\Omega_0)$ and $j\in L^2\big(\Gamma_N^0\big)$ from (\ref{eq:notationcomplex}). 

Then in this context, the discrete defeaturing error estimator defined in (\ref{eq:overallestimator}) writes as follows:
\begin{align}
\mathscr{E}\big(u_0^h\big) &:= \left[\alpha_D^2\mathscr{E}_D\big(u_0^h\big)^2 + \alpha_N^2 \mathscr{E}_N\big(u_0^h\big)^2\right]^\frac{1}{2}, \label{eq:totalerrestneg}
\end{align}
where
\begin{align}
\mathscr{E}_D\big(u_0^h\big)^2 & := \left|\gamma\right|^\frac{1}{n-1} \left\|d_\gamma^h - \overline{d_\gamma^h}^\gamma\right\|_{0,\gamma}^2 + \mathscr{E}_C^2 \qquad \text{ with } \quad \mathscr{E}_C^2 := c_\gamma^2 |\gamma|^\frac{n}{n-1} \left| \overline{d_\gamma}^\gamma \right|^2, \nonumber \\
\mathscr{E}_N\big(u_0^h\big)^2 & := \sum_{K\in \mathcal Q_0} h_K^2\|r\|^2_{0,K} + \sum_{E\in \mathcal E_0} h_{E}\|j\|^2_{0,E}, \nonumber
\end{align}
and $\alpha_D$ and $\alpha_N$ are parameters to be tuned. 

Let us now state and prove the main theorem of this section under the following technical hypothesis.
\begin{assumption} \label{as:respectivesizes}
Let $h_F := \mathrm{diam}(F)$ and $\hQFmin := \min\left\{h_K: K\in \mathcal Q_0, \, K\cap F \neq \emptyset\right\}$. Then we assume that 
$h_F\lesssim \hQFmin$, that is, $F$ is either smaller or about the same size as the mesh that covers it.
\end{assumption}

\begin{remark}\label{rmk:discussionassumption}
This assumption means that asymptotically, the number of elements of $\mathcal Q_0$ intersecting the feature cannot grow indefinitely. In the context of adaptivity with defeaturing, this hypothesis is quite natural. Indeed, if the number of feature elements grows, it means that the error is concentrated in the feature. More precisely, it either means that the defeaturing data $f$ in $F$ and $g_0$ in $\gamma_0$ are badly chosen, or that the feature is important to correctly approximate the exact solution $u$ in $\Omega$. In the first case, $\mathcal E_C$ will be large and the defeaturing data needs to be more accurately chosen. In the second case, the feature $F$ will be added by the adaptive algorithm, and thus no refinement will be needed anymore in $F$. \changes{This means that in practice, the algorithm automatically takes care of satisfying Assumption~\ref{as:respectivesizes}, c.f. for instance the numerical experiments in Section~\ref{sec:numexp}.}
\end{remark}

\begin{theorem}\label{thm:uppernegtoterror}
	In the framework presented in Section \ref{ss:igadefeat}, let $u$ and $u_0^h$ be the solutions of problem~(\ref{eq:weakoriginalpb}) and of the Galerkin formulation of problem~(\ref{eq:weaksimplpb}), respectively, where $\Omega$ is a geometry containing one negative feature $F$. Then under Assumption~\ref{as:respectivesizes}, the energy norm of the discrete defeaturing error is bounded in terms of the estimator $\mathscr E\big(u_0^h\big)$ introduced in (\ref{eq:totalerrestneg}) as follows:
	$$\left\| \nabla\left(u-u_0^h\right) \right\|_{0,\Omega} \lesssim \mathscr E\big(u_0^h\big).$$
\end{theorem}

\begin{proof}
	For all $v\in H_{0,\Gamma_D}^1(\Omega)$, let us first use integration by parts in $\Omega$ for $u_0^h\in V_0^h(\Omega_0)\subset C^1(\Omega_0)$, i.e., 
	\begin{align} \label{eq:integbypartu0h}
		\int_\Omega \nabla u_0^h \cdot \nabla v\, \mathrm dx = - \int_{\Omega} \Delta u_0^h v \,\mathrm dx + \int_{\Gamma_N} \frac{\partial u_0^h}{\partial \mathbf n}v \,\mathrm ds.
	\end{align}
	Then, let $e:=u-u_0^h$. Using (\ref{eq:weakoriginalpb}) for $u$ and using the notation introduced in (\ref{eq:notationcomplex}), we obtain
	\begin{align} 
		\int_\Omega \nabla e \cdot \nabla v \,\mathrm dx = &\int_{\Omega} \left(f + \Delta u_0^h\right) v \,\mathrm dx + \int_{\Gamma_N} \left(g-\frac{\partial u_0^h}{\partial \mathbf n}\right)v \,\mathrm ds \nonumber \\
		= &\int_{\Omega} r v \,\mathrm dx + \int_{\Gamma_N\setminus \gamma} jv \,\mathrm ds + \int_\gamma d_\gamma^h v\,\mathrm ds. \label{eq:errorehv}
	\end{align}
	
	The idea is to suitably extend $v$ to $\Omega_0$ to be able to correctly treat the elements and faces that are only partially in $\Omega$. So by choosing $\Omega^1 := \Omega$ as we are treating the single feature case, and since $\Omega_\star = \Omega$ as we are considering the negative feature case, let $v_0 := \mathsf{E}_{\Omega\to\Omega_0}(v) \in H_{0,\Gamma_D}^1(\Omega_0)$ be the generalized Stein extension of $v$ as defined in (\ref{eq:defSteinstar0}), and recall that $\Gamma_N^0 := \left(\Gamma_N\setminus \gamma\right) \cup \gamma_0$. 
	Then to deal with the elements and faces that are only partially in $\Omega$, and in view of using the Scott-Zhang-type operator properties (\ref{eq:L2H1scottzhang}) and (\ref{eq:H1H1scottzhang}) in $\Omega_0$, we add and subtract terms to (\ref{eq:errorehv}) as follows: 
	\begin{align}
		\int_\Omega \nabla e \cdot \nabla v \,\mathrm dx = \, \RN{1} + \RN{2}, 
		\quad \text{ with } \quad \RN{1} = &\int_{\Omega_0} r v_0\,\mathrm dx + \int_{\Gamma_N^0} jv_0\,\mathrm ds, \nonumber \\
		\RN{2} = &- \int_{F} r v_0\,\mathrm dx - \int_{\gamma_0} jv_0\,\mathrm ds + \int_\gamma d_\gamma^h v\,\mathrm ds. \label{eq:threeterms}
	\end{align}
	As term $\RN{1}$ is defined in $\Omega_0$ and since $\mathcal Q_0$ is fitted to $\Omega_0$, then term $\RN{1}$, which accounts for the numerical error, is defined in a union of full elements. Term $\RN{2}$ accounts for the discrete defeaturing error and the corresponding compatibility condition (see Remark \ref{rmk:compatcond}), and its contributions come from the presence of feature $F$.
	
	Let us first consider $\RN{1}$ and treat it using the Scott-Zhang-type operator $I_0^h$ introduced in (\ref{eq:L2H1scottzhang}). To do so, let $v_0^h = I_0^h(v_0)\in V^h_0(\Omega_0)$, and by adding and substracting $v_0^h$, we can rewrite
	\begin{align}
		\RN{1} = &\int_{\Omega_0} r (v_0-v_0^h)\,\mathrm dx + \int_{\Gamma_N^0} j(v_0-v_0^h)\,\mathrm ds + \int_{\Omega_0} r v_0^h\,\mathrm dx + \int_{\Gamma_N^0} jv_0^h\,\mathrm ds. \label{eq:termI}
	\end{align}
	Using (\ref{eq:weaksimplpb}) and the Galerkin formulation of~(\ref{eq:weaksimplpb}), integrating by parts, and since $v_0^h\in V^h_0(\Omega_0)$, then by Galerkin orthogonality,
	\begin{align}
	\int_{\Omega_0}r v_0^h\,\mathrm dx + \int_{\Gamma_N^0} j v_0^h\,\mathrm ds = \int_{\Omega_0} \nabla (u_0-u_0^h)\cdot\nabla v_0^h = 0. \label{eq:galerkinorthog}
	\end{align}
	Thus from (\ref{eq:galerkinorthog}), using \changes{H\"older's} inequality and the discrete Cauchy-Schwarz
	inequality, (\ref{eq:termI}) can be estimated as follows:
		\begin{align}
	\RN{1} = &\sum_{K\in \mathcal Q_0} \int_{K} r (v_0-v_0^h)\,\mathrm dx + \sum_{E\in\mathcal E_0} \int_{E} j(v_0-v_0^h)\,\mathrm ds \nonumber \\
	\leq &\sum_{K\in\mathcal Q_{0}} h_K \|r\|_{0,K} \,h_K^{-1} \left\| v_0 - v_0^h\right\|_{0,K} + \sum_{E\in \mathcal{E}_0} h_E^\frac{1}{2} \|j\|_{0,E} \,h_E^{-\frac{1}{2}} \left\| v_0 - v_0^h\right\|_{0,E} \nonumber\\
	\leq & \left( \sum_{K\in \mathcal Q_{0}} h_K^2 \|r\|^2_{0,K}\right)^\frac{1}{2} \left( \sum_{K\in \mathcal Q_{0}} h_K^{-2} \left\| v_0 - v_0^h\right\|^2_{0,K}\right)^\frac{1}{2} + \left(\sum_{E\in \mathcal{E}_0} h_E \|j\|_{0,E}^2\right)^\frac{1}{2} \left( \sum_{E\in \mathcal{E}_0} h_E^{-1} \left\| v_0 - v_0^h\right\|_{0,E}^2\right)^\frac{1}{2}. \nonumber 
	\end{align}
	Then, using property (\ref{eq:L2H1scottzhang}) of the Scott-Zhang-type operator since $\mathcal Q_0$ is $\mathcal T$-admissible, and property (\ref{eq:extensionppty0}) of the generalized Stein extension $v_0$ of $v$, we get
	\begin{equation*} 
		\sum_{K\in\mathcal Q_{0}} h_K^{-2} \left\| v_0 - v_0^h \right\|^2_{0,K} \lesssim \|\nabla v_0\|^2_{0,\Omega_0} \lesssim \|\nabla v\|^2_{0,\Omega}.
	\end{equation*}
	Moreover, for every $E\in \mathcal E_0$, let $K_E\in \mathcal Q_0$ be the element such that $E\subset \partial K_E$, and note that by the shape regularity of $\mathcal Q_0$, $h_{K_E} \simeq h_{E}$. Then using the scaled trace inequality of Lemma \ref{lemma:traceineq}, properties (\ref{eq:L2H1scottzhang}) and (\ref{eq:H1H1scottzhang}) of the Scott-Zhang-type operator, and property (\ref{eq:extensionppty0}) of the generalized Stein extension $v_0$ of $v$, we obtain
	\begin{align*}
	\sum_{E\in \mathcal{E}_0} h_E^{-1} \left\| v_0 - v_0^h\right\|_{0,E}^2 &\lesssim \sum_{E\in \mathcal{E}_0} \left( h_{K_E}^{-2} \left\| v_0 - v_0^h \right\|^2_{0,K_E} + \left\| \nabla \left( v_0 - v_0^h\right)\right\|^2_{0,K_E} \right) \lesssim \left\| \nabla v_0\right\|^2_{0,\Omega_0} \lesssim \|\nabla v\|^2_{0,\Omega}. 
	\end{align*}
	Therefore, from the last three inequalities,
	\begin{align}
	\RN{1} \lesssim &\left( \sum_{K\in \mathcal Q_0} h_K^2\|r\|^2_{0,K} + \sum_{E\in \mathcal{E}_0} h_E\|j\|^2_{0,E} \right)^\frac{1}{2}\left\| \nabla v \right\|_{0,\Omega} = \mathscr{E}_N\big(u_0^h\big)\left\| \nabla v \right\|_{0,\Omega}. \label{eq:Iest}
	\end{align}

	Now, let us consider term $\RN{2}$ of (\ref{eq:threeterms}). First, note that by integration by parts, for any constant $c\in\mathbb R$,
	\begin{equation} \label{eq:conservationu0F}
		\int_F rc\,\mathrm dx + \int_{\gamma_0} jc\,\mathrm ds + \int_\gamma \frac{\partial \left(u_0-u_0^h\right)}{\partial \mathbf n_F}c\,\mathrm ds = \int_F \nabla \left(u_0-u_0^h\right)\cdot \nabla c \,\mathrm dx = 0.
	\end{equation}
	Thus, adding (\ref{eq:conservationu0F}) to $\RN{2}$ with the choice of constant $c=\overline{v}^\gamma$, we obtain
	\begin{align}
		\RN{2} &= - \int_{F} r v_0\,\mathrm dx - \int_{\gamma_0} jv_0\,\mathrm ds + \int_\gamma d_\gamma^h v\,\mathrm ds \nonumber \\
		& = - \int_{F} r \left(v_0-\overline{v}^\gamma\right)\,\mathrm dx - \int_{\gamma_0} j\left(v_0-\overline{v}^\gamma\right)\,\mathrm ds + \int_\gamma d_\gamma^h v\,\mathrm ds + \int_\gamma \frac{\partial \left(u_0-u_0^h\right)}{\partial \mathbf n_F} \overline{v}^\gamma \,\mathrm ds \nonumber \\
		&= - \RN{2}_1 + \RN{2}_2,  \label{eq:IIvovoh}\\
		\quad \text{ with } \quad \RN{2}_1 &:= \int_{F} r \left(v_0-\overline{v}^\gamma\right)\,\mathrm dx + \int_{\gamma_0} j\left(v_0-\overline{v}^\gamma\right)\,\mathrm ds, \nonumber \\
		\RN{2}_2 &:= \int_\gamma d_\gamma^h v\,\mathrm ds + \int_\gamma \frac{\partial \left(u_0-u_0^h\right)}{\partial \mathbf n_F} \overline{v}^\gamma \,\mathrm ds. \nonumber
	\end{align}	
	Moreover, recalling the definition of $d_\gamma$ from (\ref{eq:defdsigmah}), we note that 
	$$d_\gamma^h + \frac{\partial \left(u_0-u_0^h\right)}{\partial \mathbf n_F} = d_\gamma.$$
	Using this, we can rewrite $\RN{2}_2$ of (\ref{eq:IIvovoh}) as
	\begin{align*}
		\RN{2}_2 &= \int_\gamma d_\gamma^h \left( v-\overline{v}^\gamma \right)\,\mathrm ds + \int_\gamma d_\gamma \overline{v}^\gamma \,\mathrm ds 
		= \int_\gamma \left(d_\gamma^h - \overline{d_\gamma^h}^\gamma\right) \left(v-\overline{v}^\gamma\right) \,\mathrm ds + \overline{d_\gamma}^\gamma \int_\gamma v \,\mathrm ds.
	\end{align*}
	These terms can be estimated exactly as in \cite[Theorem~4.1]{paper1defeaturing}, that is, 
	\begin{align}
	\RN{2}_2 \lesssim \left( \left|\gamma\right|^\frac{1}{n-1} \left\|d_\gamma^h - \overline{d_\gamma^h}^\gamma\right\|_{0,\gamma}^2 + c_\gamma^2 |\gamma|^\frac{n}{n-1} \left| \overline{d_\gamma}^\gamma \right|^2\right)^\frac{1}{2}\|\nabla v\|_{0,\Omega} = \mathscr{E}_D\big(u_0^h\big) \left\| \nabla v \right\|_{0,\Omega}. \label{eq:IIest}
	\end{align}
	Finally, let us consider term $\RN{2}_1$ of (\ref{eq:IIvovoh}). Remark first that $\overline{v}^\gamma = \overline{v_0}^\gamma$ since $v_0=v$ on $\gamma$ by definition. Thus, 
	\begin{align}
		\RN{2}_1 &= \int_F r\left(v_0-\overline{v_0}^\gamma\right)\,\mathrm dx + \int_{\gamma_0} j\left(v_0-\overline{v_0}^\gamma\right)\,\mathrm ds \nonumber \\
		&\leq \|r\|_{0,F}\left\|v_0-\overline{v_0}^\gamma\right\|_{0,F} + \|j\|_{0,\gamma_0}\left\|v_0-\overline{v_0}^\gamma\right\|_{0,\gamma_0}. \nonumber 
	\end{align}
	Furthermore, by Friedrichs' inequality of Lemma \ref{lemma:friedrichs}, since $\gamma \subset \partial F$ 
	and since $|\partial F| \simeq |\gamma|$, 
	\begin{equation*} 
		\left\|v_0-\overline{v_0}^\gamma\right\|_{0,F} \lesssim h_F \|\nabla v_0\|_{0,F},
	\end{equation*}
	and thus by the trace inequality of Lemma \ref{lemma:traceineq}, 
	\begin{equation*}
		\left\|v_0-\overline{v_0}^\gamma\right\|_{0,\gamma_0} \leq \left\|v_0-\overline{v_0}^\gamma\right\|_{0,\partial F} \lesssim \left(h_F^{-1} \left\|v_0-\overline{v_0}^\gamma\right\|_{0,F}^2 + h_F \|\nabla v_0\|^2_{0,F}\right)^\frac{1}{2} \lesssim h_F^\frac{1}{2} \|\nabla v_0\|_{0,F}. 
	\end{equation*}
	Therefore, combining the last three inequalities and using property (\ref{eq:extensionppty0}) of the generalized Stein extension $v_0$ of $v$, we obtain
	\begin{equation}
		\RN{2}_1 \lesssim \left(h_F\|r\|_{0,F} + h_F^\frac{1}{2} \|j\|_{0,\gamma_0}\right) \|\nabla v_0\|_{0,F} \lesssim \left(h_F^2\|r\|^2_{0,F} + h_F \|j\|^2_{0,\gamma_0}\right)^\frac{1}{2} \|\nabla v\|_{0,\Omega}. \label{eq:subopt}
	\end{equation}
	From Assumption \ref{as:respectivesizes}, $h_F\lesssim \hQFmin$, that is, for all $K\in \mathcal Q_0$ such that $K\cap F \neq \emptyset$, $h_F\lesssim h_K$, and thus it holds
	\begin{align}
		\RN{2}_1 
		\lesssim \mathscr E_N\big(u_0^h\big) \|\nabla v\|_{0,\Omega}. \label{eq:II1opt}
	\end{align}
	
	To conclude, we plug (\ref{eq:Iest}), (\ref{eq:IIest}) and (\ref{eq:II1opt}) into (\ref{eq:IIvovoh}) and (\ref{eq:threeterms}), and thus for all $v\in H_{0,\Gamma_D}^1(\Omega)$, 
	\begin{align}
		\int_\Omega \nabla e \cdot \nabla v \,\mathrm dx \lesssim  \left[ \mathscr{E}_D\big(u_0^h\big)^2 + \mathscr{E}_N\big(u_0^h\big)^2 \right]^\frac{1}{2} \|\nabla v\|_{0,\Omega}= \mathscr{E}\big(u_0^h\big)\|\nabla v\|_{0,\Omega}. \label{eq:finalstep}
	\end{align}
	We conclude by choosing $v=e\in H_{0,\Gamma_D}^1(\Omega)$ in (\ref{eq:finalstep}), and by simplifying $\|\nabla e\|_{0,\Omega}$ on both sides. 
\end{proof}

\begin{remark}\label{rmk:whereAssumptionAppears}
	Without 
Assumption \ref{as:respectivesizes}, i.e., in the case in which the size of the feature is greater than the mesh size on it, $h_F \gg \hQFmin$, then the term $\RN{2}_1$ estimated by (\ref{eq:subopt}) is sub-optimal as the scaling of $F$ is present in front of the residual terms instead of the mesh size. \changes{See also Remark~\ref{rmk:discussionassumption} for the effect of this Assumption in practice.} 
\end{remark}

\begin{remark}
	Note that the estimation of term $\RN{1}$ in (\ref{eq:termI}) gives an alternative proof to the one of \cite[Theorem~11]{buffagiannelli1} in the case of mixed boundary conditions. 
\end{remark}

%% file: complexest.tex
In this section, we extend the result of Section \ref{sec:estneg} by proving the reliability of the proposed \textit{a posteriori} estimator of the discrete defeaturing error in a geometry with one complex feature. So let $F$ be the only complex feature of $\Omega$, i.e., a feature containing both a negative component $F_\mathrm n$ and a positive component $F_\mathrm p$, and let us recall the notation introduced in Section \ref{sec:pbstatement}. In particular, for an illustration of the notation for a geometry with a complex feature, we refer again to Figure~\ref{fig:G0ppty1}. Furthermore, we recall that in the single feature framework, $\tilde u_0^h$ is the numerical approximation of the Dirichlet extension of the defeatured solution $u_0^h$ in $\tilde F_\mathrm p$, where $\tilde F_\mathrm p$ is a simple extension of the positive component of the feature (such as its bounding box), see~(\ref{eq:featurepb}). Moreover, recall definition (\ref{eq:defdsigmah}) of the continuous and discrete defeaturing error terms $d_\sigma\in L^2(\sigma)$ and $d_{\sigma}^h\in L^2(\sigma)$ for all $\sigma \in \Sigma := \left\{\gamma_\mathrm n, \gamma_{0,\mathrm p}, \gamma_\setminussign\right\}$, and definition (\ref{eq:notationcomplex}) of the interior and boundary residuals of $u_0^h$, $r\in L^2(\Omega_0)$ and $j\in L^2\big(\Gamma_N^0\big)$, and of the interior and boundary residuals of $\tilde u_0^h$, $r\in L^2\big(\tilde F_\mathrm p\big)$ and $j\in L^2\big(\tilde \Gamma_N\big)$.

In this context, recalling definition~(\ref{eq:defudhmultifeat}) of $u_\mathrm d^h$ from $u_0^h$ and $\tilde u_0^h$, then the discrete defeaturing error estimator defined in (\ref{eq:overallestimator}) writes as follows:
\begin{align}
\mathscr{E}\big(u_\mathrm d^h\big) &:= \left[\alpha_D^2\mathscr{E}_D\big(u_\mathrm d^h\big)^2 + \alpha_N^2 \mathscr{E}_N\big(u_\mathrm d^h\big)^2\right]^\frac{1}{2}, \label{eq:totalerrestcomplex}
\end{align}
where
\begin{align}
\mathscr{E}_D\big(u_\mathrm d^h\big)^2 & := \sum_{\sigma \in \Sigma} \left|\sigma\right|^\frac{1}{n-1} \left\|d_\sigma^h - \overline{d_\sigma^h}^\sigma\right\|_{0,\sigma}^2 + \mathscr{E}_C^2 \qquad \text{ with } \quad \mathscr{E}_C^2 := \sum_{\sigma \in \Sigma} c_\sigma^2 |\sigma|^\frac{n}{n-1} \left| \overline{d_\sigma}^\sigma \right|^2 , \nonumber \\
\mathscr{E}_N\big(u_\mathrm d^h\big)^2 & := \sum_{K\in \mathcal Q} h_K^2\|r\|^2_{0,K} + \sum_{E\in \mathcal E} h_{E}\|j\|^2_{0,E}, \label{eq:defEDENcompl}
\end{align}
and $\alpha_D$ and $\alpha_N$ are parameters to be tuned.

Let us now state and prove the main theorem of this section under the following hypothesis, generalizing Assumption~\ref{as:respectivesizes} \changes{(see also Remarks~\ref{rmk:discussionassumption} and~\ref{rmk:whereAssumptionAppears})}.
\begin{assumption} \label{as:respectivesizescomp}
	For $S\in\left\{F_\mathrm n, G_\mathrm p, \tilde F_\mathrm p\right\}$, let $h_S := \mathrm{diam}(S)$, let $$\mathcal Q_S = \begin{cases} \mathcal Q_0 &\text{if } S = F_\mathrm n, \\
	\tilde{\mathcal Q} & \text{otherwise},\end{cases}$$
	and let $\hQSmin := \min\left\{h_K: K\in \mathcal Q_S, \, K\cap S \neq \emptyset\right\}$.
	Then we assume that
	$h_S\lesssim \hQSmin$, that is, the feature \changes{and its extension are either smaller or about the same size as the mesh that covers them.}
\end{assumption}
As already discussed in the negative feature case, and as suggested by some numerical experiments presented in Section \ref{sec:numexp}, we will see that 
this assumption can be removed in practice. 

\begin{theorem}\label{thm:uppercomplextoterror}
	In the framework presented in Section \ref{ss:igadefeat}, let $u$ be the weak solution of problem (\ref{eq:weakoriginalpb}), and let $u_\mathrm d^h$  be the discrete defeatured solution defined in~(\ref{eq:defudhmultifeat}), where $\Omega$ is a geometry containing one complex feature $F$. Then under Assumption \ref{as:respectivesizescomp}, the energy norm of the discrete defeaturing error is bounded in terms of the estimator $\mathscr E\big(u_\mathrm d^h\big)$ introduced in (\ref{eq:totalerrestcomplex}) as follows:
	$$\left\| \nabla\left(u-u_\mathrm d^h\right) \right\|_{0,\Omega} \lesssim \mathscr E\big(u_\mathrm d^h\big).$$
\end{theorem}

\begin{proof} 
Let $e:=u-u_\mathrm d^h$. We are looking for an equation for the error similar to (\ref{eq:errorehv}).
To do so, let us consider the exact problem (\ref{eq:originalpb}) restricted to $\Omega_\star$ with the natural Neumann boundary condition on $\gamma_{0,\mathrm p}$, that is, the restriction $u\vert_{\Omega_\star}\in H_{g_D,\Gamma_D}^1(\Omega_\star)$ is the weak solution of 
\begin{equation} \label{eq:uonOmegastarstrong}
\begin{cases}
-\Delta \left(u\vert_{\Omega_\star}\right) = f & \text{in } \Omega_\star \\
u\vert_{\Omega_\star} = g_D & \text{on } \Gamma_D \\
\displaystyle\frac{\partial \left( u\vert_{\Omega_\star}\right)}{\partial \mathbf n} = g & \text{on } \Gamma_N \setminus \gamma_\mathrm p \\
\displaystyle\frac{\partial \left( u\vert_{\Omega_\star}\right)}{\partial \mathbf n_0} = \frac{\partial u}{\partial \mathbf n_0} &\text{on } \gamma_{0,\mathrm p}. 
\end{cases}
\end{equation}
By abuse of notation, we omit the explicit restriction of $u$ to $\Omega_\star$. Then, for all $v_\mathrm n\in H_{0,\Gamma_D}^1(\Omega_\star)$, 
\begin{equation} \label{eq:uonOmegastar}
\int_{\Omega_\star} \nabla u \cdot \nabla v_\mathrm n \,\mathrm dx = \int_{\Omega_\star} f v_\mathrm n \,\mathrm dx + \int_{\Gamma_N\setminus\gamma_\mathrm p} gv_\mathrm n \,\mathrm ds + \int_{\gamma_{0,\mathrm p}} \frac{\partial u}{\partial \mathbf n_0} v_\mathrm n \,\mathrm ds.
\end{equation}
Since the error $e = u-u_0^h$ in $\Omega_\star$, then if we use~(\ref{eq:uonOmegastar}) for $u$, integrate $u_0^h$ by parts as in~(\ref{eq:integbypartu0h}), and use the notations in~(\ref{eq:notationcomplex}), for all $v_\mathrm n\in H_{0,\Gamma_D}^1(\Omega_\star)$, \changes{we get}
\begin{align} 
\int_{\Omega_\star} \nabla e \cdot \nabla v_\mathrm n \,\mathrm dx &= \int_{\Omega_\star} (f+\Delta u_0^h) v_\mathrm n \,\mathrm dx + \int_{\Gamma_N\setminus\gamma_\mathrm p} \left(g-\frac{\partial u_0^h}{\partial \mathbf n}\right) v_\mathrm n \,\mathrm ds + \int_{\gamma_{0,\mathrm p}} \frac{\partial \left(u-u_0^h\right)}{\partial \mathbf n_0} v_\mathrm n \,\mathrm ds \nonumber\\
&= \int_{\Omega_\star} r v_\mathrm n \,\mathrm dx + \int_{\Gamma_N\setminus\gamma} j v_\mathrm n \,\mathrm ds + \int_{\gamma_\mathrm n} d_{\gamma_\mathrm n}^h v_\mathrm n \,\mathrm ds + \int_{\gamma_{0,\mathrm p}} \frac{\partial \left(u-u_0^h\right)}{\partial \mathbf n_0} v_\mathrm n \,\mathrm ds. \label{eq:ev0Omegastar}
\end{align}

Moreover, in a similar fashion as in (\ref{eq:uonOmegastarstrong}), consider the solution of~(\ref{eq:originalpb}), which verifies
\begin{align} \label{eq:uFp}
	\int_{F_\mathrm p} \nabla u \cdot \nabla v_\mathrm p \,\mathrm dx = \int_{F_\mathrm p} f v_\mathrm p \,\mathrm dx + \int_{\gamma_\mathrm p} gv_\mathrm p \,\mathrm ds + \int_{\gamma_{0,\mathrm p}} \frac{\partial u}{\partial \mathbf n_F} v_\mathrm p \,\mathrm ds, \quad  \forall v_\mathrm p\in H^1(F_\mathrm p).
\end{align}
Recall that $\partial F_\mathrm p = \overline{\gamma_\mathrm p} \cup \overline{\gamma_{0,\mathrm p}}$ and $\gamma_\mathrm p = \mathrm{int}\big(\overline{\gamma_\intersign} \cup \overline{\gamma_\setminussign})$, where $\gamma_\intersign$ is the part of $\gamma_\mathrm p$ that is shared with $\partial \tilde F_\mathrm p$ while $\gamma_\setminussign$ is the remaining part of $\gamma_\mathrm p$, see Figure~\ref{fig:G0ppty1}. 
Moreover, since $\mathbf n_0 = -\mathbf n_F$ on $\gamma_{0,\mathrm p}$ and since the error $e = u-\tilde u_0^h$ in $F_\mathrm p$, then if we use~(\ref{eq:uFp}) for $u$, and if we integrate $\tilde u_0^h$ by parts, for all $v_\mathrm p\in H^1(F_\mathrm p)$, 
\begin{align} 
	\int_{F_\mathrm p} \nabla e \cdot \nabla v_\mathrm p \,\mathrm dx &= \int_{F_\mathrm p} (f+\Delta \tilde u_0^h) v_\mathrm p \,\mathrm dx + \int_{\gamma_\mathrm p} \left(g-\frac{\partial \tilde u_0^h}{\partial \mathbf n}\right) v_\mathrm p \,\mathrm ds + \int_{\gamma_{0,\mathrm p}} \frac{\partial \left(u-\tilde u_0^h\right)}{\partial \mathbf n_F} v_\mathrm p \,\mathrm ds \nonumber \\
	&= \int_{F_\mathrm p} r v_\mathrm p \,\mathrm dx + \int_{\gamma_\intersign}j v_\mathrm p \,\mathrm ds + \int_{\gamma_\setminussign}d_{\gamma_\setminussign}^h v_\mathrm p \,\mathrm ds + \int_{\gamma_{0,\mathrm p}} \frac{\partial \left(u-\tilde u_0^h\right)}{\partial \mathbf n_F} v_\mathrm p \,\mathrm ds. \label{eq:evpFp}
\end{align} \\
Before combining (\ref{eq:ev0Omegastar}) and (\ref{eq:evpFp}), let us first consider the terms on $\gamma_{0, \mathrm p}$. That is, since $\mathbf n_0 = -\mathbf n_F$ on $\gamma_{0,\mathrm p}$, then for all $v\in H^\frac{1}{2}(\gamma_{0,\mathrm p})$, 
\begin{align}
	\int_{\gamma_{0,\mathrm p}} \frac{\partial (u-u_0^h)}{\partial \mathbf n_0} v \,\mathrm ds + \int_{\gamma_{0,\mathrm p}} \frac{\partial (u-\tilde u_0^h)}{\partial \mathbf n_F} v \,\mathrm ds &= \int_{\gamma_{0,\mathrm p}} \left( g_0 - \frac{\partial u_0^h}{\partial \mathbf n_0} \right) v \,\mathrm ds + \int_{\gamma_{0,\mathrm p}} \left( -g_0 - \frac{\partial \tilde u_0^h}{\partial \mathbf n_F} \right)v \,\mathrm ds \nonumber \\
	&= \int_{\gamma_{0,\mathrm p}} j v \,\mathrm ds + \int_{\gamma_{0,\mathrm p}} d_{\gamma_{0,\mathrm p}}^h v \,\mathrm ds. \label{eq:g0terms}
\end{align}
Therefore, recalling the definition of $\Gamma_N^0 := \left(\Gamma_N\setminus\gamma\right) \cup \gamma_0$ where $\gamma_0 = \mathrm{int}\big(\overline{\gamma_{0,\mathrm n} \cup \gamma_{0, \mathrm p}}\big)$ and the definition of $\Sigma := \left\{ \gamma_\mathrm n, \gamma_\setminussign, \gamma_{0,\mathrm p} \right\}$, let us combine (\ref{eq:ev0Omegastar}) and (\ref{eq:evpFp}), using (\ref{eq:g0terms}). That is, for all $v\in H^1_{0,\Gamma_D}(\Omega)$, taking $v_\mathrm n := v\vert_{\Omega^\star}$ in~(\ref{eq:ev0Omegastar}) and $v_\mathrm p:= v\vert_{F_\mathrm p}$ in~(\ref{eq:evpFp}), we obtain the following error equation:
\begin{align}
	\int_\Omega \nabla e \cdot \nabla v \,\mathrm dx &= \int_{\Omega_\star} \nabla e \cdot \nabla v \,\mathrm dx + \int_{F_\mathrm p} \nabla e \cdot \nabla v \,\mathrm dx \nonumber \\
	&= \int_{\Omega_\star} r v \,\mathrm dx + \int_{F_\mathrm p} r v \,\mathrm dx + \int_{\Gamma_N^0\setminus \gamma_{0,\mathrm n}} j v \,\mathrm ds + \int_{\gamma_\intersign} j v \,\mathrm ds + \sum_{\sigma\in\Sigma} \int_{\sigma} d_{\sigma}^h v \,\mathrm ds. \label{eq:errrepcomplex}
\end{align}

Let us now fix $v\in H^1_{0,\Gamma_D}(\Omega)$. As for Theorem \ref{thm:uppernegtoterror}, the idea is to suitably extend $v$ to $\Omega_0$ and to $\tilde F_\mathrm p$ in order to correctly treat the elements and faces that are only partially in $\Omega$, and to be able to use Galerkin orthogonality in the simplified domains $\Omega_0$ and $\tilde F_\mathrm p$. However, Galerkin orthogonality in $\tilde F_\mathrm p$ is only valid for discrete functions that vanish on $\gamma_{0,\mathrm p}$. Therefore, using the generalized Stein extensions of Assumption \ref{as:controlsize} with $\Omega^1 := \Omega$ as we are considering the single feature case, let
\begin{align*}
	v_0 := \mathsf{E}_{\Omega_\star\to\Omega_0}\left(v\vert_{\Omega_\star}\right) \in &H_{0,\Gamma_D}^1(\Omega_0),\qquad 
	\tilde v := \mathsf{E}_{\Omega\to\tilde\Omega}(v) \in H^1_{0,\Gamma_D}\Big(\tilde\Omega\Big),\\
	\text{ and } \quad \tilde v_\star &:= \mathsf{E}_{\Omega_\star\to\tilde \Omega}\left(v\vert_{\Omega_\star}\right) \in H_{0,\Gamma_D}^1\Big(\tilde\Omega\Big).
\end{align*}
In particular, we note that 
\begin{align}
v_0 = v \text{ on } \gamma_\mathrm n, \quad \tilde v = v \text{ on } \gamma_\setminussign, \quad \tilde v = \tilde v_\star = v \text{ on } \gamma_{0,\mathrm p}. \label{eq:traces}
\end{align}
Thus if we define $w:=\tilde v\vert_{\tilde F_\mathrm p}-\tilde v_\star\big\vert_{\tilde F_\mathrm p}$, then from (\ref{eq:traces}), $w\in H_{0,\gamma_{0,\mathrm p}}^1(\tilde F_\mathrm p)$.
Moreover, using properties (\ref{eq:extensionppty0})--(\ref{eq:extensionpptystartilde}) of the extension operators, then
\begin{equation}\label{eq:liftingH1}
	\left\|\nabla v_0\right\|_{0,\Omega_0} \lesssim \|\nabla v\|_{0,\Omega_\star}, \quad \left\|\nabla \tilde v\right\|_{0,\tilde\Omega} \lesssim \|\nabla v\|_{0,\Omega} \quad \text{ and } \quad \left\|\nabla \tilde v_\star\right\|_{0,\tilde\Omega} \lesssim \|\nabla v\|_{0,\Omega_\star}.
\end{equation}
And since $\tilde F_\mathrm p \subset \tilde\Omega$ and $\Omega_\star \subset \Omega$, then using (\ref{eq:liftingH1}),
\begin{equation} \label{eq:liftingerr}
	\left\|\nabla w\right\|_{0,\tilde F_\mathrm p} \lesssim \|\nabla v \|_{0,\Omega}. 
\end{equation}

Consequently, since $v = v_0$ on $\Omega_\star$ and $v=\tilde v$ on $F_\mathrm p$, then (\ref{eq:errrepcomplex}) can be rewritten as
\begin{equation} \label{eq:errrepcomplexstein}
\int_\Omega \nabla e \cdot \nabla v \,\mathrm dx = \int_{\Omega_\star} r v_0 \,\mathrm dx + \int_{\Gamma_N^0\setminus \gamma_{0,\mathrm n}} j v_0 \,\mathrm ds + \int_{F_\mathrm p} r \tilde v \,\mathrm dx + \int_{\gamma_\intersign} j \tilde v \,\mathrm ds + \sum_{\sigma\in\Sigma} \int_{\sigma} d_{\sigma}^h v \,\mathrm ds.
\end{equation}
Then, similarly to (\ref{eq:threeterms}) and since $\Omega^\star = \Omega_0\setminus F_\mathrm n$ and $F_\mathrm p = \tilde F_\mathrm p \setminus G_\mathrm p$, we add and subtract terms to~(\ref{eq:errrepcomplexstein}), and then we rearrange them, as follows:
\begin{align}
\int_\Omega \nabla e \cdot \nabla v \,\mathrm dx
= &\left( \int_{\Omega_0} r v_0 \,\mathrm dx - \int_{F_\mathrm n} r v_0 \,\mathrm dx\right) + \left(\int_{\Gamma_N^0} j v_0 \,\mathrm ds - \int_{\gamma_{0,\mathrm n}} j v_0 \,\mathrm ds\right) \nonumber \\
& + \left(\int_{\tilde F_\mathrm p} r \tilde v \,\mathrm dx - \int_{G_\mathrm p} r\tilde v \,\mathrm ds\right) +\left(- \int_{\tilde F_\mathrm p} r \tilde v_\star \,\mathrm dx + \int_{\tilde F_\mathrm p} r\tilde v_\star \,\mathrm ds\right) \nonumber \\
& + \left(\int_{\gamma_\intersign\cup\tilde \gamma} j \tilde v \,\mathrm ds - \int_{\tilde \gamma} j\tilde v\,\mathrm ds\right) + \left(- \int_{\gamma_\intersign\cup\tilde \gamma} j\tilde v_\star \,\mathrm ds + \int_{\gamma_\intersign\cup\tilde \gamma} j\tilde v_\star \,\mathrm ds\right) + \sum_{\sigma\in\Sigma} \int_{\sigma} d_{\sigma}^h v \,\mathrm ds \nonumber \\
= &\int_{\Omega_0} r v_0 \,\mathrm dx + \int_{\Gamma_N^0} j v_0 \,\mathrm ds - \int_{F_\mathrm n} r v_0 \,\mathrm dx - \int_{\gamma_{0,\mathrm n}} j v_0 \,\mathrm ds \nonumber \\
& + \int_{\tilde F_\mathrm p} r \left(\tilde v - \tilde v_\star \right) \,\mathrm dx - \int_{G_\mathrm p} r\tilde v \,\mathrm ds + \int_{\tilde F_\mathrm p} r\tilde v_\star \,\mathrm ds \nonumber \\
& + \int_{\gamma_\intersign\cup\tilde \gamma} j \left(\tilde v - \tilde v_\star \right) \,\mathrm ds - \int_{\tilde \gamma} j\tilde v\,\mathrm ds + \int_{\gamma_\intersign\cup\tilde \gamma} j\tilde v_\star \,\mathrm ds + \sum_{\sigma\in\Sigma} \int_{\sigma} d_{\sigma}^h v \,\mathrm ds. \nonumber \\
= &\,\RN{1}_0 + \tilde{\RN{1}} + \RN{2}_\mathrm n + \tilde{\RN{2}}_\mathrm p + \RN{2}_\mathrm p, \label{eq:threetermscomplex}
\end{align}
where, recalling that $\tilde \Gamma_N := \gamma_\intersign \cup \tilde\gamma$ and $w:=\tilde v\vert_{\tilde F_\mathrm p}-\tilde v_\star\big\vert_{\tilde F_\mathrm p}$, and by simple rearrangement of the terms,
\begin{align*}
\RN{1}_0 = &\int_{\Omega_0} r v_0\,\mathrm dx + \int_{\Gamma_N^0} jv_0\,\mathrm ds, \\
\tilde{\RN{1}} = &\int_{\tilde F_\mathrm p} r w\,\mathrm dx + \int_{\tilde{\Gamma}_N} jw\,\mathrm ds, \\
\RN{2}_\mathrm n = &-\int_{F_\mathrm n} r v_0\,\mathrm dx - \int_{\gamma_{0,\mathrm n}} jv_0\,\mathrm ds + \int_{\gamma_{\mathrm n}} d_{\gamma_\mathrm n}^h v\,\mathrm ds, \nonumber \\
\tilde{\RN{2}}_\mathrm p = &-\int_{G_\mathrm p} r \tilde v\,\mathrm dx - \int_{\tilde\gamma} j\tilde v\,\mathrm ds + \int_{\gamma_\setminussign} d_{\gamma_\setminussign}^h v\,\mathrm ds, \nonumber \\
\RN{2}_\mathrm p = &\int_{\tilde F_\mathrm p} r \tilde v_\star\,\mathrm dx + \int_{\tilde \Gamma_N} j \tilde v_\star\,\mathrm ds + \int_{\gamma_{0,\mathrm p}} d_{\gamma_{0,\mathrm p}}^h v\,\mathrm ds.
\end{align*}
As terms $\RN{1}_0$ and $\tilde{\RN{1}}$ are defined in $\Omega_0$ and $\tilde F_\mathrm p$, respectively, and since $\mathcal Q_0$ is fitted to $\Omega_0$ and $\tilde{\mathcal Q}$ is fitted to $\tilde F_\mathrm p$, then terms $\RN{1}_0$ and $\tilde{\RN{1}}$, which account for the numerical error, are defined in unions of full elements. 
Moreover, terms $\RN{2}_\mathrm n$, $\RN{2}_\mathrm p$ and $\tilde{\RN{2}}_\mathrm p$ account for the discrete defeaturing error and the corresponding compatibility conditions (see Remark \ref{rmk:compatcond}), and their contributions come from the presence of feature $F$; more specifically, they come from the presence of the negative component $F_\mathrm n$, the positive component $F_\mathrm p$, and the extension $\tilde F_\mathrm p$ of the latter.\\

Term $\RN{1}_0$ can be estimated exactly as term $\RN{1}$ of (\ref{eq:threeterms}), using the Galerkin orthogonality coming from the Galerkin approximation of~(\ref{eq:weaksimplpb}), properties (\ref{eq:L2H1scottzhang}) and (\ref{eq:H1H1scottzhang}) of the Scott-Zhang-type operator $I_0^h$, and property~(\ref{eq:liftingH1}) of the generalized Stein extension $v_0$ of $v$, leading to 
\begin{equation*} 
	\RN{1}_0 \lesssim \left( \sum_{K\in \mathcal Q_0} h_K^2\|r\|^2_{0,K} + \sum_{E\in \mathcal{E}_0} h_E\|j\|^2_{0,E} \right)^\frac{1}{2}\left\| \nabla v \right\|_{0,\Omega_\star}.
\end{equation*}
Since $w\in H^1_{0,\gamma_{0,\mathrm p}}\left(\tilde F_\mathrm p\right)$ from (\ref{eq:traces}), then term $\tilde{\RN{1}}$ can be estimated in the same manner. That is, let us use the Galerkin orthogonality coming from the Galerkin approximation of~(\ref{eq:weakfeaturepb}), properties (\ref{eq:L2H1scottzhangFp}) and (\ref{eq:H1H1scottzhangFp}) of the Scott-Zhang-type operator $\tilde I^h$, and the generalized Stein extension property (\ref{eq:liftingerr}) of $w$, to obtain the following estimate:
\begin{equation*} 
\tilde{\RN{1}} \lesssim \left( \sum_{K\in \tilde{\mathcal Q}} h_K^2\|r\|^2_{0,K} + \sum_{E\in \tilde{\mathcal{E}}} h_E\|j\|^2_{0,E} \right)^\frac{1}{2}\left\| \nabla w \right\|_{0,\tilde F_\mathrm p}.
\end{equation*}

Now, let us consider term $\RN{2}_\mathrm n$ of (\ref{eq:threetermscomplex}). It can be estimated exactly as term $\RN{2}$ of (\ref{eq:threeterms}) using the decomposition given in (\ref{eq:IIvovoh}), replacing $F$, $\Omega$, $\gamma$ and $\gamma_{0}$ by $F_\mathrm n$, $\Omega_\star$, $\gamma_{\mathrm n}$ and $\gamma_{0,\mathrm n}$ respectively. Therefore, using integration by parts in $F_\mathrm n$, using \cite[Theorem~5.1]{paper1defeaturing}, Friedrichs' inequality of Lemma \ref{lemma:friedrichs}, the trace inequality of Lemma \ref{lemma:traceineq}, and since $v_0 = v$ on $\gamma_\mathrm n$, we obtain
\begin{align*}
	\RN{2}_\mathrm n \lesssim &\left( \left|\gamma_{\mathrm n}\right|^\frac{1}{n-1} \left\|d_{\gamma_{\mathrm n}}^h - \overline{d_{\gamma_{\mathrm n}}^h}^{\gamma_{\mathrm n}}\right\|_{0,\gamma_{\mathrm n}}^2 + c_{\gamma_{\mathrm n}}^2 \left|\gamma_{\mathrm n}\right|^\frac{n}{n-1} \left| \overline{d_{\gamma_{\mathrm n}}}^{\gamma_{\mathrm n}} \right|^2\right)^\frac{1}{2}\|\nabla v\|_{0,\Omega_\star} \nonumber \\
	&+ \left( h^2_{F_\mathrm n} \|r\|^2_{0,F_\mathrm n} + h_{F_\mathrm n}\|j\|^2_{0,\gamma_{0,\mathrm n}} \right)^\frac{1}{2}\|\nabla v\|_{0,\Omega_\star}. 
\end{align*}
After observing that $G_\mathrm p := \tilde F_\mathrm p \setminus \overline{F_\mathrm p}$ can be seen as a negative feature of the geometry $F_\mathrm p$ for which $\tilde \gamma$ is the simplified boundary replacing $\gamma_\setminussign$, and for which $\gamma_{0,\mathrm p}$ is the Dirichlet boundary, then term $\tilde{\RN{2}}_\mathrm p$ can be estimated in the same manner as term $\RN{2}_\mathrm n$. That is, using integration by parts in $G_\mathrm p$, using \cite[Theorem~5.1]{paper1defeaturing}, Friedrichs' inequality of Lemma \ref{lemma:friedrichs}, the trace inequality of Lemma \ref{lemma:traceineq}, and since $\tilde v = v$ on $\gamma_\setminussign$, we obtain
\begin{align*}
\tilde{\RN{2}}_\mathrm p \lesssim &\left( \left|\gamma_{\setminussign}\right|^\frac{1}{n-1} \left\|d_{\gamma_{\setminussign}}^h - \overline{d_{\gamma_{\setminussign}}^h}^{\gamma_{\setminussign}}\right\|_{0,\gamma_{\setminussign}}^2 + c_{\gamma_{\setminussign}}^2 \left|\gamma_{\setminussign}\right|^\frac{n}{n-1} \left| \overline{d_{\gamma_{\setminussign}}}^{\gamma_{\setminussign}} \right|^2\right)^\frac{1}{2}\|\nabla v\|_{0,\Omega} \nonumber \\
&+ \left( h^2_{G_\mathrm p} \|r\|^2_{0,G_\mathrm p} + h_{G_\mathrm p}\|j\|^2_{0,\tilde\gamma} \right)^\frac{1}{2}\|\nabla v\|_{0,\Omega}. 
\end{align*}
Finally, since 
$$d_{\gamma_{0,\mathrm p}}^h - \frac{\partial \left(\tilde u_0-\tilde u_0^h\right)}{\partial \mathbf n_F} = d_{\gamma_{0,\mathrm p}}$$
and since $\tilde v_\star = v$ on $\gamma_{0,\mathrm p}$ from (\ref{eq:traces}), we can again apply the same steps to estimate $\RN{2}_\mathrm p$. To do so, we replace $F$, $\Omega$, $\gamma$ and $\gamma_0$ by $\tilde F_\mathrm p$, $\Omega_\star$, $\gamma_{0,\mathrm p}$ and $\tilde\Gamma_N$, respectively, in the estimation of $\RN{2}$ from (\ref{eq:IIvovoh}). Therefore, using integration by parts in $\tilde F_\mathrm p$, using \cite[Theorem~5.1]{paper1defeaturing}, Friedrichs' inequality of Lemma \ref{lemma:friedrichs} and the trace inequality of Lemma \ref{lemma:traceineq}, we obtain
\begin{align*}
\RN{2}_\mathrm p \lesssim &\left( \left|\gamma_{0, \mathrm p}\right|^\frac{1}{n-1} \left\|d_{\gamma_{0, \mathrm p}}^h - \overline{d_{\gamma_{0, \mathrm p}}^h}^{\gamma_{0, \mathrm p}}\right\|_{0,\gamma_{0, \mathrm p}}^2 + c_{\gamma_{0, \mathrm p}}^2 \left|\gamma_{0, \mathrm p}\right|^\frac{n}{n-1} \left| \overline{d_{\gamma_{0, \mathrm p}}}^{\gamma_{0, \mathrm p}} \right|^2\right)^\frac{1}{2}\|\nabla v\|_{0,\Omega_\star} \nonumber \\
&+ \left( h^2_{\tilde F_\mathrm p} \|r\|^2_{0,\tilde F_\mathrm p} + h_{\tilde F_\mathrm p}\|j\|^2_{0,\gamma_{0, \mathrm p}} \right)^\frac{1}{2}\|\nabla v\|_{0,\Omega_\star}. 
\end{align*}

Consequently, plugging in the last five inequalities into (\ref{eq:threetermscomplex}), since $\Omega = \text{int}\left(\overline{\Omega_\star} \cup \overline{F_\mathrm p}\right)$, using the discrete Cauchy-Schwarz inequality, and recalling the definition of $\mathscr{E}\big(u_\mathrm d^h\big)$ in (\ref{eq:totalerrestcomplex}), we get
\begin{equation}
\int_\Omega \nabla e \cdot \nabla v \,\mathrm dx \lesssim \left[\mathscr{E}\big(u_\mathrm d^h\big)^2 + \RN{3}^2 \right]^\frac{1}{2} \|\nabla v\|_{0,\Omega}, \label{eq:almostdone}
\end{equation}
where the terms of $\RN{1}_0$ and $\tilde{\RN{1}}$ contribute to the numerical error part $\mathscr E_N$ of the estimator, the first terms in $\RN{2}_\mathrm n$, $\tilde{\RN{2}}_\mathrm p$ and $\RN{2}_\mathrm p$ contribute to the defeaturing error part $\mathscr E_D$ (and thus \changes{including} $\mathscr E_C$) of the estimator, while their last terms are collected in $\RN{3}$, which is defined as
\begin{equation}
\RN{3}^2 := h_{F_\mathrm n}^2 \|r\|^2_{0,F_\mathrm n} + h_{F_\mathrm n}\|j\|^2_{0,\gamma_{0,\mathrm n}} + h_{G_\mathrm p}^2 \|r\|^2_{0,G_\mathrm p} + h_{G_\mathrm p}\|j\|^2_{0,\tilde\gamma} + h_{\tilde F_\mathrm p}^2 \|r\|^2_{0,\tilde F_\mathrm p} + h_{\tilde F_\mathrm p}\|j\|^2_{0,\gamma_{0,\mathrm p}}. \nonumber
\end{equation}
From Assumption \ref{as:respectivesizescomp}, $h_{S}\lesssim \hQSmin$ for all $S\in\left\{F_\mathrm n, G_\mathrm p, \tilde F_\mathrm p\right\}$. Thus for all $K\in \mathcal Q_0$ such that $K\cap F_\mathrm n \neq \emptyset$, $h_{F_\mathrm n}\lesssim h_K$, for all $K\in \tilde{\mathcal Q}$ such that $K\cap G_\mathrm p \neq \emptyset$, $h_{G_\mathrm p}\lesssim h_K$, and for all $K\in \tilde{\mathcal Q}$ such that $K\cap \tilde F_\mathrm p \neq \emptyset$, $h_{\tilde F_\mathrm p}\lesssim h_K$. Consequently, term $\RN{3}$ from (\ref{eq:almostdone}) can be rewritten as (\ref{eq:II1opt}) for the negative feature case, leading to 
\begin{equation}\label{eq:term3}
	\RN{3} \lesssim \mathscr{E}_N\big(u_\mathrm d^h\big).
\end{equation}
Therefore, combining (\ref{eq:almostdone}) and(\ref{eq:term3}), 
\begin{equation}\label{eq:done}
\int_\Omega \nabla e \cdot \nabla v \,\mathrm dx \lesssim \left[\mathscr{E}\big(u_\mathrm d^h\big)^2 +  \mathscr{E}_N\big(u_\mathrm d^h\big)^2 \right]^\frac{1}{2} \|\nabla v\|_{0,\Omega} \simeq \mathscr{E}\big(u_\mathrm d^h\big)\|\nabla v\|_{0,\Omega}.
\end{equation}
To conclude, we choose $v=e\in H_{0,\Gamma_D}^1(\Omega)$ in (\ref{eq:done}), and we simplify $\|\nabla e\|_{0,\Omega}$ on both sides.
\end{proof}

%% file: multiest.tex
In this section, we further extend the result of Section \ref{sec:estcompl} by proving the reliability of the proposed \textit{a posteriori} estimator of the discrete defeaturing error in a geometry with multiple complex features. So let $\mathcal F:= \left\{F^k\right\}_{k=1}^{N_f}$ be the set of $N_f\geq 1$ complex features of $\Omega$, and let us use the notation introduced in Section~\ref{ss:multifeature}. In particular, recall that in the multiple feature framework, $u_k^h$ is the discrete solution of the Dirichlet extension of the defeatured solution $u_0^h$ in $\tilde F_\mathrm p^k$ for all $k=1,\ldots,N_f$, where $\tilde F_\mathrm p^k$ is a simple extension of the feature's positive component (such as its bounding box). 
Moreover, recall definition (\ref{eq:defQhmulti}) of the mesh $\mathcal Q$, recall definitions (\ref{eq:defdsigmah}) of $d_\sigma\in L^2(\sigma)$ and $d_{\sigma}^h\in L^2(\sigma)$ for all $\sigma \in \Sigma$, where $\Sigma$ is defined in (\ref{eq:sigmamulti}), and recall definitions (\ref{eq:notationcomplex}) of the interior and boundary residuals of $u_0^h$, $r\in L^2(\Omega_0)$ and $j\in L^2\big(\Gamma_N^0\big)$, and of the interior and boundary residuals of $u_k^h$, $r\in L^2\Big(\tilde F_\mathrm p^k\Big)$ and $j\in L^2\Big(\tilde \Gamma_N^k\Big)$ for all $k=1,\ldots,N_f$.

In this context, recalling definition (\ref{eq:defudhmultifeat}) of $u_\mathrm d^h$ from $u_0^h$ and $u_k^h$, $k=1,\ldots,N_f$, the discrete defeaturing error estimator defined in (\ref{eq:overallestimator}) writes as follows:
\begin{align}
\mathscr{E}\big(u_\mathrm d^h\big) &:= \left[\alpha_D^2\mathscr{E}_D\big(u_\mathrm d^h\big)^2 + \alpha_N^2 \mathscr{E}_N\big(u_\mathrm d^h\big)^2\right]^\frac{1}{2}, \label{eq:totalerrestmulti}
\end{align}
where $\mathscr{E}_D\big(u_\mathrm d^h\big)$ and $\mathscr{E}_N\big(u_\mathrm d^h\big)$ are defined as in~(\ref{eq:defEDENcompl}), 
and $\alpha_D$ and $\alpha_N$ are parameters to be tuned as in the single feature case.

Let us now state and prove the main theorem of this section, in the case in which every feature in $\mathcal F$ verifies Assumption~\ref{as:respectivesizescomp} \changes{(see also Remarks~\ref{rmk:discussionassumption} and~\ref{rmk:whereAssumptionAppears})}.


\begin{theorem}\label{thm:uppermultitoterror}
	In the framework presented in Section \ref{ss:igadefeat}, let $u$ be the weak solution of problem (\ref{eq:weakoriginalpb}), and let $u_\mathrm d^h$  be the discrete defeatured solution defined in (\ref{eq:defudhmultifeat}), where $\Omega$ is a geometry containing $N_f\geq 1$ complex features satisfying Assumptions~\ref{def:separated} and~\ref{as:respectivesizescomp}. Then the energy norm of the discrete defeaturing error is bounded in terms of the estimator $\mathscr E\big(u_\mathrm d^h\big)$ introduced in (\ref{eq:totalerrestmulti}) as follows:
	\begin{equation}\label{eq:forhiddenconstant}
	\left\| \nabla\left(u-u_\mathrm d^h\right) \right\|_{0,\Omega} \lesssim \mathscr E\big(u_\mathrm d^h\big).
	\end{equation}
\end{theorem}

\begin{proof} 
	The proof is very similar to the one of Theorem~\ref{thm:uppercomplextoterror}. That is, following the same steps and letting $e:=u-u_\mathrm d^h$, we can write for all $v\in H_{0,\Gamma_D}^1(\Omega)$, 
	\begin{equation} \label{eq:threetermsmulti}
	\int_\Omega \nabla e \cdot \nabla v \,\mathrm dx = \RN{1}_0 + \sum_{k=1}^{N_f} \left(\tilde{\RN{1}}^k + \RN{2}^k_\mathrm n + \tilde{\RN{2}}^k_\mathrm p + \RN{2}^k_\mathrm p \right),
	\end{equation}
	where $\RN{1}_0$ is defined as in~(\ref{eq:threetermscomplex}), and for all $k=1,\ldots, N_f$, $\tilde{\RN{1}}^k$, $\RN{2}^k_\mathrm n$, $\tilde{\RN{2}}^k_\mathrm p$ and $\RN{2}^k_\mathrm p$ are defined as $\tilde{\RN{1}}$, $\RN{2}_\mathrm n$, $\tilde{\RN{2}}_\mathrm p$ and $\RN{2}_\mathrm p$ from~(\ref{eq:threetermscomplex}), but for feature $F^k$. Then, these terms can be estimated as in the proof of Theorem~\ref{thm:uppercomplextoterror}. In particular, for the hidden constant in~(\ref{eq:forhiddenconstant}) to be independent of $N_f$, we use Assumptions~\ref{def:separated} and~\ref{as:controlsize} and the discrete Cauchy-Schwarz inequality.
\end{proof}

%% file: femgeneralization.tex
Let us now discuss how the presented proof of the reliability of the discrete defeaturing error estimator can be tailored to the case in which the standard $C^0$-continuous FE method is used instead of IGA. The required adaptations are the following ones:
\begin{itemize}
	\item consider a mesh $\mathcal Q$ which is made of triangles/tetrahedra or of quadrilaterals/hexahedra;
	\item extend the set $\mathcal E$ of Neumann edges defined in~(\ref{eq:edgesdef}) with $\mathcal E_{\mathrm{int}}$, the set of all internal edges of $\tilde{\mathcal Q}^k$ for all $k=1,\ldots,N_f$ and all internal edges of $\mathcal Q_0$;
	\item extend as follows the definition~(\ref{eq:notationcomplex}) of $j$ to $\mathcal E_{\mathrm{int}}$: on every $\gamma_e\in \mathcal E_{\mathrm{int}}$, 
	$j := \displaystyle\frac{1}{2}\left[\frac{\partial u_\mathrm d^h}{\partial \mathbf{n}_e}\right],$ where the square brackets denote the jump of the quantity, and $\mathbf{n}_e$ represents the normal vectors to the edge $\gamma_e$;
	\item consider the same numerical error estimator $\mathscr{E}_N\big(u_\mathrm d^h\big)$ as in~(\ref{eq:ENiga})--(\ref{eq:localEN}), with the previously introduced extended definitions of $\mathcal E$ and $j$. Note that the contributions on the internal edges $\mathcal E_\textrm{int}$ are analogous to the jump contributions on the interfaces between patches appearing in the multipatch case that will be analyzed in Section~\ref{sec:gen};
	\item remove Assumption~\ref{as:admissibility} specific to IGA with THB-splines, and assume that the current meshes $\mathcal Q_0$ and $\tilde{\mathcal Q}^k$, $k=1,\ldots,N_f$ are obtained from a suitable refinement strategy of initial meshes, for instance the newest vertex bisection algorithm, while preserving the trace compatibility Assumption~\ref{as:discrspacescompat}.
\end{itemize}

After these changes, the result of Theorem~\ref{thm:uppermultitoterror} extends to finite elements. In the proof, it is necessary to add the jump contributions $j$ on the internal edges $\mathcal E_\textrm{int}$. This step does not require additional conceptual ideas with respect to the proof in the IGA framework.

%% file: gen2.tex
In the previous section, we assumed that the considered mesh was fitting the boundary of the simplified domain $\Omega_0$. In a geometric adaptive setting as presented in Section~\ref{sec:adaptive}, adding a new feature would require re-meshing the domain in order to satisfy this assumption. To avoid this, the REFINE step presented in Section~\ref{sec:refine} is designed to be used with a mesh-preserving method, which allows to take advantage of the efforts made by standard $h$-refinement in the previous iterations. 
In this section, we study how the adaptive analysis-aware defeaturing strategy presented in Section~\ref{sec:adaptive} can be performed in the special case of IGA with THB-splines, using trimming and multipatch geometry techniques. 

To illustrate this section, let us take the example of an exact geometry $\Omega$ that contains at least two complex features (with both a positive and a negative component), and such that at some iteration of the adaptive strategy presented in Section~\ref{sec:adaptive}, one feature is required to be added to the defeatured geometrical model $\Omega_0$. Then, the negative component of this feature is added by trimming, while its positive component is added with an extra (possibly trimmed) patch. Therefore, $\Omega_0$ is a trimmed multipatch domain at the next iteration, and the adaptive strategy needs to be precised in this case.

Considering trimmed multipatch domains also considerably extends the range of fully defeatured geometries that one can treat in the presented adaptive IGA defeaturing framework. 
Indeed, the image of a single isogeometric mapping $\mathbf F$ (see Section~\ref{ss:igarest}), called patch, limits the definition of $\Omega_0$: it only allows for geometries that are images of the unit square if $n=2$ or the unit cube if $n=3$. However, the previously introduced setting can easily be generalized to open connected domains defined by $N_p\geq1$ trimmed patches, glued together with $C^0$-continuity. 
Consequently, after quickly reviewing trimming and multipatch techniques, we explain in this section the required modifications of the SOLVE and ESTIMATE steps in the case of IGA in trimmed domains, and then we discuss its generalization to trimmed multipatch domains. To finish this section, we provide details of the REFINE step, in the most general context of trimmed multipatch defeatured geometries $\Omega_0$ for which the exact geometry $\Omega$ contains multiple features. 

\subsection{Isogeometric analysis in trimmed domains} \label{ss:trimmed}
Trimmed domains are obtained from a basic Boolean operation between standard domains, and they are nowadays a standard in most commercial CAD software. More precisely, 
suppose that $D^\mathrm u\subset \mathbb{R}^n$ is a domain defined as the image of an isogeometric mapping $\mathbf F: (0,1)^n \rightarrow D^\mathrm u$, generated by a THB-spline basis $\hat{\mathcal T}$ and satisfying Assumption \ref{as:isomap}, as in Section \ref{sec:iga}. Moreover, let $\left\{\omega_i\right\}_{i=1}^{N_t}$ be a set of bounded open domains in $\mathbb R^n$ that are trimmed (i.e., cut) from $D^\mathrm u$ to obtain $D$, the computational domain. That is,
\begin{equation} \label{eq:trimmingdef}
D := D^\mathrm u\setminus \overline{\omega}, \quad \text{ with } \quad \omega:= \text{int}\left(\bigcup_{i=1}^{N_t} \overline{\omega_i}\right).
\end{equation}

In this case, we can generalize the isogeometric paradigm introduced in Section \ref{ss:igarest} to trimmed geometries. That is, let $\mathcal{T}(D^\mathrm u)$ be defined as in (\ref{eq:Hdiscrspace}), and let us consider the basis $\mathcal{T}(D)$ of THB-spline basis functions whose support intersects $D$, i.e.,
\begin{align*}
\mathcal{T}(D) &:= \left\{B\in \mathcal{T}(D^\mathrm u) :  \mathrm{supp}(B)\cap {D} \neq \emptyset\right\}.
\end{align*}
Then, the IGA numerical solution of a PDE defined in a trimmed domain $D$ is sought in the finite dimensional space spanned by the THB-spline basis functions restricted to $D$, that is,
\begin{align}
V^h(D) := \spn{B\vert_{D}: B\in\mathcal T(D)}. \label{eq:VhDtrim}
\end{align} 

\begin{remark}
	In an adaptive mesh refinement framework and from the algorithmic point of view, we need to guarantee that a function $B\in \mathcal T(D)$ is deactivated when all the elements in $\support{B} \cap {D}$ are refined (see the construction of (T)HB-splines in (\ref{eq:defHBsplines})). To do so, we need to add the so-called ghost elements to the set of elements to refine \cite{trimmedshells}. That is, when a trimmed element $K$ is marked for refinement (see Section \ref{sec:mark}), all the elements in
	$$\left\{K' \in \mathcal Q(D^\mathrm u) : K'\cap D = \emptyset, \, \text{lev}(K)=\text{lev}(K') \text{ and } \exists B\in\mathcal T(D^\mathrm u) \text{ such that } K \cup K' \subset \support{B} \right\}$$
	also need to be marked for refinement. Note that this is only needed for algorithmic reasons, while it does not change the active refined basis determined by the marking strategy.
\end{remark}
For more details about isogeometric methods in trimmed domains, the reader is referred to \cite{antolinvreps,puppistabilization}.

\subsection{Multipatch isogeometric analysis} \label{ss:mp}
Multipatch domains are defined as
\begin{equation} \label{eq:multipatchdef}
D := \text{int}\left(\bigcup_{j=1}^{N_p} \overline{D^j}\right),
\end{equation}
with $N_p\geq 1$, where each domain $D^j$ is defined by an isogeometric mapping $\mathbf F^j: (0,1)^n \to D^j$ satisfying Assumption \ref{as:isomap}. {Let us assume that the patches do not overlap in the physical domain, that is, $D^i \cap D^j = \emptyset$ for all $i,j=1,\ldots,N_p$ such that $i\neq j$.} In the case of a multipatch isogeometric domain $D$, we define the corresponding multipatch mesh as
\begin{equation}\label{eq:mpmesh}
\mathcal Q(D) := \bigcup_{j=1}^{N_p} \mathcal Q \left(D^j\right),
\end{equation}
where each $\mathcal Q\left(D^j\right)$ is defined as in (\ref{eq:mesh}). Moreover, let $\hat{\mathcal B}^j$ denote the corresponding B-spline basis associated with each mesh $\mathcal Q\left(D^j\right)$, $j=1,\ldots,N_p$. 

To be able to build a suitable discrete space in a multipatch domain, we require the meshes to be conforming at the interfaces between patches, and we need to impose a $C^0$-continuity at those interfaces. To do so, let
$$\Gamma^{i,j} := \partial D^i \cap \partial D^j, \quad \forall i,j=1,\ldots,N_p, \,i\neq j$$
denote the interfaces between patches, and assume that they satisfy the following assumption.
\begin{assumption} \label{as:interpatch} For all $i,j=1,\ldots,N_p$ such that $i\neq j$, 
	\begin{itemize}
		\item $\Gamma^{i,j}$ is either empty, or a vertex, or the image of a full edge or a full face of $(0,1)^n$ for both parametrizations $\mathbf{F}^i$ and $\mathbf{F}^j$;
		\item for every B-spline basis function $\hat B^i\in \hat{\mathcal B}^i$ such that $\hat B^i\circ (\mathbf{F}^i)^{-1} \neq 0$ on $\Gamma^{i,j}$, there exists a unique B-spline basis function $\hat B^j\in \hat{\mathcal B}^j$ such that
		$\hat B^i\circ (\mathbf{F}^i)^{-1} = \hat B^j\circ (\mathbf{F}^j)^{-1}$ on $\Gamma^{i,j}$;
		\item the control points associated to the interface functions of adjacent patches coincide. 
	\end{itemize}
\end{assumption}
This assumption allows us to ensure the $C^0$-continuity of the discrete functions at the patch interfaces, by associating the corresponding degrees of freedom on each side of the interface. For more details, the reader is referred to \cite[Section~3.2.2]{reviewadaptiveiga}. In this case, the IGA numerical solution of a PDE defined in a multipatch domain $D$ is sought in the finite dimensional space 
$$V^h(D) := \left\{ v^h\in C^0(D): v^h\vert_{D^j} \in \spn{\mathcal B^j}, \,\forall j=1,\ldots,N_p\right\},$$
where for all $j=1,\ldots,N_p$,
$$
\mathcal B^j := \left\{B:=\hat B\circ \left(\mathbf{F}^j\right)^{-1} :  \hat B\in\hat{\mathcal{B}}^j\right\}.
$$ 
Note that $C^0$-continuity is imposed here between patches. The construction of spaces with higher continuity is currently a very active area of research, see \cite{toshniwal2017smooth,kapl2019isogeometric,bracco2020isogeometric} for instance.

Finally, one can easily generalize to multipatch domains the definition of (T)HB-splines basis functions of Sections \ref{ss:hbs} and \ref{ss:thb}. The corresponding discrete isogeometric spaces defined in Section \ref{ss:igarest} can also easily be generalized to this setting, by considering on each hierarchical level a multi-patch space satisfying Assumption~\ref{as:interpatch}. For more details, the reader is referred to \cite{BUCHEGGER2016159} and \cite[Section~3.4]{garau2018algorithms}. 

\subsection{Defeaturing trimmed multipatch domains} \label{sso:trimmeddomains}
In this section, we first discuss the generalization of Section~\ref{sec:combinedest} and of the discrete defeaturing error estimator~(\ref{eq:totalerrestmulti}) to a single patch trimmed geometry $\Omega_0$, relying on \cite{trimmingest}. The main differences come from the fact that the discrete spline space contains functions whose support are cut by the trimming boundary. In the numerical contribution of the error estimator $\mathscr E(u_\mathrm d^h)$, one therefore needs to adapt the mesh-dependent scaling factors in front of the residuals.

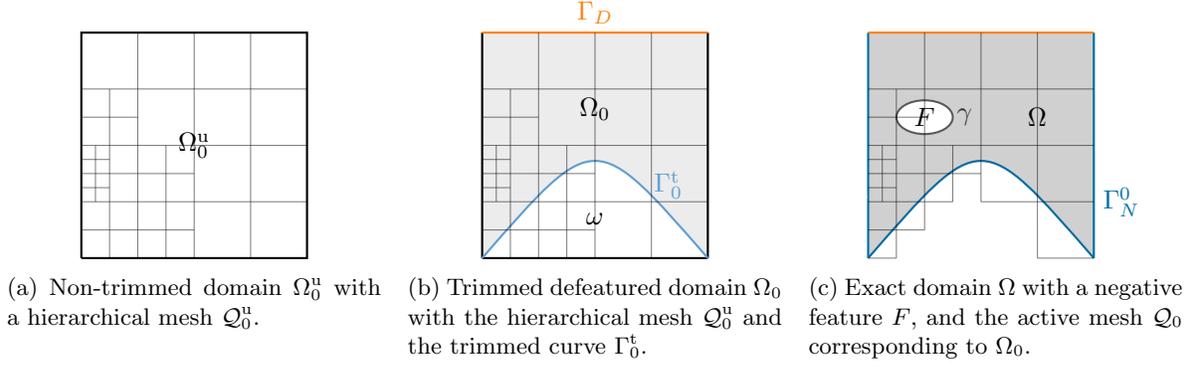
\begin{figure}
	\centering
	\begin{subfigure}[t]{0.3\textwidth}
		\begin{center}
			\begin{tikzpicture}[scale=0.75]
			\draw[thick] (-2,-2) -- (-2,2) -- (2,2) -- (2,-2) -- (-2,-2);
			\draw[opacity=0.5] (-1,-2) -- (-1,2);
			\draw[opacity=0.5] (0,-2) -- (0,2);
			\draw[opacity=0.5] (1,-2) -- (1,2);
			\draw[opacity=0.5] (-2,-1) -- (2,-1);
			\draw[opacity=0.5] (-2,0) -- (2,0);
			\draw[opacity=0.5] (-2,1) -- (2,1);
			\draw[opacity=0.5] (-2,0.5) -- (-1,0.5);
			\draw[opacity=0.5] (-2,-0.25) -- (-1.5,-0.25);
			\draw[opacity=0.5] (-2,-0.75) -- (-1.5,-0.75);
			\draw[opacity=0.5] (-1.75,0) -- (-1.75,-1);
			\draw[opacity=0.5]  (-1.5,-2) -- (-1.5,1);
			\draw[opacity=0.5]  (-0.5,-2) -- (-0.5,0);
			\draw[opacity=0.5]  (-2,-1.5) -- (0,-1.5);
			\draw[opacity=0.5]  (-2,-0.5) -- (0,-0.5);
			\draw (0,0) node{$\Omega_0^\mathrm u$};
			\end{tikzpicture}
			\caption{Non-trimmed domain $\Omega_{0}^\mathrm u$ with a hierarchical mesh $\mathcal Q_{0}^{\mathrm u}$.}
		\end{center}
	\end{subfigure}
	~
	\begin{subfigure}[t]{0.3\textwidth}
		\begin{center}
			\begin{tikzpicture}[scale=0.75]
			\fill[c3!95!black,opacity=0.2] (-2,-2) .. controls (0,0.3) .. (2,-2) -- (2,2) -- (-2,2) -- (-2,-2);
			\draw[thick] (-2,-2) -- (-2,2);
			\draw[thick] (2,-2) -- (2,2);
			\draw[thick] (2,-2) -- (-2,-2);
			\draw[c4,thick] (-2,-2) .. controls (0,0.3) .. (2,-2);
			\draw[c4] (1.3,-0.7) node{$\Gamma^\mathrm t_0$};
			\draw (0,0.65) node{$\Omega_0$} ;
			\draw (0,-1.3) node{$\omega$} ;
			\draw[c2,thick] (-2,2) -- (2,2);
			\draw[c2] (0,2) node[above]{$\Gamma_D$};
			\draw[opacity=0.5] (-1,-2) -- (-1,2);
			\draw[opacity=0.5] (0,-2) -- (0,2);
			\draw[opacity=0.5] (1,-2) -- (1,2);
			\draw[opacity=0.5] (-2,-1) -- (2,-1);
			\draw[opacity=0.5] (-2,0) -- (2,0);
			\draw[opacity=0.5] (-2,1) -- (2,1);
			\draw[opacity=0.5] (-2,0.5) -- (-1,0.5);
			\draw[opacity=0.5] (-2,-0.25) -- (-1.5,-0.25);
			\draw[opacity=0.5] (-2,-0.75) -- (-1.5,-0.75);
			\draw[opacity=0.5] (-1.75,0) -- (-1.75,-1);
			\draw[opacity=0.5]  (-1.5,-2) -- (-1.5,1);
			\draw[opacity=0.5]  (-0.5,-2) -- (-0.5,0);
			\draw[opacity=0.5]  (-2,-1.5) -- (0,-1.5);
			\draw[opacity=0.5]  (-2,-0.5) -- (0,-0.5);
			\end{tikzpicture}
			\caption{Trimmed defeatured domain $\Omega_0$ with the hierarchical mesh $\mathcal Q_{0}^{\mathrm u}$ and the trimmed curve $\Gamma_0^\mathrm t$.}
		\end{center}
	\end{subfigure}
	~
	\begin{subfigure}[t]{0.3\textwidth}
		\begin{center}
			\begin{tikzpicture}[scale=0.75]
			\fill[c3!95!black,opacity=0.5] (-2,-2) .. controls (0,0.3) .. (2,-2) -- (2,2) -- (-2,2) -- (-2,-2);
			\draw[thick,c1] (-2,-2) -- (-2,2);
			\draw[thick,c1] (2,-2) -- (2,2);
			\draw[c1,thick] (-2,-2) .. controls (0,0.3) .. (2,-2);
			\draw[fill,white,thick] (-1,0.5) ellipse (0.5 and 0.3);
			\draw[c,thick] (-1,0.5) ellipse (0.5 and 0.3);
			\draw (1,0.5) node{$\Omega$} ;
			\draw (-1,0.5) node{$F$} ;
			\draw[c!50!black,thick] (-0.6,0.5) node[right]{$\gamma$} ;
			\draw[c2,thick] (-2,2) -- (2,2);
			\draw[opacity=0.5] (-1,-1.5) -- (-1,2);
			\draw[opacity=0.5] (0,-1) -- (0,2);
			\draw[opacity=0.5] (1,-2) -- (1,2);
			\draw[opacity=0.5] (-2,-1) -- (-0.5,-1);
			\draw[opacity=0.5] (0,-1) -- (2,-1);
			\draw[opacity=0.5] (-2,0) -- (2,0);
			\draw[opacity=0.5] (-2,1) -- (2,1);
			\draw[opacity=0.5] (-2,0.5) -- (-1,0.5);
			\draw[opacity=0.5] (-2,-0.25) -- (-1.5,-0.25);
			\draw[opacity=0.5] (-2,-0.75) -- (-1.5,-0.75);
			\draw[opacity=0.5] (-1.75,0) -- (-1.75,-1);
			\draw[opacity=0.5]  (-1.5,-2) -- (-1.5,1);
			\draw[opacity=0.5]  (-0.5,-1) -- (-0.5,0);
			\draw[opacity=0.5]  (-2,-1.5) -- (-1,-1.5);
			\draw[opacity=0.5]  (-2,-0.5) -- (0,-0.5);
			\draw[opacity=0.5] (-2,-2) -- (-1.5,-2);
			\draw[opacity=0.5] (1,-2) -- (2,-2);
			\draw[c1] (2,-1) node[right]{$\Gamma_N^0$};
			\draw[white] (-1.7,-1) node[left]{$\Gamma\;$};
			\end{tikzpicture}
			\caption{Exact domain $\Omega$ with a negative feature $F$, and the active mesh $\mathcal Q_0$ corresponding to $\Omega_0$.}
		\end{center}
	\end{subfigure}
	\caption{Example of a trimmed defeatured geometry with the corresponding notation.} \label{fig:mptrdefeatgeom}
\end{figure}

To do so, let us assume that $\Omega$ contains $N_f\geq 1$ complex features, and that its corresponding defeatured geometry $\Omega_0$ is a domain trimmed from $\Omega_0^\mathrm u\subset \mathbb{R}^n$, as illustrated in Figure~\ref{fig:mptrdefeatgeom}. 
That is, 
\begin{equation} \label{eq:Omega0trim}
	\Omega_0 := \Omega_0^\mathrm u \setminus \overline{\omega} \subset \mathbb R^n, 
\end{equation}
where $\omega$ 
is a union of bounded open domains in $\mathbb R^n$. For simplicity, assume that Neumann boundary conditions are imposed on the trimmed boundary, i.e.,
\begin{equation} \label{eq:Gamma}
	\Gamma_0^\mathrm t := \partial \Omega_0 \setminus \overline{\partial \Omega_{0}^\mathrm u} \subset \Gamma_N^0. 
\end{equation}
Otherwise, the imposition of Dirichlet boundary conditions on the trimmed boundary must be performed in a weak sense, and it would require stabilization techniques, see e.g., \cite{puppistabilization}. 

Moreover, let $\mathcal Q_{0}^{\mathrm u} := \mathcal Q(\Omega_{0}^\mathrm u)$ from~(\ref{eq:mesh}) be (a refinement of) the hierarchical physical mesh on $\Omega_{0}^\mathrm u$, and let $\mathcal E_{0}^{\mathrm u}$ be the set of faces $E$ of $\mathcal Q_{0}^{\mathrm u}$ such that $\left|E\cap\Gamma_N^0\right|>0$. 
Furthermore, for all $k=1,\ldots,N_f$, let the positive component extension $\tilde F_\mathrm p^k$ of feature $F^k$ be a standard THB-spline domain as in Section~\ref{sec:iga}, let $\tilde{\mathcal Q}^k := \mathcal Q\Big(\tilde F_\mathrm p^k\Big)$ be (a refinement of) the hierarchical mesh on $\tilde F_\mathrm p^k$, and let $\tilde{\mathcal E}^k$ be the set of faces of $\tilde{\mathcal Q}^k$ that are part of $\tilde\Gamma_N^k$ for all $k=1,\ldots,N_f$, as in~(\ref{eq:edgesdef}). 
We suppose that $\mathcal Q_{0}^{\mathrm u}$ and $\tilde{\mathcal Q}^{k}$, $k=1,\ldots,N_f$, 
satisfy Assumptions~\ref{as:shapereg} and~\ref{as:admissibility}. 
Then we redefine 
\begin{equation} 
\mathcal Q := \mathcal Q_\trim \cup \mathcal Q_\untr \label{eq:Qhtrim}
\end{equation}
to be the union of the active mesh elements (i.e., the elements intersecting $\Omega_0$ and all $\tilde F_\mathrm p^k$), where
$$\mathcal Q_\trim:= \left\{K\in \mathcal Q^{\mathrm u}_{0} : K\cap \Omega_0 \neq K, |K\cap\Omega_0| > 0 \right\}$$ 
is the set of cut (trimmed) elements, and
$$\mathcal Q_\untr := \left\{K\in \mathcal Q^{\mathrm u}_{0} : K\subset \Omega_0\right\} \cup \tilde{\mathcal Q}, \quad \text{ with } \quad \tilde{\mathcal Q} := \bigcup_{k=1}^{N_f} \tilde{\mathcal Q}^k,$$ 
is the set composed of the other active (non-trimmed) elements in $\Omega_0$ and in all $\tilde F_\mathrm p^k$. Similarly, let $$\mathcal E := \mathcal E_\trim \cup \mathcal E_\untr,$$ where
\begin{align}
&\mathcal E_\trim := \left\{ E\in \mathcal E^{\mathrm u}_{0} : E\cap\Gamma_N^0\neq E \right\} \nonumber \\ 
\text{and} \qquad &\mathcal E_\untr := \left\{ E \in \mathcal E^{\mathrm u}_{0} : E\subset \Gamma_N^0 \right\} \cup \tilde{\mathcal E} \quad \text{ with } \quad \tilde{\mathcal E} := \bigcup_{k=1}^{N_f} \tilde{\mathcal E}^k. \label{eq:Ehtrim} 
\end{align}
Finally, for all $K\in\mathcal Q$, let 
\begin{equation} \label{eq:OmegaGammatK}
\Gamma^\mathrm t_K := \Gamma^\mathrm t_0 \cap \text{int}(K) 
\qquad \text{ and } \qquad \Omega_K := \begin{cases}
\Omega_0 &\text{ if } K\in\mathcal Q_0 \\
\tilde F_\mathrm p &\text{ if } K\in \tilde{\mathcal Q},
\end{cases}
\end{equation}
and for all $E\in\mathcal E$, let
\begin{equation} \label{eq:GammaE}
\Gamma_N^E := \begin{cases} 
\Gamma_N^0 &\text{ if } E\in\mathcal E_0 \\ 
\tilde\Gamma_N &\text{ if } E\in \tilde{\mathcal E}. 
\end{cases}
\end{equation}

In this framework, one can SOLVE the defeaturing problem as presented in Section~\ref{sec:solve}. Then, to ESTIMATE the error, the numerical error contribution $\mathscr{E}_N\big(u_\mathrm d^h\big)$ of the discrete defeaturing error estimator $\mathscr E\big(u_\mathrm d^h\big)$ defined in~(\ref{eq:totalerrestmulti}) needs to be redefined as follows:
\begin{align}\label{eq:ENtrim}
\mathscr{E}_N\big(u_\mathrm d^h\big)^2 := \sum_{K\in \mathcal Q} \delta_K^2\|r\|^2_{0,K\cap\Omega_K} + \sum_{E\in \mathcal E} \delta_{E}^2\|j\|^2_{0,E\cap\Gamma_N^E} + \sum_{K\in\mathcal Q_\trim} h_K\|j\|^2_{0,\Gamma^\mathrm t_K},
\end{align}
where
\begin{align} 
\delta_K &:= \begin{cases}
h_K & \text{ if } K\in\mathcal Q_\untr \\
c_{K\cap\Omega_K} \left|K\cap\Omega_K\right|^\frac{1}{n} & \text{ if } K\in\mathcal Q_\trim,
\end{cases}\nonumber \\
\text{ and } \quad
\delta_E &:= \begin{cases}
h_E^\frac{1}{2} & \text{ if } K\in\mathcal E_\untr \\
c_{E\cap\Gamma_N^E} \left|E\cap\Gamma_N^E\right|^\frac{1}{2(n-1)} & \text{ if } E\in\mathcal E_\trim,
\end{cases}\label{eqo:deltas}
\end{align}
and $c_{K\cap\Omega_K}$ and $c_{E\cap\Gamma_N^E}$ are defined as in~(\ref{eq:cgamma}). The term $\mathscr{E}_D\big(u_\mathrm d^h\big)$ of the estimator~(\ref{eq:totalerrestmulti}) remains unchanged.\\

Then, with the help of the numerical error estimator on trimmed (T)HB-spline geometries from \cite{trimmingest}, the proof of Theorem~\ref{thm:uppermultitoterror} can straightforwardly be generalized to this framework under the following technical assumption replacing Assumption~\ref{as:respectivesizescomp}:
\begin{assumption} \label{as:respectivesizescomptrim}
	For $S\in\left\{F_\mathrm n^k, G_\mathrm p^k, \tilde F_\mathrm p^k\right\}_{k=1}^{N_f}$, let $h_S := \mathrm{diam}(S)$, let 
	\begin{align*}
	\mathcal Q_S = \begin{cases} \mathcal Q_0 &\text{ if } S = F_\mathrm n^k \text{ for some } k=1,\ldots,N_f \\
	\tilde{\mathcal Q}^k &\text{ if } S = G_\mathrm p^k \text{ or } S = \tilde F_\mathrm p^k \text{ for some } k=1,\ldots,N_f,
	\end{cases}
	\end{align*}
	and let 
	$$\deltaQSmin := \min\left\{\delta_K: K\in \mathcal Q_S, \, K\cap S \neq \emptyset\right\}.$$
	Assume that $h_S\lesssim \deltaQSmin$, that is, each feature is either smaller or about the same size as the trimmed mesh that covers it.
\end{assumption}
Unfortunately, in the trimmed case, this technical assumption affects the generality of the result since $\deltaQSmin$ depends on the trimming boundary. As in the non-trimmed case, and as suggested by some numerical experiments presented in Section~\ref{sec:numexp}, we will see that this assumption can be removed in practice. 

\begin{remark} \label{rmk:trimFpk}
	The feature extension $\tilde F_\mathrm p^k$ could also be a trimmed domain for some (or all) $k=1,\ldots,N_f$, and this section easily extends to this case. However, since $\tilde F_\mathrm p^k$ is chosen to be a simple domain containing $F_\mathrm p^k$, then it is more naturally thought as a non-trimmed domain. 
\end{remark}

Now, assume that $\Omega_0$ is a trimmed multipatch geometry, for which patches are glued together with $C^0$-continuity. Then we need to include the jumps between patches of the normal derivative of $u_\mathrm d^h$, in the numerical contribution~(\ref{eq:ENtrim}) of the discrete defeaturing error estimator~(\ref{eq:totalerrestmulti}). 
Under Assumption~\ref{as:respectivesizescomptrim}, and if we assume that the trimmed boundary does not intersect the interfaces between patches, the proof of Theorem~\ref{thm:uppermultitoterror} could easily be adapted to this framework. 
Indeed, its generalization to trimmed geometries has just been discussed. And if one is able to build a Scott-Zhang type operator as in~(\ref{eq:scottzhang}) but on a multipatch domain, then it would be enough to add the normal derivative jump contributions to all terms in~(\ref{eq:threetermsmulti}) (or equivalently in~(\ref{eq:threetermscomplex})), and the proof of Theorem~5.3 from~\cite{trimmingest} would readily be extended in the same way. Even if details are not given, the construction of such a Scott-Zhang type operator for (T)HB-splines on multipatch domains is discussed in \cite[Section~6.1.5]{reviewadaptiveiga}.

\begin{figure}
\begin{subfigure}[t]{0.48\textwidth}
\begin{center}
	\begin{tikzpicture}[scale=3.5]
	\fill[c3, fill, opacity=0.3] (0.4,0.8) rectangle (0.5,1);
	\fill[c3, fill, opacity=0.3] (0.7,1.2) -- (0.7,1) -- (0.4,1) -- (0.4,1.05) -- cycle;
	\fill[pattern=north west lines, pattern color=c, opacity=0.5](0.3,0.7) rectangle (0.6,1);
	\draw[thick] (0,0.3) -- (1,0.3) ;
	\draw[thick] (0,0.3) -- (0,1) ;
	\draw[thick] (1,0.3) -- (1,1) ;
	\draw[thick] (0,1) -- (0.3,1) ;
	\draw[thick] (0.7,1) -- (1,1) ;
	\draw[c1,thick] (0.3,1) -- (0.3,0.7) ;
	\draw[c1,thick,dash dot] (0.4,1.05) -- (0.4,0.8); 
	\draw[c1,thick,dash dot] (0.5,1) -- (0.5,0.8); 
	\draw[c1,thick,dash dot] (0.4,0.8) -- (0.5,0.8); 
	\draw[c1,thick,dash dot] (0.5,1) -- (0.6,1); 
	\draw[c1,thick](0.6,1) -- (0.6,0.7) ;
	\draw[c1,thick,dash dot] (0.7,1) -- (0.7,1.2) ;
	\draw[c1,thick] (0.45,0.7) node[below]{$\gamma_\mathrm n$} ;
	\draw[c1,thick] (0.58,1.15) node[above]{$\gamma_\mathrm p$} ;
	\draw[c1,thick,dash dot] (0.4,1.05) -- (0.7,1.2) ;
	\draw[c1,thick] (0.3,0.7) -- (0.6,0.7) ;
	\draw (0.15,0.45) node{$\Omega$} ;
	\draw[c] (0.6,1.06)node{\small$F_\mathrm p$};
	\draw[c] (0.15,0.92)node{\small $F_\mathrm n$};
	\draw[c] (0.35,0.92) -- (0.23,0.92);
	\end{tikzpicture}
	\caption{Domain with a general complex feature.}
\end{center}
\end{subfigure}
~
\begin{subfigure}[t]{0.48\textwidth}
	\begin{center}
		\begin{tikzpicture}[scale=3.5]
		\fill[c3, fill, opacity=0.3] (0.4,0.8) rectangle (0.5,1);
		\fill[c3, fill, opacity=0.3] (0.7,1.2) -- (0.7,1) -- (0.4,1) -- (0.4,1.2) -- cycle;
		\draw[thick] (0,0.3) -- (1,0.3) ;
		\draw[thick] (0,0.3) -- (0,1) ;
		\draw[thick] (1,0.3) -- (1,1) ;
		\draw[thick] (0,1) -- (0.3,1) ;
		\draw[thick] (0.7,1) -- (1,1) ;
		\draw[thick] (0.3,1) -- (0.3,0.7) ;
		\draw[c1,dash dot] (0.4,1.05) -- (0.7,1.2); 
		\draw[c1,thick] (0.4,1.2) -- (0.4,0.8); 
		\draw[c1,thick] (0.5,1) -- (0.5,0.8); 
		\draw[c1,thick] (0.4,0.8) -- (0.5,0.8); 
		\draw[c1,thick] (0.5,1) -- (0.6,1); 
		\draw[thick] (0.6,1) -- (0.6,0.7) ;
		\draw[c1,thick] (0.7,1) -- (0.7,1.2) ;
		\draw[c2,thick] (0.6,1) -- (0.7,1);
		\draw[c2,thick] (0.7,1) node[below]{$\gamma_{0,\mathrm p}$} ;
		\draw[c1,thick] (0.4,1.2) -- (0.7,1.2) ;
		\draw[thick] (0.3,0.7) -- (0.6,0.7);
		\draw (0.85,0.45) node{$\Omega_\star$} ;
		\draw (0.55,1.1)node{\small$\tilde F_\mathrm p$};
		\end{tikzpicture}
		\caption{Domain $\Omega_\star$ with the simplified extension $\tilde F_\mathrm p$ of the positive component $F_\mathrm p$.}
	\end{center}
\end{subfigure}
\caption{Example of domain containing one feature whose simplified positive component $\tilde F_\mathrm p$ cannot be represented by a single patch.}
\label{fig:mpftilde}
\end{figure}
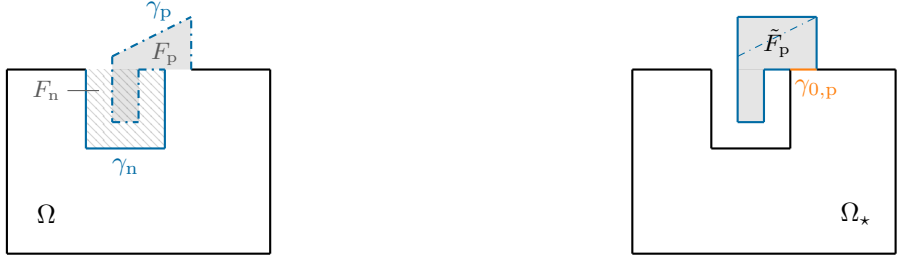

\begin{remark}
	Not only $\Omega_0$, but also the feature extension $\tilde F_\mathrm p^k$ could be a (possibly trimmed) multipatch domain for some $k=1,\ldots,N_f$. In fact, since $\tilde F_\mathrm p^k$ is chosen to be a simple domain containing $F_\mathrm p^k$, then it is more naturally thought as a single patch domain. But since its definition requires $\gamma_{0, \mathrm p} \subset \partial \tilde F_\mathrm p^k$ and $F_\mathrm p^k \subset \tilde F_\mathrm p^k$, then one sometimes has to build it as a multipatch geometry; an example is given in Figure~\ref{fig:mpftilde}. 
\end{remark}

%% file: refineiga.tex
\subsection{Refine by preserving meshes and discrete spaces properties} \label{ss:refineiga}
In this section, we make more precise the REFINE step of the adaptive analysis-aware defeaturing strategy presented in Section \ref{sec:refine} in the context of IGA on multipatch trimmed geometries. 
Note that one could equally choose to update the (partially) defeatured geometrical model first and then refine the mesh, or inversely. 

At each iteration, we need to make sure that the refined meshes and their corresponding discrete spaces satisfy Assumptions~\ref{as:discrspacescompat} and~\ref{as:admissibility}. That is, for all $k=1,\ldots,N_f$, we need to enforce the compatibility of traces on $\gamma_{0, \mathrm p}^k$ between $V^h(\Omega_0)$ and $V^h\big(\tilde F_\mathrm p^k\big)$, and we need to preserve the class of $\mathcal T$-admissibility of all meshes, $\mathcal Q_0$ and $\tilde{\mathcal Q}^k$. To do so, we extend the mesh refinement strategy presented in \cite{buffagiannelli1,Bracco:263433,bracco} that ensures the refined mesh to be of the same class of $\mathcal T$-admissibility $\mu$ as the original mesh. A similar procedure can also be found in \cite[Algorithm~3.7]{3dbem}. For the full details, the interested reader is referred to \cite{phdthesis}. 

Then, if we let $\Omega_0^{(i)}$ be the partially defeatured domain at iteration $i$ of the adaptive procedure, and from the set of marked features $\left\{ F^{k} \right\}_{k\in I_\mathrm m}$, we follow (\ref{eq:Omega0ip12abs}) and (\ref{eq:Omega0p1abs}) to build the partially defeatured domain $\Omega_0^{(i+1)}$ at iteration $i+1$. 
That is, in the first half step (\ref{eq:Omega0ip12abs}), domain $\Omega_0^{\left(i+\frac{1}{2}\right)}$ is built by trimming as in (\ref{eq:trimmingdef}): the negative part of marked features $\bigcup_{k\in I_\mathrm m} F_\mathrm  n^k$ is trimmed from $\Omega_0^{(i)}$ to obtain $\Omega_0^{\left(i+\frac{1}{2}\right)}$. Then from definition (\ref{eq:defGpk}) of $G_\mathrm p^k$, for all $k\in I_\mathrm m$, 
$$F_\mathrm p^k = {\tilde F_\mathrm p^k} \setminus \overline{G_\mathrm p^k}.$$
So $F_\mathrm p^k$ is built by trimming as in (\ref{eq:trimmingdef}), that is, $G_\mathrm p^k$ is trimmed from $\tilde{F_\mathrm p^k}$ to obtain $F_\mathrm p^k$. Finally, in the second half step (\ref{eq:Omega0p1abs}), we glue together the patches of ${\Omega_0^{\left(i+\frac{1}{2}\right)}}$ with the patches of $F_\mathrm p^k$ for all $k\in I_\mathrm m$ with $C^0$-continuity as in (\ref{eq:multipatchdef}), and we obtain $\Omega_0^{(i+1)}$.

%% file: numexamples2.tex
In this section, we perform numerical experiments to illustrate the validity of the proposed \textit{a posteriori} discrete defeaturing error estimator~(\ref{eq:totalerrestmulti}), generalized to trimmed multipatch domains in Section~\ref{sec:gen}. Thanks to these experiments, we also demonstrate that the adaptive procedure presented in Section \ref{sec:adaptive} and made more precise in the IGA framework in Sections~\ref{sec:combinedest} and~\ref{sec:gen} ensures the convergence of the discrete defeaturing error $\left\|\nabla\left(u-u_\mathrm d^h\right)\right\|_{0,\Omega}$. 
All the numerical experiments presented in this section have been implemented with the help of GeoPDEs \cite{geopdes}, \changes{an} open-source and free Octave/Matlab package for the resolution of PDEs specifically designed for IGA. For the local meshing process required for the integration of trimmed elements, an in-house tool presented in \cite{antolinvreps} has been used and linked with GeoPDEs. 
The proposed strategy combining mesh and geometric adaptivity has been implemented on top of the already-existing THB-spline based isogeometric mesh refinement strategy of {GeoPDEs}. In particular, a specific module has been created for defeaturing, that takes care of the estimation of the contribution $\mathscr{E}_D\big(u_\mathrm d^h\big)$, and of the adaptive construction of partially defeatured geometries using multipatch and trimming techniques.

If not otherwise specified, THB-splines of degree $p=2$ are used in this section. Moreover, we call $N_\mathrm{dof}$ the total number of active degrees of freedom of the considered geometrical model $\Omega_0$. More precisely, $N_\mathrm{dof}$ accounts for the number of active degrees of freedom in $\Omega_0$, to which we add the number of active degrees of freedom in $\tilde F^k$ for all $k=1,\ldots,N_f$. If at some point, all features are added to the geometrical model, that is, if the geometrical model is the exact geometry $\Omega$, then $N_\mathrm{dof}$ accounts for the total number of active degrees of freedom in $\Omega$.

\subsection{Convergence of the discrete defeaturing error and estimator} \label{sec:cvtest}
With the two numerical experiments presented in this section, we analyze the convergence of the proposed estimator with respect to the mesh size $h$ and with respect to the size of a given feature. We also compare the convergence of the estimator with the convergence of the discrete defeaturing error. The first experiment is performed in a geometry with one negative feature, while the second one is performed in a geometry containing one positive feature. In both experiments of this section, we only consider global mesh refinements without performing the proposed adaptive strategy yet, and we consider $\alpha_D = \alpha_N = 1$. 

\subsubsection{Negative feature} \label{sec:cvtestneg}
In this experiment, we consider a geometry with one negative feature. More precisely, for $k=-3,-2,\ldots,6$, let $\varepsilon =\displaystyle \frac{10^{-2}}{2^k}$ and let $\Omega_\varepsilon := \Omega_0 \setminus \overline{F_\varepsilon}$, where $\Omega_0$ is the disc centered at $(0,0)^T$ of radius $0.5$, and  $F_\varepsilon$ is the disc centered at $(0,0)^T$ of radius $\varepsilon<0.5$. The geometry is illustrated in Figure~\ref{fig:exdomain}. In other words, $\Omega_0$ is the defeatured geometry obtained from $\Omega_\varepsilon$ by filling the negative feature $F_\varepsilon$. 

\begin{figure}
	\centering
	\begin{subfigure}[t]{0.45\textwidth}
		\begin{center}
			\begin{tikzpicture}[scale=4]
			\draw[thick,fill=gray!30] (0,0) circle (0.5) ;
			\draw[thick,fill=white] (0,0) circle (0.1) ;
			\fill (0,0) circle (0.01) ;
			\draw[->] (0,0) -- (0.1,0) ;
			\draw[->] (0,0) -- (0.353553,-0.353553) ;
			\draw (0.13,-0.25) node{$0.5$};
			\draw (0.05,0.035) node{$\varepsilon$} ;
			\end{tikzpicture}
			\caption{Exact domain $\Omega_\varepsilon$.}\label{fig:exdomain}
		\end{center}
	\end{subfigure}
	~
	\begin{subfigure}[t]{0.5\textwidth}
		\begin{center}
			\begin{tikzpicture}[scale=4]
			\draw[thick] (0,0) circle (0.5) ;
			\draw[thick] (-0.25,-0.25) -- (-0.25,0.25) -- (0.25,0.25) -- (0.25,-0.25) -- cycle;
			\draw[thick] (0.25,0.25) -- (0.353553390593274, 0.353553390593274);
			\draw[thick] (-0.25,-0.25) -- (-0.353553390593274, -0.353553390593274);
			\draw[thick] (-0.25,0.25) -- (-0.353553390593274, 0.353553390593274);
			\draw[thick] (0.25,-0.25) -- (0.353553390593274, -0.353553390593274);
			\draw[gray] (-0.5,0) -- (0.5,0);
			\draw[gray] (0,-0.5) -- (0,0.5);
			\draw[gray] (0.301776695296637, 0.301776695296637) --
			(0.296601198941337, 0.304108322256964) --
			(0.291378345216663, 0.306406036144739) --
			(0.286108652750058, 0.308669138165459) --
			(0.280792662718130, 0.310896933879288) --
			(0.275430938843308, 0.313088733675021) --
			(0.270024067370295, 0.315243853247709) --
			(0.264572657021958, 0.317361614079393) --
			(0.259077338934321, 0.319441343922416) --
			(0.253538766570351, 0.321482377284748) --
			(0.247957615612248, 0.323484055916748) --
			(0.242334583832009, 0.325445729298769) --
			(0.236670390940032, 0.327366755129009) --
			(0.230965778411582, 0.329246499810984) --
			(0.225221509290970, 0.331084338940002) --
			(0.219438367973323, 0.332879657787993) --
			(0.213617159963875, 0.334631851786050) --
			(0.207758711614732, 0.336340327004036) --
			(0.201863869839107, 0.338004500626574) --
			(0.195933501803052, 0.339623801424778) --
			(0.189968494594784, 0.341197670223037) --
			(0.183969754871685, 0.342725560360186) --
			(0.177938208485160, 0.344206938144388) --
			(0.171874800083528, 0.345641283301063) --
			(0.165780492693193, 0.347028089413171) --
			(0.159656267278362, 0.348366864353216) --
			(0.153503122279639, 0.349657130706266) --
			(0.147322073131839, 0.350898426183375) --
			(0.141114151761435, 0.352090304024729) --
			(0.134880406064075, 0.353232333391895) --
			(0.128621899362642, 0.354324099748540) --
			(0.122339709846382, 0.355365205229014) --
			(0.116034929991665, 0.356355268994181) --
			(0.109708665964959, 0.357293927573944) --
			(0.103362037008672, 0.358180835195871) --
			(0.0969961748105129, 0.359015664099390) --
			(0.0906122228570895, 0.359798104835032) --
			(0.0842113357724756, 0.360527866548198) --
			(0.0777946786425225, 0.361204677247000) --
			(0.0713634263257146, 0.361828284053687) --
			(0.0649187627514022, 0.362398453439253) --
			(0.0584618802062707, 0.362914971440805) --
			(0.0519939786099316, 0.363377643861329) --
			(0.0455162647805443, 0.363786296451495) --
			(0.0390299516914018, 0.364140775073193) --
			(0.0325362577194318, 0.364440945844517) --
			(0.0260364058865823, 0.364686695265932) --
			(0.0195316230950792, 0.364877930327410) --
			(0.0130231393575560, 0.365014578596350) --
			(0.00651218702306616, 0.365096588286119) --
			(1.49742520266681e-17, 0.365123928305091) --
			(-0.00651218702306613, 0.365096588286119) --
			(-0.0130231393575560, 0.365014578596350) --
			(-0.0195316230950792, 0.364877930327410) --
			(-0.0260364058865823, 0.364686695265932) --
			(-0.0325362577194318, 0.364440945844517) --
			(-0.0390299516914018, 0.364140775073193) --
			(-0.0455162647805442, 0.363786296451495) --
			(-0.0519939786099316, 0.363377643861329) --
			(-0.0584618802062707, 0.362914971440805) --
			(-0.0649187627514022, 0.362398453439253) --
			(-0.0713634263257145, 0.361828284053687) --
			(-0.0777946786425225, 0.361204677247000) --
			(-0.0842113357724756, 0.360527866548198) --
			(-0.0906122228570895, 0.359798104835032) --
			(-0.0969961748105129, 0.359015664099390) --
			(-0.103362037008672, 0.358180835195871) --
			(-0.109708665964959, 0.357293927573944) --
			(-0.116034929991665, 0.356355268994181) --
			(-0.122339709846382, 0.355365205229014) --
			(-0.128621899362642, 0.354324099748540) --
			(-0.134880406064075, 0.353232333391895) --
			(-0.141114151761435, 0.352090304024729) --
			(-0.147322073131839, 0.350898426183375) --
			(-0.153503122279639, 0.349657130706266) --
			(-0.159656267278362, 0.348366864353216) --
			(-0.165780492693193, 0.347028089413171) --
			(-0.171874800083528, 0.345641283301062) --
			(-0.177938208485160, 0.344206938144388) --
			(-0.183969754871685, 0.342725560360186) --
			(-0.189968494594784, 0.341197670223037) --
			(-0.195933501803052, 0.339623801424778) --
			(-0.201863869839107, 0.338004500626574) --
			(-0.207758711614732, 0.336340327004036) --
			(-0.213617159963875, 0.334631851786050) --
			(-0.219438367973323, 0.332879657787993) --
			(-0.225221509290970, 0.331084338940002) --
			(-0.230965778411582, 0.329246499810984) --
			(-0.236670390940032, 0.327366755129009) --
			(-0.242334583832009, 0.325445729298769) --
			(-0.247957615612248, 0.323484055916748) --
			(-0.253538766570351, 0.321482377284748) --
			(-0.259077338934321, 0.319441343922416) --
			(-0.264572657021958, 0.317361614079393) --
			(-0.270024067370295, 0.315243853247709) --
			(-0.275430938843308, 0.313088733675021) --
			(-0.280792662718130, 0.310896933879288) --
			(-0.286108652750058, 0.308669138165459) --
			(-0.291378345216663, 0.306406036144739) --
			(-0.296601198941337, 0.304108322256964) --
			(-0.301776695296637, 0.301776695296637);
			\draw[gray, rotate around={-90:(0,0)}] (0.301776695296637, 0.301776695296637) --
			(0.296601198941337, 0.304108322256964) --
			(0.291378345216663, 0.306406036144739) --
			(0.286108652750058, 0.308669138165459) --
			(0.280792662718130, 0.310896933879288) --
			(0.275430938843308, 0.313088733675021) --
			(0.270024067370295, 0.315243853247709) --
			(0.264572657021958, 0.317361614079393) --
			(0.259077338934321, 0.319441343922416) --
			(0.253538766570351, 0.321482377284748) --
			(0.247957615612248, 0.323484055916748) --
			(0.242334583832009, 0.325445729298769) --
			(0.236670390940032, 0.327366755129009) --
			(0.230965778411582, 0.329246499810984) --
			(0.225221509290970, 0.331084338940002) --
			(0.219438367973323, 0.332879657787993) --
			(0.213617159963875, 0.334631851786050) --
			(0.207758711614732, 0.336340327004036) --
			(0.201863869839107, 0.338004500626574) --
			(0.195933501803052, 0.339623801424778) --
			(0.189968494594784, 0.341197670223037) --
			(0.183969754871685, 0.342725560360186) --
			(0.177938208485160, 0.344206938144388) --
			(0.171874800083528, 0.345641283301063) --
			(0.165780492693193, 0.347028089413171) --
			(0.159656267278362, 0.348366864353216) --
			(0.153503122279639, 0.349657130706266) --
			(0.147322073131839, 0.350898426183375) --
			(0.141114151761435, 0.352090304024729) --
			(0.134880406064075, 0.353232333391895) --
			(0.128621899362642, 0.354324099748540) --
			(0.122339709846382, 0.355365205229014) --
			(0.116034929991665, 0.356355268994181) --
			(0.109708665964959, 0.357293927573944) --
			(0.103362037008672, 0.358180835195871) --
			(0.0969961748105129, 0.359015664099390) --
			(0.0906122228570895, 0.359798104835032) --
			(0.0842113357724756, 0.360527866548198) --
			(0.0777946786425225, 0.361204677247000) --
			(0.0713634263257146, 0.361828284053687) --
			(0.0649187627514022, 0.362398453439253) --
			(0.0584618802062707, 0.362914971440805) --
			(0.0519939786099316, 0.363377643861329) --
			(0.0455162647805443, 0.363786296451495) --
			(0.0390299516914018, 0.364140775073193) --
			(0.0325362577194318, 0.364440945844517) --
			(0.0260364058865823, 0.364686695265932) --
			(0.0195316230950792, 0.364877930327410) --
			(0.0130231393575560, 0.365014578596350) --
			(0.00651218702306616, 0.365096588286119) --
			(1.49742520266681e-17, 0.365123928305091) --
			(-0.00651218702306613, 0.365096588286119) --
			(-0.0130231393575560, 0.365014578596350) --
			(-0.0195316230950792, 0.364877930327410) --
			(-0.0260364058865823, 0.364686695265932) --
			(-0.0325362577194318, 0.364440945844517) --
			(-0.0390299516914018, 0.364140775073193) --
			(-0.0455162647805442, 0.363786296451495) --
			(-0.0519939786099316, 0.363377643861329) --
			(-0.0584618802062707, 0.362914971440805) --
			(-0.0649187627514022, 0.362398453439253) --
			(-0.0713634263257145, 0.361828284053687) --
			(-0.0777946786425225, 0.361204677247000) --
			(-0.0842113357724756, 0.360527866548198) --
			(-0.0906122228570895, 0.359798104835032) --
			(-0.0969961748105129, 0.359015664099390) --
			(-0.103362037008672, 0.358180835195871) --
			(-0.109708665964959, 0.357293927573944) --
			(-0.116034929991665, 0.356355268994181) --
			(-0.122339709846382, 0.355365205229014) --
			(-0.128621899362642, 0.354324099748540) --
			(-0.134880406064075, 0.353232333391895) --
			(-0.141114151761435, 0.352090304024729) --
			(-0.147322073131839, 0.350898426183375) --
			(-0.153503122279639, 0.349657130706266) --
			(-0.159656267278362, 0.348366864353216) --
			(-0.165780492693193, 0.347028089413171) --
			(-0.171874800083528, 0.345641283301062) --
			(-0.177938208485160, 0.344206938144388) --
			(-0.183969754871685, 0.342725560360186) --
			(-0.189968494594784, 0.341197670223037) --
			(-0.195933501803052, 0.339623801424778) --
			(-0.201863869839107, 0.338004500626574) --
			(-0.207758711614732, 0.336340327004036) --
			(-0.213617159963875, 0.334631851786050) --
			(-0.219438367973323, 0.332879657787993) --
			(-0.225221509290970, 0.331084338940002) --
			(-0.230965778411582, 0.329246499810984) --
			(-0.236670390940032, 0.327366755129009) --
			(-0.242334583832009, 0.325445729298769) --
			(-0.247957615612248, 0.323484055916748) --
			(-0.253538766570351, 0.321482377284748) --
			(-0.259077338934321, 0.319441343922416) --
			(-0.264572657021958, 0.317361614079393) --
			(-0.270024067370295, 0.315243853247709) --
			(-0.275430938843308, 0.313088733675021) --
			(-0.280792662718130, 0.310896933879288) --
			(-0.286108652750058, 0.308669138165459) --
			(-0.291378345216663, 0.306406036144739) --
			(-0.296601198941337, 0.304108322256964) --
			(-0.301776695296637, 0.301776695296637);
			\draw[gray, rotate around={180:(0,0)}] (0.301776695296637, 0.301776695296637) --
			(0.296601198941337, 0.304108322256964) --
			(0.291378345216663, 0.306406036144739) --
			(0.286108652750058, 0.308669138165459) --
			(0.280792662718130, 0.310896933879288) --
			(0.275430938843308, 0.313088733675021) --
			(0.270024067370295, 0.315243853247709) --
			(0.264572657021958, 0.317361614079393) --
			(0.259077338934321, 0.319441343922416) --
			(0.253538766570351, 0.321482377284748) --
			(0.247957615612248, 0.323484055916748) --
			(0.242334583832009, 0.325445729298769) --
			(0.236670390940032, 0.327366755129009) --
			(0.230965778411582, 0.329246499810984) --
			(0.225221509290970, 0.331084338940002) --
			(0.219438367973323, 0.332879657787993) --
			(0.213617159963875, 0.334631851786050) --
			(0.207758711614732, 0.336340327004036) --
			(0.201863869839107, 0.338004500626574) --
			(0.195933501803052, 0.339623801424778) --
			(0.189968494594784, 0.341197670223037) --
			(0.183969754871685, 0.342725560360186) --
			(0.177938208485160, 0.344206938144388) --
			(0.171874800083528, 0.345641283301063) --
			(0.165780492693193, 0.347028089413171) --
			(0.159656267278362, 0.348366864353216) --
			(0.153503122279639, 0.349657130706266) --
			(0.147322073131839, 0.350898426183375) --
			(0.141114151761435, 0.352090304024729) --
			(0.134880406064075, 0.353232333391895) --
			(0.128621899362642, 0.354324099748540) --
			(0.122339709846382, 0.355365205229014) --
			(0.116034929991665, 0.356355268994181) --
			(0.109708665964959, 0.357293927573944) --
			(0.103362037008672, 0.358180835195871) --
			(0.0969961748105129, 0.359015664099390) --
			(0.0906122228570895, 0.359798104835032) --
			(0.0842113357724756, 0.360527866548198) --
			(0.0777946786425225, 0.361204677247000) --
			(0.0713634263257146, 0.361828284053687) --
			(0.0649187627514022, 0.362398453439253) --
			(0.0584618802062707, 0.362914971440805) --
			(0.0519939786099316, 0.363377643861329) --
			(0.0455162647805443, 0.363786296451495) --
			(0.0390299516914018, 0.364140775073193) --
			(0.0325362577194318, 0.364440945844517) --
			(0.0260364058865823, 0.364686695265932) --
			(0.0195316230950792, 0.364877930327410) --
			(0.0130231393575560, 0.365014578596350) --
			(0.00651218702306616, 0.365096588286119) --
			(1.49742520266681e-17, 0.365123928305091) --
			(-0.00651218702306613, 0.365096588286119) --
			(-0.0130231393575560, 0.365014578596350) --
			(-0.0195316230950792, 0.364877930327410) --
			(-0.0260364058865823, 0.364686695265932) --
			(-0.0325362577194318, 0.364440945844517) --
			(-0.0390299516914018, 0.364140775073193) --
			(-0.0455162647805442, 0.363786296451495) --
			(-0.0519939786099316, 0.363377643861329) --
			(-0.0584618802062707, 0.362914971440805) --
			(-0.0649187627514022, 0.362398453439253) --
			(-0.0713634263257145, 0.361828284053687) --
			(-0.0777946786425225, 0.361204677247000) --
			(-0.0842113357724756, 0.360527866548198) --
			(-0.0906122228570895, 0.359798104835032) --
			(-0.0969961748105129, 0.359015664099390) --
			(-0.103362037008672, 0.358180835195871) --
			(-0.109708665964959, 0.357293927573944) --
			(-0.116034929991665, 0.356355268994181) --
			(-0.122339709846382, 0.355365205229014) --
			(-0.128621899362642, 0.354324099748540) --
			(-0.134880406064075, 0.353232333391895) --
			(-0.141114151761435, 0.352090304024729) --
			(-0.147322073131839, 0.350898426183375) --
			(-0.153503122279639, 0.349657130706266) --
			(-0.159656267278362, 0.348366864353216) --
			(-0.165780492693193, 0.347028089413171) --
			(-0.171874800083528, 0.345641283301062) --
			(-0.177938208485160, 0.344206938144388) --
			(-0.183969754871685, 0.342725560360186) --
			(-0.189968494594784, 0.341197670223037) --
			(-0.195933501803052, 0.339623801424778) --
			(-0.201863869839107, 0.338004500626574) --
			(-0.207758711614732, 0.336340327004036) --
			(-0.213617159963875, 0.334631851786050) --
			(-0.219438367973323, 0.332879657787993) --
			(-0.225221509290970, 0.331084338940002) --
			(-0.230965778411582, 0.329246499810984) --
			(-0.236670390940032, 0.327366755129009) --
			(-0.242334583832009, 0.325445729298769) --
			(-0.247957615612248, 0.323484055916748) --
			(-0.253538766570351, 0.321482377284748) --
			(-0.259077338934321, 0.319441343922416) --
			(-0.264572657021958, 0.317361614079393) --
			(-0.270024067370295, 0.315243853247709) --
			(-0.275430938843308, 0.313088733675021) --
			(-0.280792662718130, 0.310896933879288) --
			(-0.286108652750058, 0.308669138165459) --
			(-0.291378345216663, 0.306406036144739) --
			(-0.296601198941337, 0.304108322256964) --
			(-0.301776695296637, 0.301776695296637);
			\draw[gray, rotate around={90:(0,0)}] (0.301776695296637, 0.301776695296637) --
			(0.296601198941337, 0.304108322256964) --
			(0.291378345216663, 0.306406036144739) --
			(0.286108652750058, 0.308669138165459) --
			(0.280792662718130, 0.310896933879288) --
			(0.275430938843308, 0.313088733675021) --
			(0.270024067370295, 0.315243853247709) --
			(0.264572657021958, 0.317361614079393) --
			(0.259077338934321, 0.319441343922416) --
			(0.253538766570351, 0.321482377284748) --
			(0.247957615612248, 0.323484055916748) --
			(0.242334583832009, 0.325445729298769) --
			(0.236670390940032, 0.327366755129009) --
			(0.230965778411582, 0.329246499810984) --
			(0.225221509290970, 0.331084338940002) --
			(0.219438367973323, 0.332879657787993) --
			(0.213617159963875, 0.334631851786050) --
			(0.207758711614732, 0.336340327004036) --
			(0.201863869839107, 0.338004500626574) --
			(0.195933501803052, 0.339623801424778) --
			(0.189968494594784, 0.341197670223037) --
			(0.183969754871685, 0.342725560360186) --
			(0.177938208485160, 0.344206938144388) --
			(0.171874800083528, 0.345641283301063) --
			(0.165780492693193, 0.347028089413171) --
			(0.159656267278362, 0.348366864353216) --
			(0.153503122279639, 0.349657130706266) --
			(0.147322073131839, 0.350898426183375) --
			(0.141114151761435, 0.352090304024729) --
			(0.134880406064075, 0.353232333391895) --
			(0.128621899362642, 0.354324099748540) --
			(0.122339709846382, 0.355365205229014) --
			(0.116034929991665, 0.356355268994181) --
			(0.109708665964959, 0.357293927573944) --
			(0.103362037008672, 0.358180835195871) --
			(0.0969961748105129, 0.359015664099390) --
			(0.0906122228570895, 0.359798104835032) --
			(0.0842113357724756, 0.360527866548198) --
			(0.0777946786425225, 0.361204677247000) --
			(0.0713634263257146, 0.361828284053687) --
			(0.0649187627514022, 0.362398453439253) --
			(0.0584618802062707, 0.362914971440805) --
			(0.0519939786099316, 0.363377643861329) --
			(0.0455162647805443, 0.363786296451495) --
			(0.0390299516914018, 0.364140775073193) --
			(0.0325362577194318, 0.364440945844517) --
			(0.0260364058865823, 0.364686695265932) --
			(0.0195316230950792, 0.364877930327410) --
			(0.0130231393575560, 0.365014578596350) --
			(0.00651218702306616, 0.365096588286119) --
			(1.49742520266681e-17, 0.365123928305091) --
			(-0.00651218702306613, 0.365096588286119) --
			(-0.0130231393575560, 0.365014578596350) --
			(-0.0195316230950792, 0.364877930327410) --
			(-0.0260364058865823, 0.364686695265932) --
			(-0.0325362577194318, 0.364440945844517) --
			(-0.0390299516914018, 0.364140775073193) --
			(-0.0455162647805442, 0.363786296451495) --
			(-0.0519939786099316, 0.363377643861329) --
			(-0.0584618802062707, 0.362914971440805) --
			(-0.0649187627514022, 0.362398453439253) --
			(-0.0713634263257145, 0.361828284053687) --
			(-0.0777946786425225, 0.361204677247000) --
			(-0.0842113357724756, 0.360527866548198) --
			(-0.0906122228570895, 0.359798104835032) --
			(-0.0969961748105129, 0.359015664099390) --
			(-0.103362037008672, 0.358180835195871) --
			(-0.109708665964959, 0.357293927573944) --
			(-0.116034929991665, 0.356355268994181) --
			(-0.122339709846382, 0.355365205229014) --
			(-0.128621899362642, 0.354324099748540) --
			(-0.134880406064075, 0.353232333391895) --
			(-0.141114151761435, 0.352090304024729) --
			(-0.147322073131839, 0.350898426183375) --
			(-0.153503122279639, 0.349657130706266) --
			(-0.159656267278362, 0.348366864353216) --
			(-0.165780492693193, 0.347028089413171) --
			(-0.171874800083528, 0.345641283301062) --
			(-0.177938208485160, 0.344206938144388) --
			(-0.183969754871685, 0.342725560360186) --
			(-0.189968494594784, 0.341197670223037) --
			(-0.195933501803052, 0.339623801424778) --
			(-0.201863869839107, 0.338004500626574) --
			(-0.207758711614732, 0.336340327004036) --
			(-0.213617159963875, 0.334631851786050) --
			(-0.219438367973323, 0.332879657787993) --
			(-0.225221509290970, 0.331084338940002) --
			(-0.230965778411582, 0.329246499810984) --
			(-0.236670390940032, 0.327366755129009) --
			(-0.242334583832009, 0.325445729298769) --
			(-0.247957615612248, 0.323484055916748) --
			(-0.253538766570351, 0.321482377284748) --
			(-0.259077338934321, 0.319441343922416) --
			(-0.264572657021958, 0.317361614079393) --
			(-0.270024067370295, 0.315243853247709) --
			(-0.275430938843308, 0.313088733675021) --
			(-0.280792662718130, 0.310896933879288) --
			(-0.286108652750058, 0.308669138165459) --
			(-0.291378345216663, 0.306406036144739) --
			(-0.296601198941337, 0.304108322256964) --
			(-0.301776695296637, 0.301776695296637);
			\end{tikzpicture}
			\caption{Patches (in black) and initial elements (in gray) used for the geometrical description of $\Omega_0$.} \label{fig:patchesOmega0}
		\end{center}
	\end{subfigure}
	\caption{Numerical test \ref{sec:cvtestneg} -- Exact and defeatured domains used for the convergence analysis.}
\end{figure}
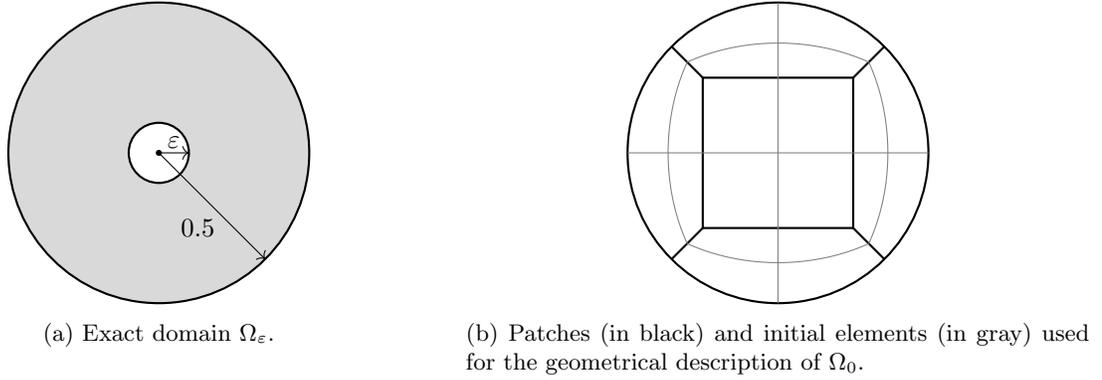

	\begin{figure}[h!]
		\begin{subfigure}{0.5\textwidth}
\begin{tikzpicture}[scale=0.85]
\begin{axis}[
xmode=log, ymode=log, zmode=log, grid=both, view={-20}{20}, xlabel = $\varepsilon$, ylabel = $h$, zlabel = Error/Estimator, ymin=0.002, ymax=0.15,colormap/RdYlBu]
\addplot3[surf,
] file {
	data/test_fourier_data.txt
};
\addplot3[surf,
] file {
	data/test_fourier_data2.txt
};
\end{axis}
\end{tikzpicture}
\caption{A view on the error (lower surface) and \\estimator (upper surface).}\label{fig:a}
\end{subfigure}
 ~
\begin{subfigure}{0.5\textwidth}
	\begin{tikzpicture}[scale=0.85]
	\begin{axis}[
	xmode=log, ymode=log, zmode=log, grid=both, view={-20}{-20}, xlabel = $\varepsilon$, ylabel = $h$, zlabel = Error/Estimator, ymin=0.002, ymax=0.15,colormap/RdYlBu]
	\addplot3[surf,
	] file {
		data/test_fourier_data.txt
	};
	\addplot3[surf,
	] file {
		data/test_fourier_data2.txt
	};
	\end{axis}
	\end{tikzpicture}
	\caption{Another view on the error (lower surface) and estimator (upper surface).}\label{fig:b}
\end{subfigure}
~
	\begin{subfigure}{0.48\textwidth}
		\vspace{0.15cm}
		\begin{center}
			\begin{tikzpicture}[scale=0.85]
			\begin{axis}[xmode=log, ymode=log, legend style={at={(0.5,-0.17)}, legend columns=3, anchor=north, draw=none}, xlabel=$h$, ymax=8e-2, ymin=3e-7, ylabel=Error/Estimator, grid style=dotted,grid]
			\addlegendimage{black,dashed, thick, mark options=solid}
			\addlegendentry{E}
			\addlegendimage{black,thick}
			\addlegendentry{E}
			\addplot[mark=none, black, densely dotted, thick] table [x=h, y=h2, col sep=comma] {data/test_fourier_wrt_h.csv};
			\addplot[mark=o, c4, thick] table [x=h, y=errsmall, col sep=comma] {data/test_fourier_wrt_h.csv};
			\addplot[mark=square, c, thick] table [x=h, y=errintermediate, col sep=comma] {data/test_fourier_wrt_h.csv};
			\addplot[mark=triangle, c2, thick, mark options={solid,scale=1.5}] table [x=h, y=errlarge, col sep=comma] {data/test_fourier_wrt_h.csv};
			\addplot[mark=*, c4, dashed, thick, mark options={solid,fill=c4}] table [x=h, y=estsmall, col sep=comma] {data/test_fourier_wrt_h.csv};
			\addplot[mark=square*, c, dashed, thick, mark options=solid] table [x=h, y=estintermediate, col sep=comma] {data/test_fourier_wrt_h.csv};
			\addplot[mark=triangle*, c2, dashed, thick, mark options={solid,scale=1.5}] table [x=h, y=estlarge, col sep=comma] {data/test_fourier_wrt_h.csv};
 			\legend{$\mathscr E\big(u_\mathrm d^h\big)$, $\left\|\nabla\left(u-u_\mathrm d^h\right)\right\|_{0,\Omega_\varepsilon}$, $h^2$};
			\end{axis}
			\end{tikzpicture}
			\caption{In {\color{c4}blue circles}, $\varepsilon = 1.5625\cdot 10^{-4}$; in {\color{c}gray squares}, $\varepsilon = 2.5\cdot 10^{-3}$; in {\color{c2}orange triangles}, $\varepsilon = 8\cdot 10^{-2}$.} \label{fig:c}
		\end{center}
	\end{subfigure}
	~
	\begin{subfigure}{0.48\textwidth}
		\begin{center}
			\vspace{0.15cm}
			\begin{tikzpicture}[scale=0.85]
			\begin{axis}[xmode=log, ymode=log, legend style={at={(0.5,-0.17)}, legend columns=3, anchor=north, draw=none}, xlabel=$\varepsilon$, ymax=8e-2, ymin=3e-7, ylabel=Error/Estimator, grid style=dotted,grid]
			\addlegendimage{black,dashed, thick, mark options=solid}
			\addlegendentry{E}
			\addlegendimage{black,thick}
			\addlegendentry{E}
			\addplot[mark=none, densely dotted, black, thick] table [x=eps, y=eps2, col sep=comma] {data/test_fourier_wrt_eps_bis.csv};
			\addplot[mark=*, c4, dashed, thick, mark options={solid,fill=c4}] table [x=eps, y=errsmall, col sep=comma] {data/test_fourier_wrt_eps.csv};
			\addplot[mark=square*, c, dashed, thick, mark options=solid] table [x=eps, y=errintermediate, col sep=comma] {data/test_fourier_wrt_eps.csv};
			\addplot[mark=triangle*, c2, dashed, thick, mark options={solid,scale=1.5}] table [x=eps, y=errlarge, col sep=comma] {data/test_fourier_wrt_eps.csv};
			\addplot[mark=o, c4, thick] table [x=eps, y=estsmall, col sep=comma] {data/test_fourier_wrt_eps.csv};
			\addplot[mark=square, c, thick] table [x=eps, y=estintermediate, col sep=comma] {data/test_fourier_wrt_eps.csv};
			\addplot[mark=triangle, c2, thick, mark options={solid,scale=1.5}] table [x=eps, y=estlarge, col sep=comma] {data/test_fourier_wrt_eps.csv};
			\legend{$\mathscr E\big(u_\mathrm d^h\big)$, $\left\|\nabla\left(u-u_\mathrm d^h\right)\right\|_{0,\Omega_\varepsilon}$,$\varepsilon^2 \left|\log{\varepsilon}\right|^\frac{1}{2}$};
			\end{axis}
			\end{tikzpicture}
			\caption{In {\color{c4}blue circles}, $h = 2.7621\cdot 10^{-3}$; in {\color{c}gray squares}, $h = 1.1049\cdot 10^{-2}$; in {\color{c2}orange triangles}, $h = 4.4194\cdot 10^{-2}$.} \label{fig:d}
		\end{center}
	\end{subfigure}
	\caption{Numerical test~\ref{sec:cvtestneg} -- Convergence of the discrete defeaturing error and estimator with respect to the mesh size $h$ under global $h$-refinement, and with respect to the feature size $\varepsilon$. In~(\ref{sub@fig:a}) and~(\ref{sub@fig:b}) are two different views on the surface error and on the surface estimator (the first one being below the second one). In~(\ref{sub@fig:c}), convergence with respect to $h$ for three fixed values of $\varepsilon$. In~(\ref{sub@fig:d}), convergence with respect to $\varepsilon$ for three fixed values of $h$.} \label{fig:doublecv}
\end{figure}
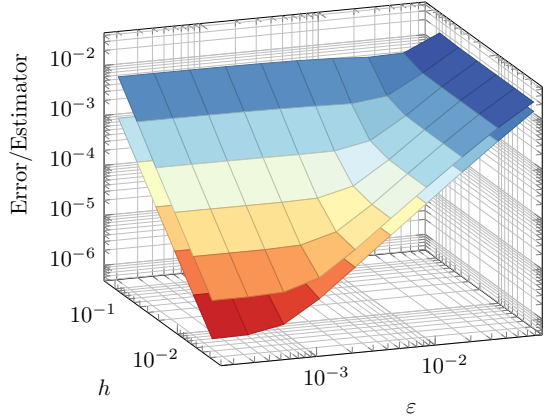
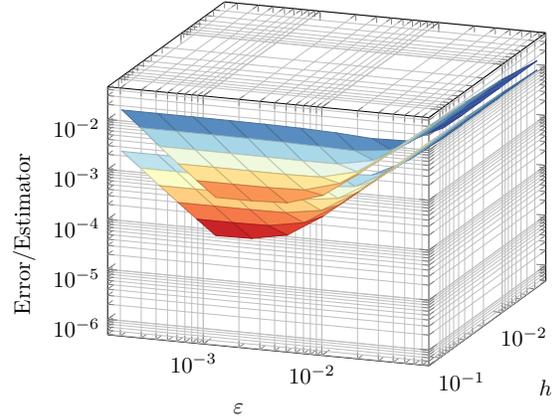
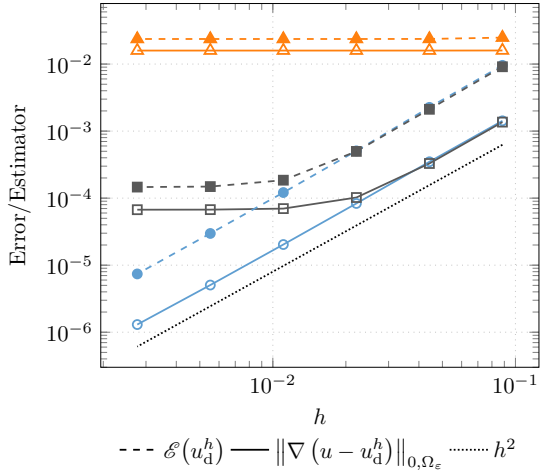
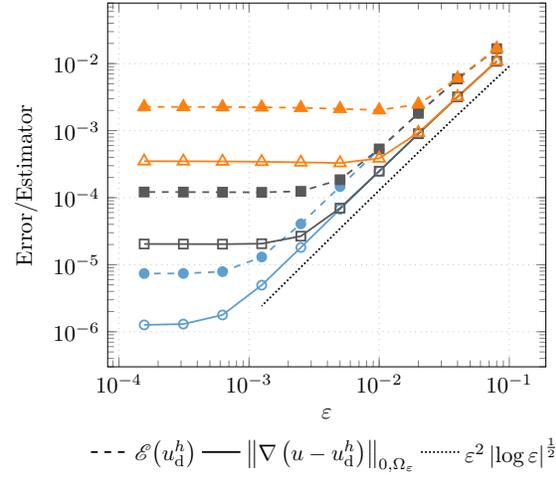

We consider Poisson's problem~(\ref{eq:originalpb}) solved in $\Omega_\varepsilon$, we take $f \equiv -1$ in $\Omega_\varepsilon$, $g_D \equiv 0$ on $\Gamma_D := \partial \Omega_0$, and $g\equiv 0$ on $\Gamma_N = \gamma := \partial F_\varepsilon = \partial \Omega_\varepsilon \setminus \partial \Omega_0$. The exact solution of problem~(\ref{eq:originalpb}) in $\Omega_\varepsilon$ is given by 
$$u(x,y) = \frac{\varepsilon^2}{2}\left[\log\left(\frac{1}{2}\right)-\log\left(\sqrt{x^2+y^2}\right)\right] + \frac{x^2+y^2}{4} -\frac{1}{16}.$$
Then, we solve the defeatured Poisson's problem~(\ref{eq:simplpb}) in $\Omega_0$, where $f$ is extended by $-1$ in $F_\varepsilon = \Omega_0\setminus \overline{\Omega_\varepsilon}$. The exact defeatured solution $u_0$ is given in $\Omega_0$ by
$$u_0(x,y) = \frac{x^2+y^2}{4}- \frac{1}{16}.$$
So in particular, the exact defeaturing error, i.e., without discretization error, is given by
\begin{equation} \label{eq:exactdefeaterror}
\left\| \nabla\left(u-u_0\right) \right\|_{0,\Omega} = \sqrt{\frac{\pi}{2}} \,\varepsilon^2 \left[ \log\left(\frac{1}{2}\right)-\log\left(\varepsilon\right)\right]^\frac{1}{2} \sim \varepsilon^2 \left|\log(\varepsilon)\right|^\frac{1}{2}.
\end{equation}

Now, $\Omega_0$ is divided into five conforming patches as illustrated in Figure~\ref{fig:patchesOmega0}, and we consider $4^j$ elements in each patch, for $j=0,\ldots,5$. Thus in this experiment, $h=\displaystyle\frac{\sqrt{2}}{2^{j+1}}$, and we solve the Galerkin formulation of the defeatured problem~(\ref{eq:weaksimplpb}) using B-spline based IGA. 
Therefore, we look at the convergence of the discrete defeaturing error and estimator with respect to the size of the mesh, subject to global dyadic refinement, and with respect to the size of the feature~$\varepsilon$. Note that we never need to numerically create and mesh the exact geometry $\Omega_\varepsilon$. Indeed, we do not perform adaptivity in this experiment, thus $\Omega_0$ remains the geometrical model in which the PDE is solved, and for the computation of the error, the exact solution $u$ is already known.

The results are presented in Figure~\ref{fig:doublecv}. When $\varepsilon$ is fixed and small, the numerical component of the error dominates over its defeaturing component, and thus the overall error converges as $h^p= h^2$, as expected. When $\varepsilon$ is fixed and large, the defeaturing error dominates and thus the overall error does not converge with respect to $h$, and we observe a plateau. Similarly, when $h$ is fixed and small, the numerical component of the error is negligible with respect to the defeaturing component, and thus the overall error converges as $\varepsilon^2 \left|\log(\varepsilon)\right|^\frac{1}{2}$, also as expected from~(\ref{eq:exactdefeaterror}). But when $h$ is fixed and large, the numerical error dominates and thus the overall error does not converge with respect to $\varepsilon$, and we observe a plateau. 

The exact same behavior and convergence rates are observed for the discrete defeaturing error estimator, confirming its reliability proven in Theorem~\ref{thm:uppernegtoterror}, and showing also its efficiency. The effectivity index in the numerical-error-dominant regime is larger (around~$5.8$) than the one in the defeaturing-error-dominant regime (around~$1.5$), \changes{which is consistent with previous results in the literature} (see e.g., \cite{paper1defeaturing,paper3multifeature}). Finally, the change of behavior between the two different regimes happens when $\varepsilon \approx h$, both for the overall error and for the proposed estimator. 

\subsubsection{Positive feature} \label{sec:cvtestpos}
In this experiment, we consider a geometry with one positive feature.
That is, for $k=-5,-4,\ldots,4$, let $\varepsilon =\displaystyle \frac{10^{-2}}{2^k}$ and let $\Omega_0 := \Omega_\varepsilon \setminus \overline{F_\varepsilon}$, where $\Omega_0$ is an L-shaped domain 
and $F_\varepsilon$ is a fillet of radius $\varepsilon$ rounding the corner of the L-shaped domain, as illustrated in~Figure \ref{fig:Lshapefillet}. 
Or in other words, $\Omega_0$ is the defeatured geometry obtained from $\Omega_\varepsilon$ by removing the positive feature $F_\varepsilon$. 

\begin{figure}
	\centering
\savebox{\imagebox}{
			\begin{tikzpicture}[scale=3.5]
			\draw (-0.1,0) -- (-0.1,1);
			\draw (-0.08,0) -- (-0.12,0);
			\draw (-0.08,1) -- (-0.12,1);
			\draw (-0.1,0.5) node[left]{$1$};
			\draw (1.1,0) -- (1.1,0.5);
			\draw (1.08,0) -- (1.12,0);
			\draw (1.08,0.5) -- (1.12,0.5);
			\draw (1.1,0.25) node[right]{$0.5$};
			\draw[thick,c4] (1,0) -- (0,0) -- (0,1);
			\draw[c4] (0.5,0) node[below]{$\Gamma_D$};
			\draw[thick] (0,1) -- (0.5, 1) -- (0.5,0.7);
			\draw[thick] (0.7,0.5) -- (1,0.5) -- (1,0);
			\draw[thick, c2, domain=-180:-90] plot ({0.7+0.2*cos(\x)}, {0.7+0.2*sin(\x)});
			\draw (0.35,0.35) node{$\Omega_\varepsilon$};
			\draw[c2] (0.35,0.6) node{$\gamma=\gamma_\setminussign$};
			\draw (0,-0.2) -- (1,-0.2);
			\draw (0,-0.18) -- (0,-0.22);
			\draw (1,-0.18) -- (1,-0.22);
			\draw (0.5,-0.2) node[below]{$1$};
			\draw (0,1.1) -- (0.5,1.1);
			\draw (0,1.08) -- (0,1.12);
			\draw (0.5,1.08) -- (0.5,1.12);
			\draw (0.25,1.1) node[above]{$0.5$};
			\draw[fill] (0,0) circle (0.015cm);
			\draw[] (0.05,0) node[below]{\footnotesize $(0,0)^T$};
			\draw[->] (0.7,0.7) -- (0.5585786,0.5585786);
			\draw (0.66,0.6) node{$\varepsilon$}; 
			\draw[fill] (0.7,0.7) circle (0.015cm);
			\draw[] (0.73,0.7) node[above]{\footnotesize $\left(\begin{matrix}0.5+\varepsilon\\0.5+\varepsilon\end{matrix}\right)$};
			\end{tikzpicture}}
			\begin{subfigure}[t]{0.39\textwidth}
				\centering\usebox{\imagebox}
	\caption{Exact domain $\Omega_\varepsilon$.}
	\end{subfigure}
	~
	\begin{subfigure}[t]{0.28\textwidth}
		\centering\raisebox{\dimexpr.5\ht\imagebox-.5\height}{
			\begin{tikzpicture}[scale=3.5]
			\draw[thick] (0,0) -- (0,1) -- (0.5, 1) -- (0.5,0.5) -- (1,0.5) -- (1,0) -- (0, 0);
			\draw [thick,domain=-180:-90] plot ({0.7+0.2*cos(\x)}, {0.7+0.2*sin(\x)});
			\draw[thick] (0.5,0.7) -- (0.7,0.7) -- (0.7,0.5);
			\draw[->] (0.59,0.43) -- (0.54,0.54);
			\draw (0.35,0.35) node{$\Omega_0$};
			\draw (0.61,0.37) node{$F_\varepsilon$};
			\draw (0.62,0.62) node{$G_\varepsilon$};
			\end{tikzpicture}}
	\caption{Simplified domain $\Omega_0$, feature $F_\varepsilon$ and feature extension $G_\varepsilon$.}
	\end{subfigure}
	~
	\begin{subfigure}[t]{0.28\textwidth}
		\centering\raisebox{\dimexpr.5\ht\imagebox-.5\height}{
			\begin{tikzpicture}[scale=3.5]
			\draw[thick] (0,0) -- (0,1) -- (0.5, 1) -- (0.5,0.5) -- (1,0.5) -- (1,0) -- (0, 0);
			\draw[thick,c2] (0.5,0.7) -- (0.7,0.7) -- (0.7,0.5);
			\draw[c2] (0.7,0.6) node[right]{$\tilde \gamma$};
			\draw[thick,c4] (0.7,0.5) -- (0.5,0.5) -- (0.5,0.7);
			\draw[c4] (0.6,0.5) node[below]{$\gamma_{0}$};
			\draw[gray,thick] (0,0) -- (0.5,0.5);
			\draw[gray,thick] (0,0.7) -- (0.5,0.7);
			\draw[gray,thick] (0.7,0) -- (0.7,0.5);
			\draw (0.6,0.6) node{$\tilde F_\varepsilon$};
			\end{tikzpicture}}
		\caption{Four patches of $\Omega_0$, and extended feature $\tilde F_\varepsilon$ as a single patch.} \label{fig:patchesfillet}
	\end{subfigure}
	\caption{Numerical test \ref{sec:cvtestpos} -- L-shaped domain with a positive fillet feature and its extension.}\label{fig:Lshapefillet}
\end{figure}
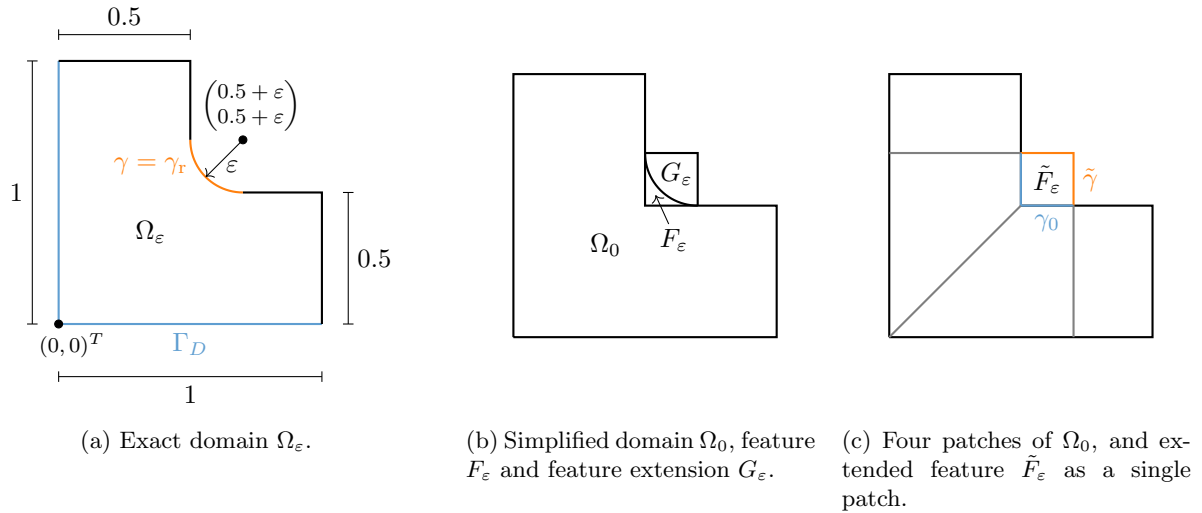

We consider Poisson's problem~(\ref{eq:originalpb}) solved in $\Omega_\varepsilon$ with 
$$\Gamma_D := \big( (0,1)\times \{0\} \big) \cup \big( \{0\} \times (0,1)\big)$$
as illustrated in Figure~\ref{fig:Lshapefillet}, and $\Gamma_N := \partial \Omega_\varepsilon\setminus \overline{\Gamma_D}$. We choose the data $f$, $g$ and $g_D$ such that the exact solution is given by $u(x,y) = \sin(2\pi x)\sin(2\pi y)$ in $\Omega_\varepsilon$. Then, we solve the defeatured Poisson problem~(\ref{eq:simplpb}) in $\Omega_0$, where $g_0$ is chosen on $\gamma_0$ so that the exact defeatured solution is also given by $u_0(x,y) = \sin(2\pi x)\sin(2\pi y)$ in $\Omega_0$. That is, we do not introduce any error coming from defeaturing in $\Omega_0$. However, we solve the extension problem~(\ref{eq:featurepb}) in $\tilde F_\varepsilon$, the bounding box of $F_\varepsilon$, choosing $\tilde g\equiv 0$ on $\tilde \gamma$, and naturally extending $f$ in $\tilde F_\varepsilon$. That is, we consider $f(x,y) = 8\pi^2\sin(2\pi x)\sin(2\pi y)$ in $\Omega_0$ and in $\tilde F_\varepsilon$. Therefore, the defeaturing component of the overall error comes from the bad choice of Neumann boundary condition $\tilde g$ in $\tilde \gamma$, which does not correspond to the exact solution $u$.

Then, $\Omega_0$ is divided into four conforming patches as illustrated in Figure~\ref{fig:patchesfillet}, while $\tilde F_\mathrm p$ is a single patch. Moreover, we consider $4^j$ elements in each patch, for $j=2,\ldots,7$, so that in this experiment, $h=\displaystyle\frac{\sqrt{2}}{2^{j+1}}$. We solve the Galerkin formulation of problems~(\ref{eq:weaksimplpb}) and~(\ref{eq:weakfeaturepb}), respectively corresponding to the defeatured and extension problems, using B-spline based IGA. 
Therefore, we look at the convergence of the discrete defeaturing error and estimator both under global dyadic refinement, and with respect to the size of the feature~$\varepsilon$. 

	\begin{figure}[h!]
	\begin{subfigure}{0.5\textwidth}
		\begin{tikzpicture}[scale=0.85]
		\begin{axis}[
		xmode=log, ymode=log, zmode=log, grid=both, view={-20}{20}, xlabel = $\varepsilon$, ylabel = $h$, zlabel = Error/Estimator, colormap/RdYlBu] 
		\addplot3[surf,
		] file {
			data/test_fillet_data2.txt
		};
		\addplot3[surf,
		] file {
			data/test_fillet_data.txt
		};
		\end{axis}
		\end{tikzpicture}
		\caption{A view on the error (lower surface) and \\estimator (upper surface).}\label{fig:afillet}
	\end{subfigure}
	~
	\begin{subfigure}{0.5\textwidth}
		\begin{tikzpicture}[scale=0.85]
		\begin{axis}[
		xmode=log, ymode=log, zmode=log, grid=both, view={-20}{-20}, xlabel = $\varepsilon$, ylabel = $h$, zlabel = Error/Estimator, colormap/RdYlBu] 
		\addplot3[surf,
		] file {
			data/test_fillet_data2.txt
		};
		\addplot3[surf,
		] file {
			data/test_fillet_data.txt
		};
		\end{axis}
		\end{tikzpicture}
		\caption{Another view on the error (lower surface) and estimator (upper surface).}\label{fig:bfillet}
	\end{subfigure}
	~
	\begin{subfigure}{0.48\textwidth}
		\vspace{0.15cm}
		\begin{center}
			\begin{tikzpicture}[scale=0.85]
			\begin{axis}[xmode=log, ymode=log, legend style={at={(0.5,-0.17)}, legend columns=3, anchor=north, draw=none}, xlabel=$h$, ymax=10, ymin=3e-5, ylabel=Error/Estimator, grid style=dotted,grid]
			\addlegendimage{black,dashed, thick, mark options=solid}
			\addlegendentry{E}
			\addlegendimage{black,thick}
			\addlegendentry{E}
			\addplot[mark=none, black, densely dotted, thick] table [x=h, y=h2, col sep=comma] {data/test_fillet_wrt_h.csv};
			\addplot[mark=o, c4, thick] table [x=h, y=errsmall, col sep=comma] {data/test_fillet_wrt_h.csv};
			\addplot[mark=square, c, thick] table [x=h, y=errintermediate, col sep=comma] {data/test_fillet_wrt_h.csv};
			\addplot[mark=triangle, c2, thick, mark options={solid,scale=1.5}] table [x=h, y=errlarge, col sep=comma] {data/test_fillet_wrt_h.csv};
			\addplot[mark=*, c4, dashed, thick, mark options={solid,fill=c4}] table [x=h, y=estsmall, col sep=comma] {data/test_fillet_wrt_h.csv};
			\addplot[mark=square*, c, dashed, thick, mark options=solid] table [x=h, y=estintermediate, col sep=comma] {data/test_fillet_wrt_h.csv};
			\addplot[mark=triangle*, c2, dashed, thick, mark options={solid,scale=1.5}] table [x=h, y=estlarge, col sep=comma] {data/test_fillet_wrt_h.csv};
			\legend{$\mathscr E\big(u_\mathrm d^h\big)$, $\left\|\nabla\left(u-u_\mathrm d^h\right)\right\|_{0,\Omega_\varepsilon}$, $h^2$};
			\end{axis}
			\end{tikzpicture}
			\caption{In {\color{c4}blue circles}, $\varepsilon = 6.25\cdot 10^{-4}$; in {\color{c}gray squares}, $\varepsilon = 2\cdot 10^{-2}$; in {\color{c2}orange triangles}, \hbox{$\varepsilon = 1.6\cdot 10^{-1}$}.} \label{fig:cfillet}
		\end{center}
	\end{subfigure}
	~
	\begin{subfigure}{0.48\textwidth}
		\begin{center}
			\vspace{0.15cm}
			\begin{tikzpicture}[scale=0.85]
			\begin{axis}[xmode=log, ymode=log, legend style={at={(0.5,-0.17)}, legend columns=3, anchor=north, draw=none}, xlabel=$\varepsilon$, ymax=10, ymin=3e-5, ylabel=Error/Estimator, grid style=dotted,grid]
			\addlegendimage{black,dashed, thick, mark options=solid}
			\addlegendentry{E}
			\addlegendimage{black,thick}
			\addlegendentry{E}
			\addplot[mark=none, densely dotted, black, thick] table [x=eps, y=eps2, col sep=comma] {data/test_fillet_wrt_eps_bis.csv};
			\addplot[mark=*, c4, dashed, thick, mark options={solid,fill=c4}] table [x=eps, y=errsmall, col sep=comma] {data/test_fillet_wrt_eps.csv};
			\addplot[mark=square*, c, dashed, thick, mark options=solid] table [x=eps, y=errintermediate, col sep=comma] {data/test_fillet_wrt_eps.csv};
			\addplot[mark=triangle*, c2, dashed, thick, mark options={solid,scale=1.5}] table [x=eps, y=errlarge, col sep=comma] {data/test_fillet_wrt_eps.csv};
			\addplot[mark=o, c4, thick] table [x=eps, y=estsmall, col sep=comma] {data/test_fillet_wrt_eps.csv};
			\addplot[mark=square, c, thick] table [x=eps, y=estintermediate, col sep=comma] {data/test_fillet_wrt_eps.csv};
			\addplot[mark=triangle, c2, thick, mark options={solid,scale=1.5}] table [x=eps, y=estlarge, col sep=comma] {data/test_fillet_wrt_eps.csv};
			\legend{$\mathscr E\big(u_\mathrm d^h\big)$, $\left\|\nabla\left(u-u_\mathrm d^h\right)\right\|_{0,\Omega_\varepsilon}$,$\varepsilon^2 \left|\log{\varepsilon}\right|^\frac{1}{2}$};
			\end{axis}
			\end{tikzpicture}
			\caption{In {\color{c4}blue circles}, $h = 5.5243\cdot 10^{-3}$; in {\color{c}gray squares}, $h = 2.2097\cdot 10^{-2}$; in {\color{c2}orange triangles}, $h = 1.7678\cdot 10^{-1}$.} \label{fig:dfillet}
		\end{center}
	\end{subfigure}
	\caption{Numerical test \ref{sec:cvtestpos} -- Convergence of the discrete defeaturing error and estimator with respect to the mesh size $h$ under global $h$-refinement, and with respect to the feature size $\varepsilon$. In~(\ref{sub@fig:afillet}) and~(\ref{sub@fig:bfillet}) are two different views on the surface error and on the surface estimator (the first one being below the second one). In~(\ref{sub@fig:cfillet}), convergence with respect to $h$ for three fixed values of $\varepsilon$. In~(\ref{sub@fig:dfillet}), convergence with respect to $\varepsilon$ for three fixed values of $h$.}  \label{fig:doublecvfillet}
\end{figure}
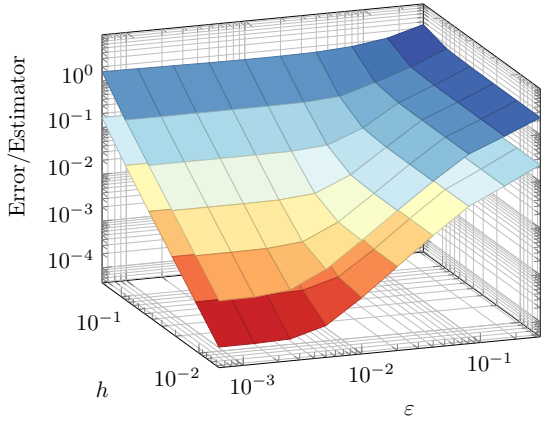
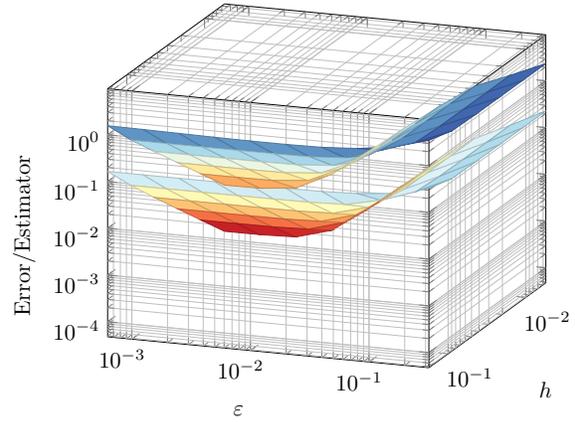
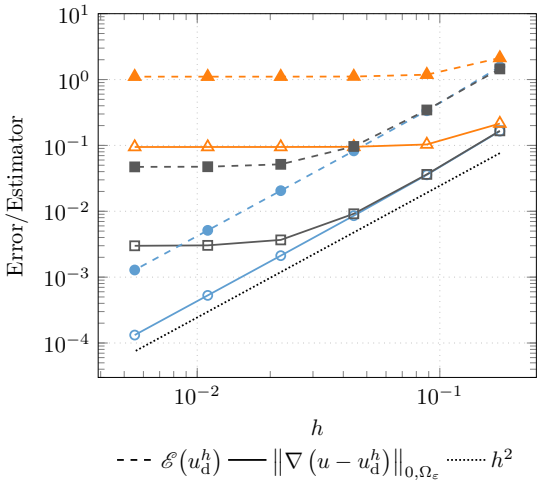
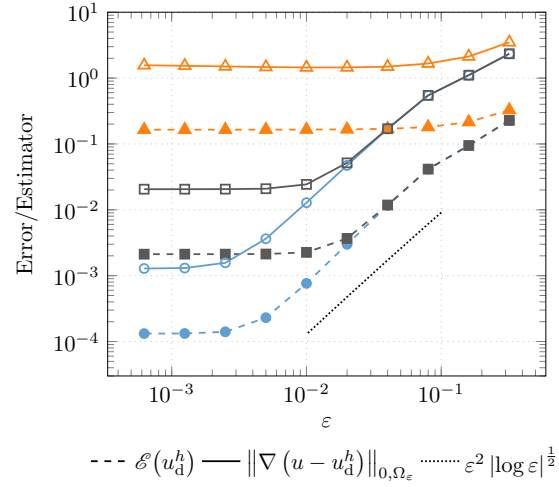

The results are presented in Figure~\ref{fig:doublecvfillet}. As for a negative feature, when $\varepsilon$ is fixed and small, the numerical component of the error dominates over its defeaturing component, and thus the overall error converges as $h^p=h^2$, as expected. When $\varepsilon$ is fixed and large, the defeaturing component of the error dominates and thus the overall error does not converge with respect to $h$, and we observe a plateau. Similarly, when $h$ is fixed and small, the numerical error is negligible with respect to the defeaturing component of the error, and thus the overall error converges as $\varepsilon^2 \left|\log(\varepsilon)\right|^\frac{1}{2}$, also as expected from the previous numerical experiment and from \cite{paper1defeaturing}. But again, when $h$ is fixed and large, the numerical error dominates and thus the overall error does not converge with respect to $\varepsilon$, and we observe a plateau. 

The overall estimator $\mathscr{E}\big(u_\mathrm d^h\big)$ follows the exact same behavior and convergence rates as the error. This numerical test shows that the reliability of the estimator proven in Theorem~\ref{thm:uppercomplextoterror} can be extended to non-Lipschitz features such as $F_\varepsilon$, and it also shows the efficiency of the estimator.
In this case, the effectivity index in the numerical-error-dominant regime is slightly larger (around~$10.7$) than the one in the defeaturing-error-dominant regime (around~$9.7$), \changes{which is consistent with previous results in the literature.} Indeed, it is observed in \cite{paper1defeaturing} that the effectivity index coming from the defeaturing component of the error estimator is larger in the case of extended positive features, as the extension $G_\varepsilon$ of $F_\varepsilon$ can itself be seen as a negative feature whose simplified geometry is $\tilde F_\varepsilon$.

\subsection{Convergence of the adaptive strategy} \label{sec:adaptstrategytests}
In the following experiments, we analyze the convergence of the adaptive strategy proposed in Section~\ref{sec:adaptive} and specialized to the isogeometric framework. The analysis is first performed in a geometry containing a negative feature, then in a geometry containing a positive feature. Moreover, we compare the proposed strategy with the standard adaptive algorithm which only performs mesh refinement, and which does not consider the defeaturing error contribution. The latter algorithm is indeed widely used nowadays, because of the lack of a sound discrete defeaturing error estimator as the one proposed in this work. 

\subsubsection{Negative feature}\label{sec:halfdiscwithhole}
Let us first consider a half disc with a circular hole, i.e., a geometry containing a negative feature. More precisely, and as illustrated in Figure~\ref{fig:halfdiscgeom}, let 
\begin{align*}
\Omega_0 &:= \left\{ \boldsymbol{x}=(x,y)^T\in\mathbb R^2 : \|\boldsymbol x\|_{\ell^2}<\frac{1}{2}, \,y<0 \right\}, \\ 
F &:= \left\{ \boldsymbol{x}=(x,y)^T\in\mathbb R^2 : \|\boldsymbol x\|_{\ell^2}<5\cdot 10^{-3}, \,y<0 \right\}, \\ 
\Omega &:= \Omega_0 \setminus \overline{F}.
\end{align*}
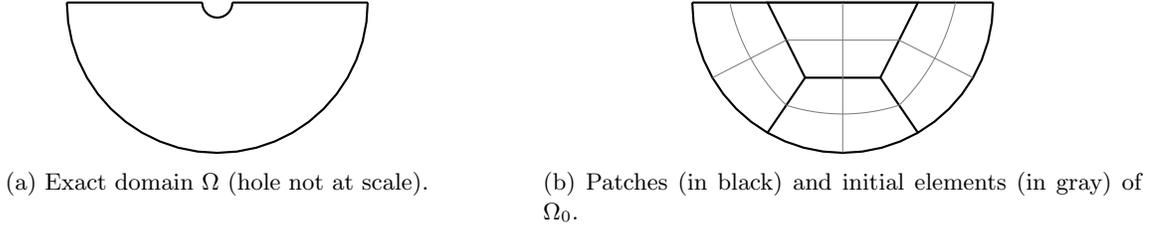
\begin{figure}
	\centering
	\begin{subfigure}[t]{0.48\textwidth}
		\begin{center}
			\begin{tikzpicture}[scale=2]
			\draw [black,thick,domain=-180:0] plot ({cos(\x)}, {sin(\x)});
			\draw [black,thick,domain=-180:0] plot ({0.1*cos(\x)}, {0.1*sin(\x)});
			\draw[black, thick] (0.1, 0) -- (1,0);
			\draw[black, thick] (-0.1, 0) -- (-1,0);
			\end{tikzpicture}
			\caption{Exact domain $\Omega$ (hole not at scale).}
		\end{center}
	\end{subfigure}
	~
	\begin{subfigure}[t]{0.48\textwidth}
		\begin{center}
			\begin{tikzpicture}[scale=4]
			\draw [black,thick,domain=-180:0] plot ({0.5*cos(\x)}, {0.5*sin(\x)});
			\draw[thick] (-0.125,-0.25) -- (-0.25,0) -- (0.25,0) -- (0.125,-0.25) -- cycle;
			\draw[thick] (-0.125,-0.25) -- (-0.25, -0.433012701892219);
			\draw[thick] (0.125,-0.25) -- (0.25, -0.433012701892219);
			\draw[black,thick] (-0.5,0) -- (0.5,0);
			\draw[gray] (0,-0.5) -- (0,0);
			\draw[gray] (-0.1875, -0.125) -- (0.1875, -0.125);
			\draw[gray] (-0.1875, -0.125) -- (-0.433012701892219, -0.25);
			\draw[gray] (0.1875, -0.125) -- (0.433012701892219, -0.25);
			\draw[gray] (-0.375000000000000,	2.77555756156289e-17) --
			(-0.373333166757099,	-0.00765962670967697) --
			(-0.371567082743171,	-0.0153307349990486) --
			(-0.369701573631691, -0.0230108257722329) --
			(-0.367736530577345,	-0.0306973685023467) --
			(-0.365671911152847,	-0.0383878037000088) --
			(-0.363507740186621,	-0.0460795454670884) --
			(-0.361244110497723,	-0.0537699841310397) --
			(-0.358881183524584,	-0.0614564889548132) --
			(-0.356419189844499,	-0.0691364109170090) --
			(-0.353858429581006,	-0.0768070855566216) --
			(-0.351199272696692,	-0.0844658358764268) --
			(-0.348442159169221,	-0.0921099752987869) --
			(-0.345587599048781,	-0.0997368106673905) --
			(-0.342636172395458,	-0.107343645288212) --
			(-0.339588529095453,	-0.114927782002764) --
			(-0.336445388555389,	-0.122486526286538) --
			(-0.333207539274400,	-0.130017189365354) --
			(-0.329875838294021,	-0.137517091342252) --
			(-0.326451210526340,	-0.144983564327397) --
			(-0.322934647961253,	-0.152413955563473) --
			(-0.319327208754045,	-0.159805630538941) --
			(-0.315630016194951,	-0.167155976081567) --
			(-0.311844257562695,	-0.174462403424636) --
			(-0.307971182864434,	-0.181722351238342) --
			(-0.304012103464865,	-0.188933288618890) --
			(-0.299968390607657,	-0.196092718028031) --
			(-0.295841473832697,	-0.203198178175837) --
			(-0.291632839292987,	-0.210247246839739) --
			(-0.287344027975343,	-0.217237543613062) --
			(-0.282976633829376,	-0.224166732576502) --
			(-0.278532301809477,	-0.231032524886277) --
			(-0.274012725834846,	-0.237832681272969) --
			(-0.269419646672790,	-0.244565014445375) --
			(-0.264754849750806,	-0.251227391394026) --
			(-0.260020162903074,	-0.257817735589403) --
			(-0.255217454057250,	-0.264334029070209) --
			(-0.250348628867510,	-0.270774314417501) --
			(-0.245415628299982,	-0.277136696610834) --
			(-0.240420426176750,	-0.283419344763020) --
			(-0.235365026684717,	-0.289620493730513) --
			(-0.230251461855617,	-0.295738445596854) --
			(-0.225081789023525,	-0.301771571027062) --
			(-0.219858088266156,	-0.307718310491286) --
			(-0.214582459836243,	-0.313577175356453) --
			(-0.209257021589206,	-0.319346748845124) --
			(-0.203883906413233,	-0.325025686861150) --
			(-0.198465259667801,	-0.330612718682190) --
			(-0.193003236636513,	-0.336106647519535) --
			(-0.187500000000000,	-0.341506350946110);
			\begin{scope}[yscale=1,xscale=-1]
			\draw[gray] (-0.375000000000000,	2.77555756156289e-17) --
			(-0.373333166757099,	-0.00765962670967697) --
			(-0.371567082743171,	-0.0153307349990486) --
			(-0.369701573631691, -0.0230108257722329) --
			(-0.367736530577345,	-0.0306973685023467) --
			(-0.365671911152847,	-0.0383878037000088) --
			(-0.363507740186621,	-0.0460795454670884) --
			(-0.361244110497723,	-0.0537699841310397) --
			(-0.358881183524584,	-0.0614564889548132) --
			(-0.356419189844499,	-0.0691364109170090) --
			(-0.353858429581006,	-0.0768070855566216) --
			(-0.351199272696692,	-0.0844658358764268) --
			(-0.348442159169221,	-0.0921099752987869) --
			(-0.345587599048781,	-0.0997368106673905) --
			(-0.342636172395458,	-0.107343645288212) --
			(-0.339588529095453,	-0.114927782002764) --
			(-0.336445388555389,	-0.122486526286538) --
			(-0.333207539274400,	-0.130017189365354) --
			(-0.329875838294021,	-0.137517091342252) --
			(-0.326451210526340,	-0.144983564327397) --
			(-0.322934647961253,	-0.152413955563473) --
			(-0.319327208754045,	-0.159805630538941) --
			(-0.315630016194951,	-0.167155976081567) --
			(-0.311844257562695,	-0.174462403424636) --
			(-0.307971182864434,	-0.181722351238342) --
			(-0.304012103464865,	-0.188933288618890) --
			(-0.299968390607657,	-0.196092718028031) --
			(-0.295841473832697,	-0.203198178175837) --
			(-0.291632839292987,	-0.210247246839739) --
			(-0.287344027975343,	-0.217237543613062) --
			(-0.282976633829376,	-0.224166732576502) --
			(-0.278532301809477,	-0.231032524886277) --
			(-0.274012725834846,	-0.237832681272969) --
			(-0.269419646672790,	-0.244565014445375) --
			(-0.264754849750806,	-0.251227391394026) --
			(-0.260020162903074,	-0.257817735589403) --
			(-0.255217454057250,	-0.264334029070209) --
			(-0.250348628867510,	-0.270774314417501) --
			(-0.245415628299982,	-0.277136696610834) --
			(-0.240420426176750,	-0.283419344763020) --
			(-0.235365026684717,	-0.289620493730513) --
			(-0.230251461855617,	-0.295738445596854) --
			(-0.225081789023525,	-0.301771571027062) --
			(-0.219858088266156,	-0.307718310491286) --
			(-0.214582459836243,	-0.313577175356453) --
			(-0.209257021589206,	-0.319346748845124) --
			(-0.203883906413233,	-0.325025686861150) --
			(-0.198465259667801,	-0.330612718682190) --
			(-0.193003236636513,	-0.336106647519535) --
			(-0.187500000000000,	-0.341506350946110);
			\end{scope}
			\draw[gray] (-0.187500000000000,	-0.341506350946110) --
			(-0.180329930120586,	-0.343766274229212) --
			(-0.173101823075371,	-0.345943453681239) --
			(-0.165817667218458,	-0.348036512633383) --
			(-0.158479508988139,	-0.350044117347471) --
			(-0.151089451316604,	-0.351964979056462) --
			(-0.143649651920466,	-0.353797855958375) --
			(-0.136162321474197,	-0.355541555158102) --
			(-0.128629721668967,	-0.357194934551667) --
			(-0.121054163159782,	-0.358756904647522) --
			(-0.113438003404256,	-0.360226430319642) --
			(-0.105783644396710,	-0.361602532487261) --
			(-0.0980935303017108,	-0.362884289716288) --
			(-0.0903701449915307,	-0.364070839737600) --
			(-0.0826160094923844,	-0.365161380877615) --
			(-0.0748336793446469,	-0.366155173396791) --
			(-0.0670257418825982,	-0.367051540731913) --
			(-0.0591948134395540,	-0.367849870638324) --
			(-0.0513435364845431,	-0.368549616228529) --
			(-0.0434745766969648,	-0.369150296903899) --
			(-0.0355906199859104,	-0.369651499176535) --
			(-0.0276943694610588,	-0.370052877378680) --
			(-0.0197885423622534,	-0.370354154257403) --
			(-0.0118758669550381,	-0.370555121452668) --
			(-0.00395907939956922,	-0.370655639857232) --
			(0.00395907939956913,	-0.370655639857232) --
			(0.0118758669550381,	-0.370555121452668) --
			(0.0197885423622533,	-0.370354154257403) --
			(0.0276943694610587,	-0.370052877378680) --
			(0.0355906199859103,	-0.369651499176535) --
			(0.0434745766969647, -0.369150296903898) --
			(0.0513435364845430,	-0.368549616228529) --
			(0.0591948134395539,	-0.367849870638324) --
			(0.0670257418825981,	-0.367051540731913) --
			(0.0748336793446468,	-0.366155173396791) --
			(0.0826160094923843,	-0.365161380877615) --
			(0.0903701449915306,	-0.364070839737600) --
			(0.0980935303017107,	-0.362884289716288) --
			(0.105783644396710,	-0.361602532487261) --
			(0.113438003404256,	-0.360226430319642) --
			(0.121054163159782,	-0.358756904647522) --
			(0.128629721668967,	-0.357194934551667) --
			(0.136162321474197,	-0.355541555158102) --
			(0.143649651920466,	-0.353797855958374) --
			(0.151089451316604,	-0.351964979056462) --
			(0.158479508988139,	-0.350044117347471) --
			(0.165817667218458,	-0.348036512633383) --
			(0.173101823075371,	-0.345943453681239) --
			(0.180329930120586,	-0.343766274229212) --
			(0.187500000000000,	-0.341506350946110);
			\end{tikzpicture}
			\caption{Patches (in black) and initial elements (in gray) of $\Omega_0$.}
			\label{fig:patchesOmega0halfdisc}
		\end{center}
	\end{subfigure}
	\caption{Numerical test \ref{sec:halfdiscwithhole} -- Exact and initial defeatured domains used for the adaptive strategy analysis.}
	\label{fig:halfdiscgeom}
\end{figure}
We consider Poisson's problem~(\ref{eq:originalpb}) in $\Omega$, with $f\equiv -1$, $g_D\equiv 0$ on $$\Gamma_D := \left\{ \boldsymbol x=(x,y)^T\in\mathbb R^2: \|\boldsymbol x \|_{\ell^2}=\frac{1}{2}, \, y<0 \right\}, $$ 
and $g\equiv 0$ on $\Gamma_N := \partial \Omega \setminus \overline{\Gamma_D}$. The exact solution of this problem is given by
$$u(x,y) = -\frac{5\cdot 10^{-5}}{4}\log\left(2\sqrt{x^2+y^2}\right) + \frac{x^2+y^2}{4} -\frac{1}{16}, \quad \text{ for all } (x,y)\in\Omega.$$
Then, we consider the defeatured problem~(\ref{eq:simplpb}) in $\Omega_0$, where $f$ is extended by $-1$ in $F$, and $g_0 \equiv 0$ on $\gamma_0 = \partial F \setminus \overline{\partial \Omega}$. We solve it using THB-spline based IGA, with $\Omega_0$ being divided into $4$ conforming patches, each of which is initially divided into $4$ elements, as illustrated in Figure~\ref{fig:patchesOmega0halfdisc}. 

We first perform the adaptive strategy of Section~\ref{sec:adaptive}. Then, we perform the same adaptive strategy but we never add the feature $F=:F^1$ to the geometrical model. This is done by not taking into account the contribution $\mathscr{E}_D\big(u_\mathrm d^h\big)=\mathscr{E}_D^1\big(u_\mathrm d^h\big)$ in the MARK module (see Section~\ref{sec:mark}). That is, we only perform standard mesh refinement steps by neglecting the defeaturing error contribution, while still computing the overall error and the proposed estimator. 
In this experiment, we use $\alpha_N = \alpha_D = 1$, that is, we give the same weight to the contribution of the numerical part of the error estimate as to the one of the defeaturing part of the error estimate. Moreover, we choose the marking parameter $\theta = 0.5$. 
When performing the REFINE module presented in Section~\ref{sec:refine}, we impose the mesh to be $\mathcal T$-admissible of class~$2$, and the elements are dyadically refined. Moreover, when the feature $F$ is marked for refinement, it is added to the geometrical model by trimming as explained in Section~\ref{sec:gen}. Both adaptive strategies are stopped whenever the number of degrees of freedom exceeds $10^4$. 

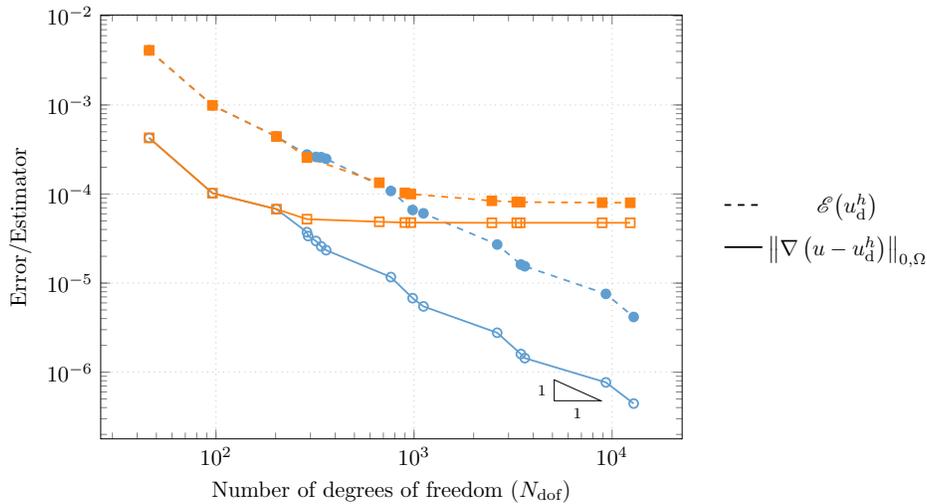
\begin{figure}
				\begin{center}
		\begin{tikzpicture} [scale=0.84]
		\begin{axis}[xmode=log, ymode=log, legend style={row sep=5pt, at={(1.25,0.6)}, legend columns=1, anchor=north, draw=none}, xlabel=Number of degrees of freedom ($N_\text{dof}$), ylabel=Error/Estimator, width=0.65\textwidth, height = 0.5\textwidth, grid style=dotted,grid]
		\addlegendimage{black,dashed, thick, mark options=solid}
		\addlegendentry{E}
		\addlegendimage{,black,thick}
		\addlegendentry{E}
		\addplot[mark=*, c4, dashed, thick, mark options=solid] table [x=ndof, y=est, col sep=comma] {data/test_half_disc_std2.csv};
		\addplot[mark=o, c4, thick] table [x=ndof, y=err, col sep=comma] {data/test_half_disc_std2.csv};
		\addplot[mark=square*, c2, dashed, thick, mark options=solid] table [x=ndof, y=est, col sep=comma] {data/test_half_disc_no_hole2.csv};
		\addplot[mark=square, c2, thick] table [x=ndof, y=err, col sep=comma] {data/test_half_disc_no_hole2.csv};
		\logLogSlopeReverseTriangle{0.78}{0.08}{0.09}{1}{black}
		\legend{$\mathscr E\big(u_\mathrm d^h\big)$, $\left\|\nabla\left(u-u_\mathrm d^h\right)\right\|_{0,\Omega}$}; 
		\end{axis}
		\end{tikzpicture}
		\caption{Numerical test \ref{sec:halfdiscwithhole} -- Convergence of the discrete defeaturing error and estimator with respect to the number of degrees of freedom. In {\color{c4}blue circles}, we consider the adaptive strategy presented in Section~\ref{sec:adaptive} and specialized to IGA, for which the feature is added after iteration $4$. In {\color{c2}orange squares}, we only consider mesh refinements, i.e., the feature is never added to the geometry.}\label{fig:halfdisc}
	\end{center} 
\end{figure}

The results are presented in Figure~\ref{fig:halfdisc}, and the final meshes obtained with each refinement strategy are drawn in Figure~\ref{fig:halfdiscmeshes}. As it is to be expected, the $C^0$-lines influence the refined mesh as the approximation is more accurate in the central bi-linear patch than in the bi-quadratic patches. 
The orange curves in Figure~\ref{fig:halfdisc} correspond to the mesh in Figure~\ref{fig:halfdiscmeshes} (right) and represent the behavior of the total error if the feature is never added. \changes{We can see that both the error and estimator present a plateau after the first iterations. This is expected, since after a few iterations in which the numerical component of the total error is reduced, its defeaturing component begins to dominate.}
Let us now analyze the results obtained with the adaptive strategy of Section~\ref{sec:adaptive}, represented as blue lines with circles in Figure~\ref{fig:halfdisc}. 
In this case, the feature is added after the third iteration, and the overall error converges as the inverse of the number of degrees of freedom $N_\text{dof}$, that is, as $N_\text{dof}^{-\frac{p}{2}}$, as expected. We verify that the geometrical feature is not added too late, otherwise we would first observe a plateau in the overall error followed by a large drop, followed again by a normal convergence. At the opposite, if the feature was added too early in the geometrical model, the convergence would not be affected, but it would be computationally more costly. Moreover, the discrete defeaturing error estimator follows very well the behavior of the error with a relatively low effectivity index of $9.41$ on average, confirming once again its efficiency and reliability. 

\begin{figure}
	\begin{subfigure}[t]{0.48\textwidth}
		\begin{center}
			\includegraphics[width=200pt, height=100pt]{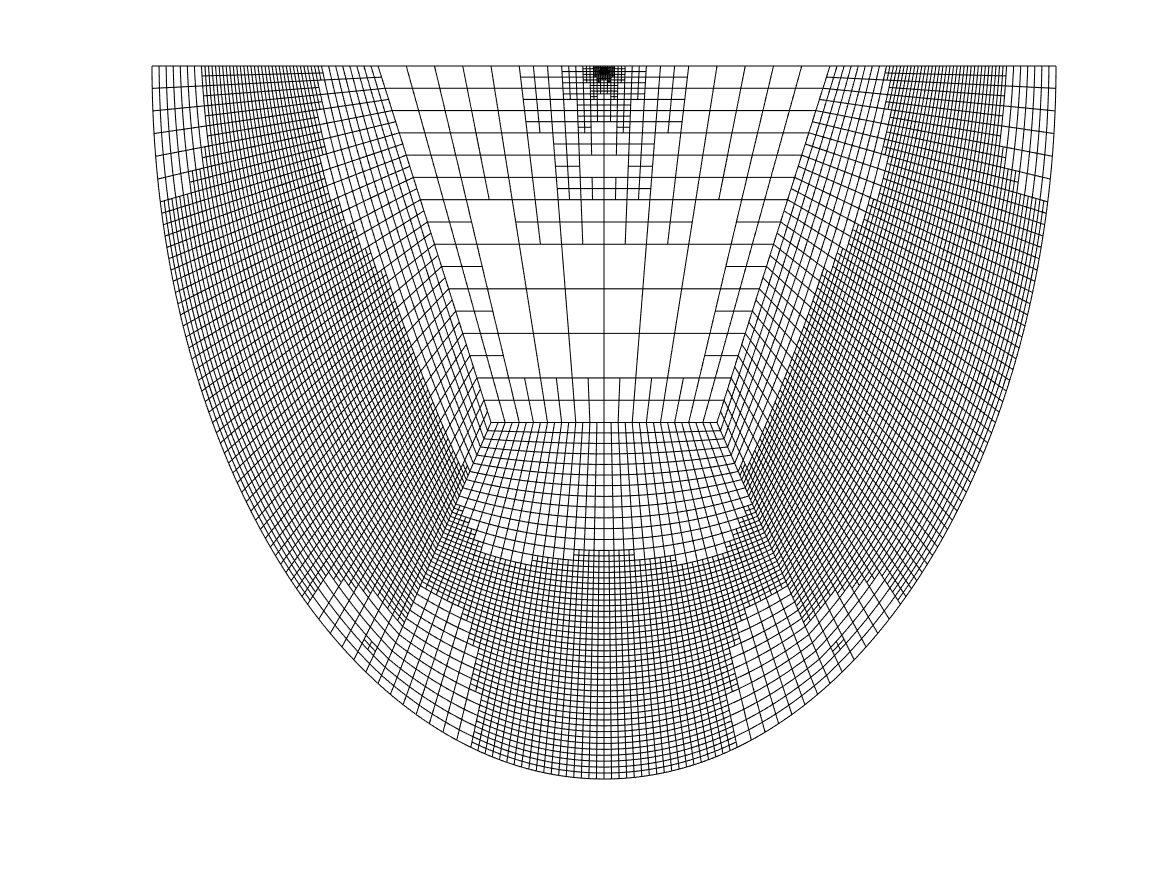}
			\caption{Final mesh obtained with the adaptive defeaturing strategy presented in Section~\ref{sec:adaptive}, where the feature has been added by trimming.} \label{fig:meshadapt}
		\end{center}
	\end{subfigure}
	~
	\begin{subfigure}[t]{0.48\textwidth}
		\begin{center}
			\includegraphics[width=200pt, height=100pt]{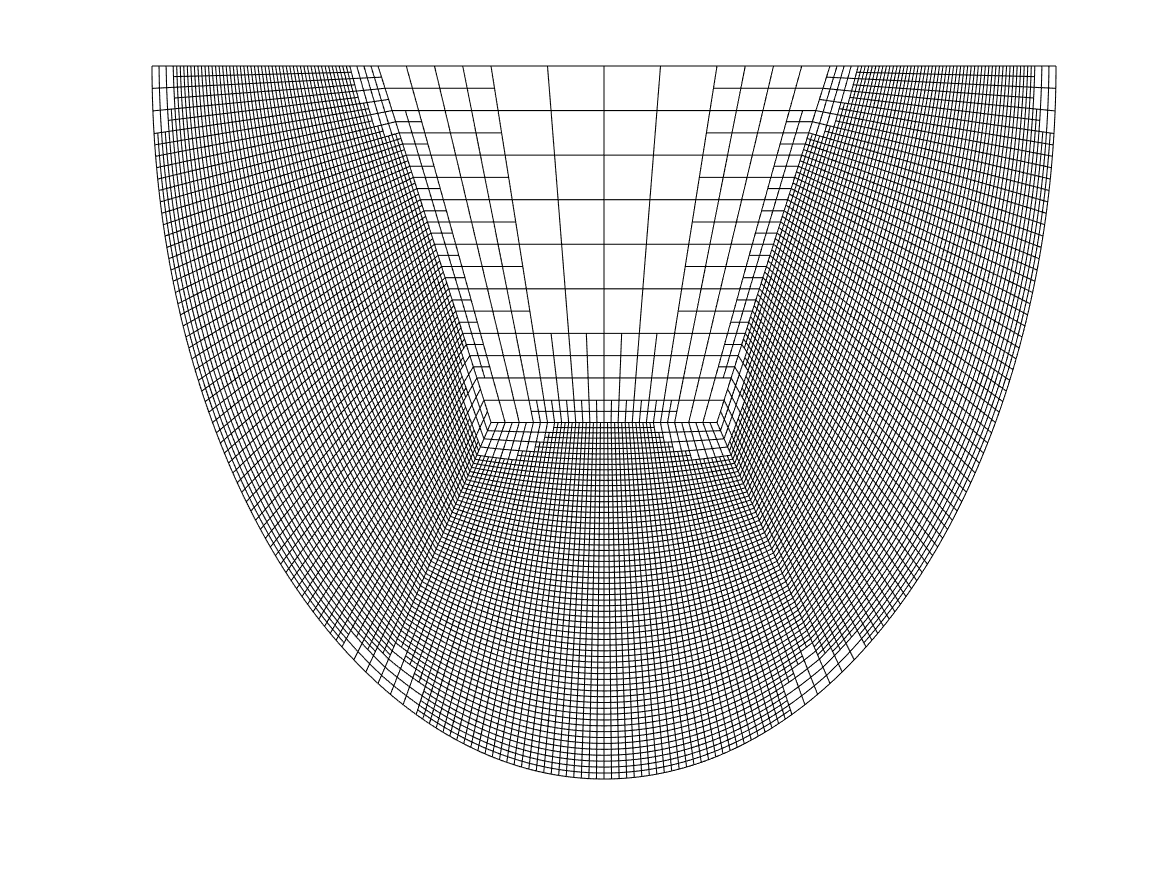}
			\caption{Final mesh obtained with a standard IGA mesh refinement strategy, without geometric refinement.} \label{fig:meshwrongadapt}
		\end{center}
	\end{subfigure}
	\caption{Numerical test \ref{sec:halfdiscwithhole} -- Final meshes obtained with adaptive defeaturing, and with a standard mesh refinement strategy without geometric refinement.} \label{fig:halfdiscmeshes}
\end{figure}


\subsubsection{Positive feature}\label{sec:positiveadaptivestrategy}
Let us now consider the same problem setting as in Section \ref{sec:cvtestpos}, i.e., a defeaturing problem on an L-shaped domain containing a positive fillet feature, and let us fix $\varepsilon=0.1$. As in the previous Section~\ref{sec:halfdiscwithhole}, we first perform the adaptive strategy proposed in Section~\ref{sec:adaptive}; then, we perform the classical mesh refinement adaptive strategy without geometric adaptivity, that is, the feature $F:=F_\varepsilon$ is never added to the geometrical model. In both cases, we start with a mesh containing one element per direction and per patch (see Figure~\ref{fig:patchesfillet} for the patch decomposition of~$\Omega_0$), and the algorithm is stopped whenever the total number of degrees of freedom $N_\mathrm{dof}$ exceeds $10^4$. If $F$ is not in the geometrical model, recall that the total number of degrees of freedom $N_\mathrm{dof}$ accounts for the number of degrees of freedom of the Galerkin approximation of problem~(\ref{eq:weaksimplpb}) in $\Omega_0$, to which we add the number of degrees of freedom of the Galerkin approximation of problem~(\ref{eq:weakfeaturepb}) in $\tilde F:=\tilde F_\varepsilon$. If the positive fillet $F$ is marked for refinement, then it is added to the geometrical model as a trimmed patch, the considered patch being its bounding box $\tilde F$ as in Figure~\ref{fig:Lshapefillet}. In this case, we recall that $N_\mathrm{dof}$ accounts for the number of active degrees of freedom of the discrete problem in the exact (multipatch trimmed) domain $\Omega := \Omega_\varepsilon$. 
In this experiment, we use $\alpha_N = \alpha_D = 1$, and we choose the marking parameter $\theta = 0.5$. During the refining phase, we impose the mesh to be $\mathcal T$-admissible of class~$2$, and the marked elements are dyadically refined at each iteration. 

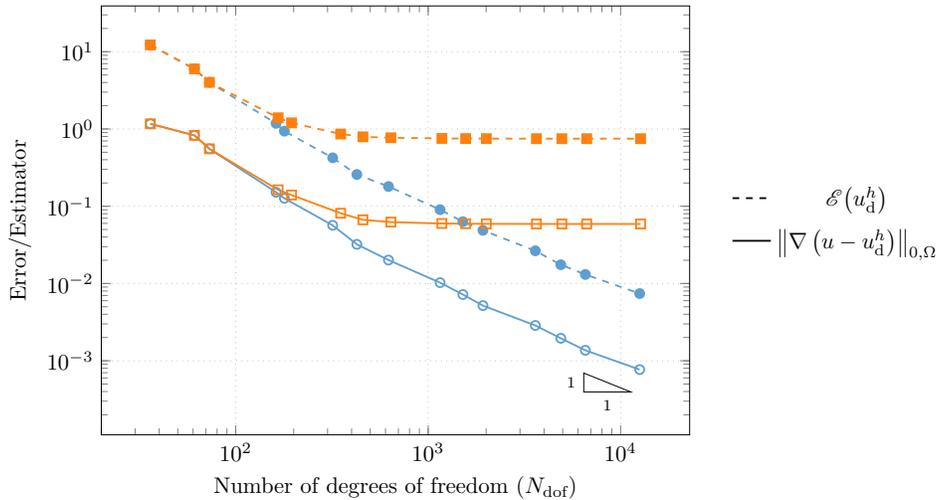
\begin{figure}
	\begin{center}
		\begin{tikzpicture} [scale=0.85]
		\begin{axis}[xmode=log, ymode=log, legend style={row sep=5pt, at={(1.25,0.6)}, legend columns=1, anchor=north, draw=none}, xlabel=Number of degrees of freedom ($N_\text{dof}$), ymin=1.1e-4, ylabel=Error/Estimator, width=0.65\textwidth, height = 0.5\textwidth, grid style=dotted,grid]
		\addlegendimage{black,dashed, thick, mark options=solid}
		\addlegendentry{E}
		\addlegendimage{,black,thick}
		\addlegendentry{E}
		\addplot[mark=*, c4, dashed, thick, mark options=solid] table [x=ndof, y=est, col sep=comma] {data/test_fillet_adapt_std.csv};
		\addplot[mark=o, c4, thick] table [x=ndof, y=err, col sep=comma] {data/test_fillet_adapt_std.csv};
		\addplot[mark=square*, c2, dashed, thick, mark options=solid] table [x=ndof, y=est, col sep=comma] {data/test_fillet_adapt_no_hole.csv};
		\addplot[mark=square, c2, thick] table [x=ndof, y=err, col sep=comma] {data/test_fillet_adapt_no_hole.csv};
		\logLogSlopeReverseTriangle{0.82}{0.08}{0.1}{1}{black}
		\legend{$\mathscr E\big(u_\mathrm d^h\big)$, $\left\|\nabla\left(u-u_\mathrm d^h\right)\right\|_{0,\Omega}$}; 
		\end{axis}
		\end{tikzpicture}
		\caption{Numerical test \ref{sec:positiveadaptivestrategy} -- Convergence of the discrete defeaturing error and estimator with respect to the number of degrees of freedom. In {\color{c4}blue circles}, we consider the adaptive strategy of Section~\ref{sec:adaptive} and specialized to IGA, for which the feature is added after iteration~$4$. In {\color{c2}orange squares}, we only consider mesh refinements, i.e., the feature is never added to the geometry.}\label{fig:filletadapt}
	\end{center} 
\end{figure}

The results are presented in Figure~\ref{fig:filletadapt}. 
The fillet is added after the third iteration of the adaptive strategy of Section~\ref{sec:adaptive}, whose results are represented as blue lines with circles in Figure~\ref{fig:filletadapt}. As for the previous numerical experiment of Section~\ref{sec:halfdiscwithhole}, we can see that the feature is not added too late, as there is no plateau nor a large drop in the convergence of the error and estimator. The overall error converges as the inverse of the number of degrees of freedom $N_\text{dof}$, that is, as $N_\text{dof}^{-\frac{p}{2}}$, as expected. Moreover, the discrete defeaturing error estimator follows very well the behavior of the overall error with a relatively low effectivity index of $8.6$ on average, confirming once again its efficiency and reliability. 

Then, we compare the results with the ones obtained with the adaptive strategy in which the feature is never added to the geometrical model, represented as orange lines with squares in Figure~\ref{fig:filletadapt}. Beginning from iteration~$4$, we observe a plateau in the overall error with respect to the number of degrees of freedom, as the defeaturing component of the error starts to dominate over the numerical component. The overall error cannot decrease any further if the geometrical model is not refined, i.e., the fillet should be added. 
The discrete defeaturing error estimator follows the behavior of the error, and in particular, the effectivity index in the last eight iterations is equal to $12.6$ in average. This is coherent with what has been observed in the numerical experiment of Section~\ref{sec:cvtestpos}. 

\changes{\begin{remark}
Note that the L-shaped domain $\Omega_0$ suggests the case in which the problem at hand displays a singularity at the re-entrant corner of $\Omega_0$. This case is of great importance, as the presence of the positive feature $F$ in the exact geometry $\Omega$ should remove the singularity. However, it is not trivial to design such a numerical experiment in which one can compare the exact and the defeatured solutions, when dealing with Neumann boundary conditions around the feature. This will thus be the subject of a further study in which Dirichlet boundary conditions are also allowed on the feature's boundary.
\end{remark}}


\subsection{Fully adaptive strategy in a geometry with many features} \label{sso:multi}
In this last numerical experiment \changes{inspired from \cite[Section~6.1.3]{paper3multifeature}, we analyze the proposed fully adaptive strategy in a geometry with many features, and we study the impact of the size of the features on the discrete defeaturing error and estimator. In particular, since the defeaturing contribution of the estimator depends upon the size of the features and the size of the solution gradients ``around'' the feature, we will show that the proposed adaptive strategy is able to tell when small features count more than big ones, even in presence of numerical approximation errors.}

More precisely, let $\Omega_0:= (0,1)^2$ be the fully defeatured domain, and let $\Omega := \Omega_0 \setminus \bigcup_{k=1}^{N_f} \overline{F^k}$, where $N_f=27$ and the features $F^k$ are some circular holes of radii in the interval $(0,\,5\cdot 10^{-3})$, distributed with some randomness in $\Omega_0$, as illustrated in Figure~\ref{fig:geom27holes}. For the sake of reproducibility, the values of the centers \changes{and diameters} of the features are reported in Table~\ref{tbl:radiusescenters}. Note in particular that this geometry satisfies Assumption~\ref{def:separated}. 

\begin{table}
	\centering
	{\def\arraystretch{1.2}
		\begin{tabular}{@{}cccccccccc@{}}
			\hline
			Feature index $k$ & $1$ & $2$ & $3$ & $4$ & $5$ & $6$ & $7$ & $8$ & $9$ \\
			\hline
			Diameter $[\cdot 10^{-2}]$ & $8.13$ & $6.64$ & $3.89$ & $7.40$ & $8.18$ & $6.00$ & $0.85$ & $9.22$ & $0.54$ \\
			Center $[\cdot 10^{-2}]$ & $0.98$ & $2.84$ & $5.46$ & $7.16$ & $8.99$ & $0.67$ & $3.12$ & $4.95$ & $7.06$ \\
			& $0.93$ & $1.24$ & $0.57$ & $0.93$ & $1.04$ & $3.40$ & $3.03$ & $3.08$ & $2.48$ \\
			\hline\hline
			Feature index $k$ & $10$  & $11$  & $12$  & $13$  & $14$ & $15$  & $16$  & $17$  & $18$ \\
			\hline
			Diameter $[\cdot 10^{-2}]$ & $5.27$ & $1.19$ & $3.80$ & $8.13$ & $2.44$ & $8.84$ & $7.13$ & $3.78$ & $2.49$ \\
			Center $[\cdot 10^{-2}]$ & $8.86$ & $0.67$ & $3.28$ & $5.01$ & $7.44$ & $8.93$ & $1.10$ & $2.44$ & $5.45$ \\
			& $2.90$ & $5.35$ & $4.46$ & $5.09$ & $4.88$ & $5.07$ & $6.93$ & $6.78$ & $7.73$ \\
			\hline\hline
			Feature index $k$ & $19$  & $20$  & $21$  & $22$  & $23$  & $24$  & $25$ & $26$  & $27$ \\
			\hline
			Diameter $[\cdot 10^{-2}]$ & $2.53$ & $6.67$ & $0.50$ & $6.85$ & $6.20$ & $7.47$ & $8.77$ & $2.00$ & $1.00$ \\
			Center $[\cdot 10^{-2}]$ & $7.27$ & $9.21$ & $0.22$ & $3.26$ & $5.01$ & $7.06$ & $8.99$ & $4.00$ & $1.00$ \\
			& $7.33$ & $6.96$ & $8.24$ & $9.15$ & $9.10$ & $8.78$ & $8.98$ & $7.00$ & $9.00$ \\
			\hline
	\end{tabular}}
	\caption{Numerical test~\ref{sso:multi} -- Data of the $27$ circular features.} \label{tbl:radiusescenters}
\end{table}

We are interested in the solution of problem~(\ref{eq:originalpb}) defined in $\Omega$, and we solve the Galerkin formulation of the defeatured problem~(\ref{eq:weaksimplpb}) in $\Omega_0$. We consider Poisson's problem with $f(x,y) := -128e^{-8(x+y)}$ in $\Omega_0$, $g_D(x,y) := e^{-8(x+y)}$ on 
$$\Gamma_D := \big( [0,1)\times\{0\} \big) \cup \big( \{0\}\times[0,1) \big),$$
the bottom and left sides, $g(x,y):= -8e^{-8(x+y)}$ on $\partial \Omega_0 \setminus \overline{\Gamma_D}$, and $g \equiv 0$ on $\partial F^k$ for all $k=1,\ldots,N_f$. That is, as illustrated in Figure~\ref{fig:exactsolfinal}, the exact solution $u$ has a high gradient close to the bottom left corner, and it is almost constantly zero in the top right area of the domain. \changes{Therefore, one can expect: 
\begin{itemize}
	\item the mesh to be refined around the bottom left angle,
	\item and the presence of small features around the bottom left angle to be more important than the one of large features around the top right angle, with respect to the solution's accuracy.
\end{itemize}}

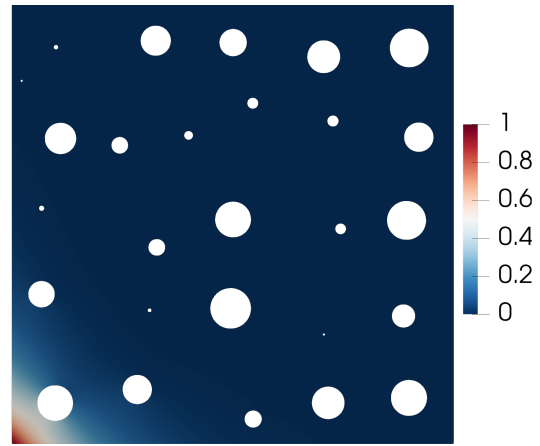
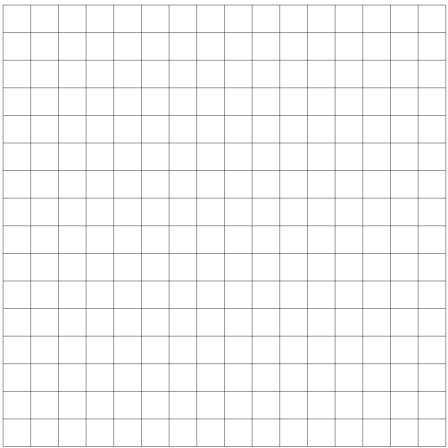
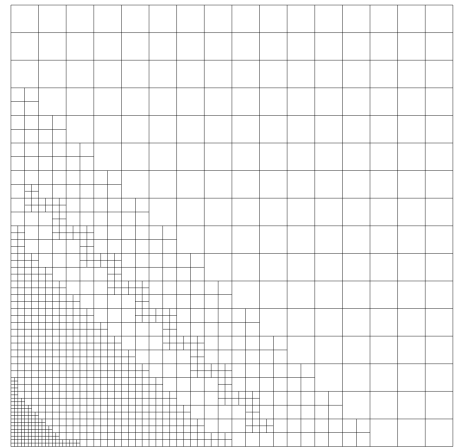
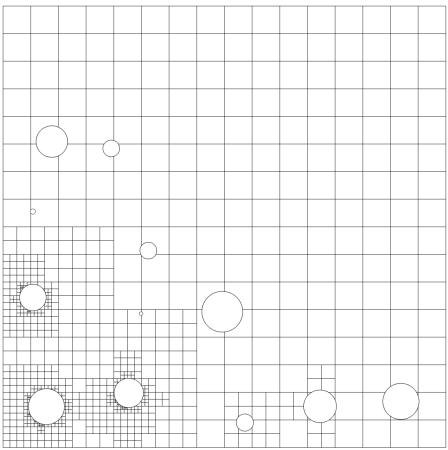
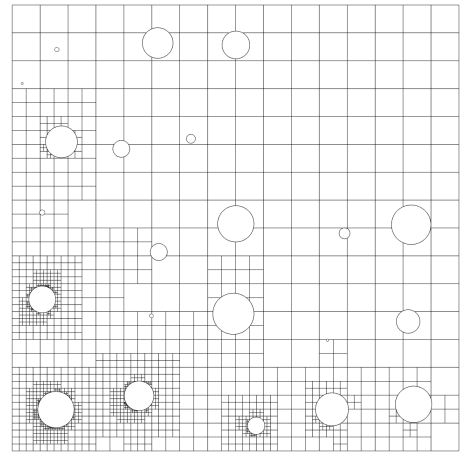
\begin{figure}
	\begin{subfigure}[t]{0.48\textwidth}
		\begin{center}
			\begin{tikzpicture}[scale=5.8]
			\draw[thick] (0,0) -- (0,1) -- (1,1) -- (1,0) -- cycle;
			\draw[c1,thick] (0.097984328699318,0.092663690863534) circle (0.040818234590090);
			\draw[c1,thick] (0.283884293965162,0.123563491365580) circle (0.033396825383881);
			\draw[c1,thick] (0.546112220339815,0.056574896706227) circle (0.019919282299261);
			\draw[c1,thick] (0.715916527014347,0.093265457598229) circle (0.037000758141615);
			\draw[c1,thick] (0.898581968051930,0.104463986172862) circle (0.040763485970873);
			\draw[c1,thick] (0.067305141757723,0.339769841694328) circle (0.030034479317649);
			\draw[c1,thick] (0.311761232757188,0.303173986566693) circle (0.004299706450270);
			\draw[c1,thick] (0.495145556524009,0.307616854246403) circle (0.046235800109005);
			\draw[c1,thick] (0.705998218664498,0.248238323585443) circle (0.002659781531656);
			
			\draw[c1,thick] (0.886286875701112,0.290351799644088) circle (0.026702495267737);
			\draw[c1,thick] (0.067359418419866,0.534641171336466) circle (0.005885327661975);
			\draw[c1,thick] (0.328109457773039,0.446027904596783) circle (0.019014299024100);
			\draw[c1,thick] (0.500590433180793,0.509270838208646) circle (0.040283253301511);
			\draw[c1,thick] (0.743939933090058,0.487966072170030) circle (0.012409591557169);
			\draw[c1,thick] (0.893169207667069,0.507196247921792) circle (0.044442265909357);
			\draw[c1,thick] (0.110235870451937,0.693080838594426) circle (0.035264679475915);
			\draw[c1,thick] (0.244342487141727,0.677497408280121) circle (0.018814843808214);
			\draw[c1,thick] (0.545240931016100,0.773012684854573) circle (0.012891961099593);
			
			\draw[c1,thick] (0.726683377942314,0.732602609237550) circle (0.012785374523927);
			\draw[c1,thick] (0.920762287773320,0.695958574121097) circle (0.033724359452694);
			\draw[c1,thick] (0.022384764760657,0.823728658042989) circle (0.002586188504724);
			\draw[c1,thick] (0.325708044929628,0.914763687652304) circle (0.034528856760891);
			\draw[c1,thick] (0.500755877184029,0.910427640482701) circle (0.031027812215100);
			\draw[c1,thick] (0.705821367824273,0.878409519462713) circle (0.037668462996873);
			\draw[c1,thick] (0.899163331742718,0.898273740857340) circle (0.043725564834567);
			\draw[c1,thick] (0.4,0.7) circle (0.010000000000000);
			\draw[c1,thick] (0.1,0.9) circle (0.005000000000000);
			
			\draw[c1,thick] (0.097984328699318,0.092663690863534) node{\tiny $1$};
			\draw[c1,thick] (0.283884293965162,0.123563491365580) node{\tiny $2$};
			\draw[c1,thick] (0.546112220339815,0.056574896706227) node{\tiny $3$};
			\draw[c1,thick] (0.715916527014347,0.093265457598229) node{\tiny $4$};
			\draw[c1,thick] (0.898581968051930,0.104463986172862) node{\tiny $5$};
			\draw[c1,thick] (0.067305141757723,0.339769841694328) node{\tiny $6$};
			\draw[c1,thick] (0.3,0.303173986566693) node[right]{\tiny $7$}; 
			\draw[c1,thick] (0.495145556524009,0.307616854246403) node{\tiny $8$};
			\draw[c1,thick] (0.715,0.248238323585443) node[left]{\tiny $9$}; 
			
			\draw[c1,thick] (0.886286875701112,0.290351799644088) node{\tiny $10$};
			\draw[c1,thick] (0.067,0.534641171336466) node[right]{\tiny $11$};
			\draw[c1,thick] (0.32,0.446027904596783) node[left]{\tiny $12$};
			\draw[c1,thick] (0.500590433180793,0.509270838208646) node{\tiny $13$};
			\draw[c1,thick] (0.743939933090058,0.487966072170030) node[left]{\tiny $14$};
			\draw[c1,thick] (0.893169207667069,0.507196247921792) node{\tiny $15$};
			\draw[c1,thick] (0.110235870451937,0.693080838594426) node{\tiny $16$};
			\draw[c1,thick] (0.244342487141727,0.677497408280121) node[right]{\tiny $17$};
			\draw[c1,thick] (0.545240931016100,0.773012684854573) node[left]{\tiny $18$};
			
			\draw[c1,thick] (0.726683377942314,0.732602609237550) node[left]{\tiny $19$};
			\draw[c1,thick] (0.920762287773320,0.695958574121097) node{\tiny $20$};
			\draw[c1,thick] (0.01,0.823728658042989) node[right]{\tiny $21$}; 
			\draw[c1,thick] (0.325708044929628,0.914763687652304) node{\tiny $22$};
			\draw[c1,thick] (0.500755877184029,0.910427640482701) node{\tiny $23$};
			\draw[c1,thick] (0.705821367824273,0.878409519462713) node{\tiny $24$};
			\draw[c1,thick] (0.899163331742718,0.898273740857340) node{\tiny $25$};
			\draw[c1,thick] (0.4,0.7) node[right]{\tiny $26$};
			\draw[c1,thick] (0.1,0.9) node[right]{\tiny $27$};
			\draw[white] (0,-0.012) -- (1,-0.012);
			\end{tikzpicture}
			\caption{Exact domain $\Omega$ with $27$ features.} \label{fig:geom27holes}
		\end{center}
	\end{subfigure}
	~
	\begin{subfigure}[t]{0.48\textwidth}
		\begin{center}
			\includegraphics[scale = 0.115]{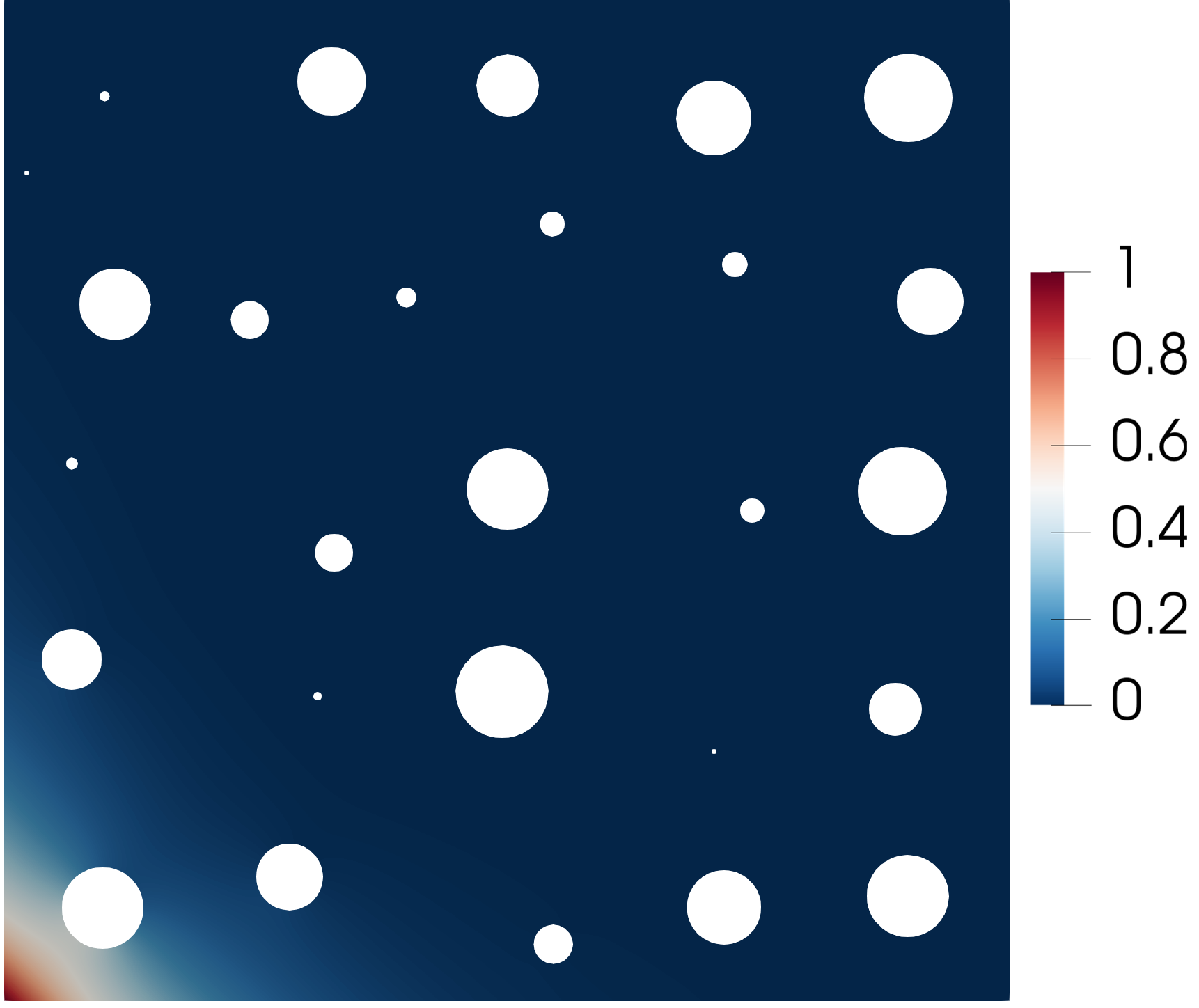}
			\caption{Exact solution in the exact domain.} \label{fig:exactsolfinal}
		\end{center}
	\end{subfigure}
	~
	\begin{subfigure}[t]{0.48\textwidth}
	\begin{center}
		\includegraphics[scale = 0.305,trim=55 70 30 40, clip]{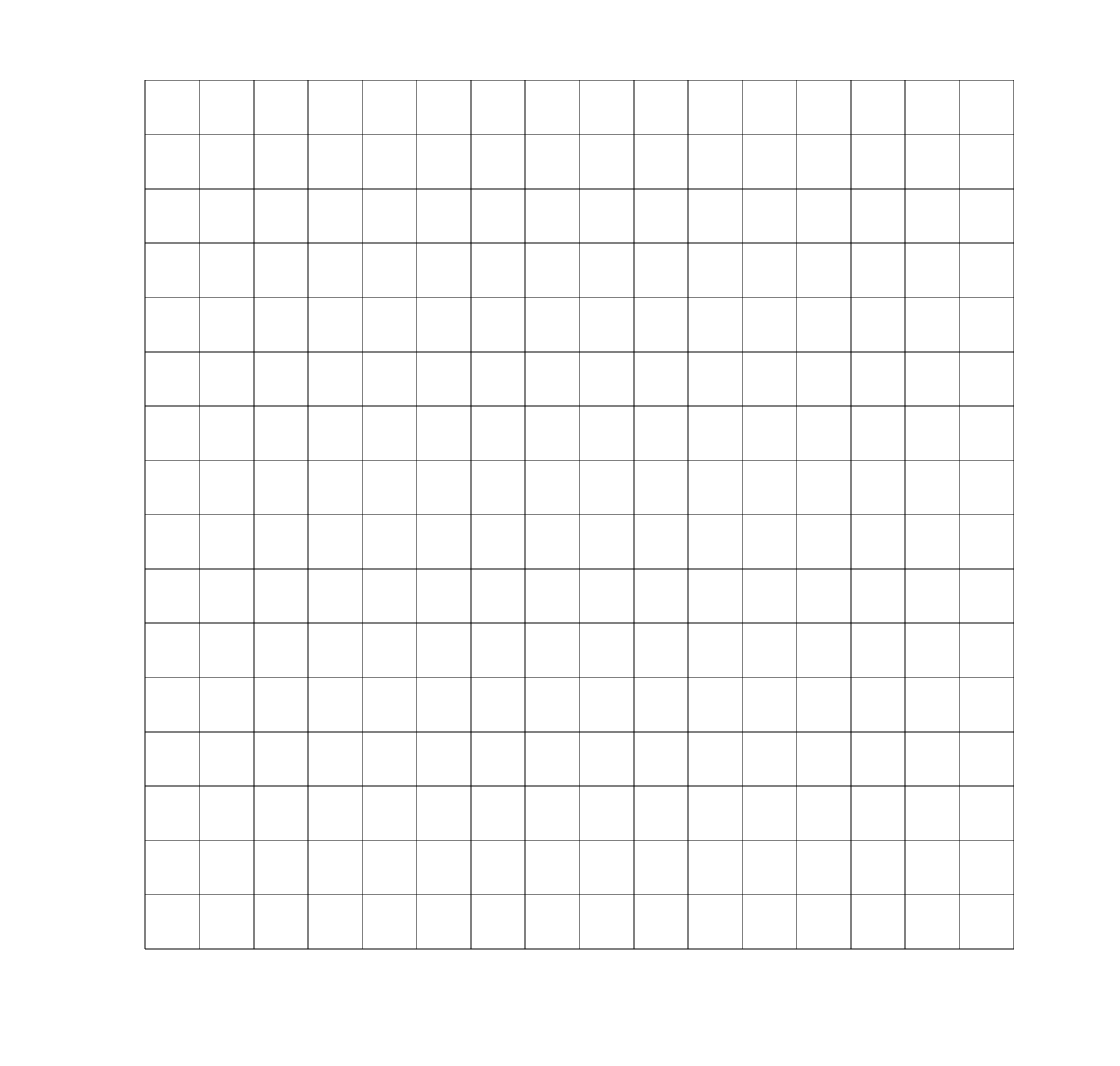} 
		\caption{Initial mesh.} \label{fig:initialmeshfinal}
	\end{center}
	\end{subfigure}
	~
	\begin{subfigure}[t]{0.48\textwidth}
	\begin{center}
		\includegraphics[scale = 0.305,trim=55 70 30 40, clip]{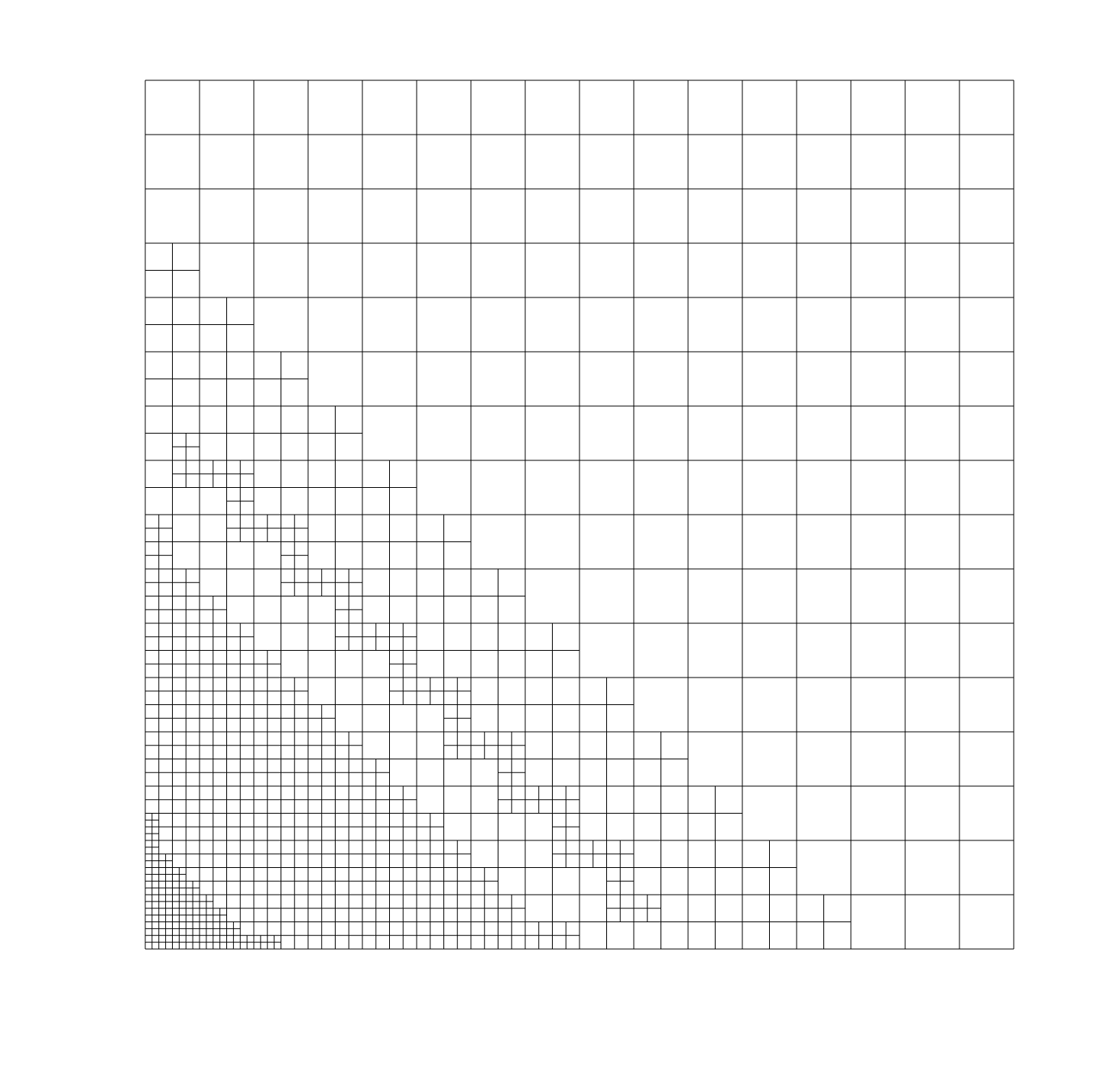} 
		\caption{Final mesh obtained without geometric refinement.} \label{fig:finalmeshbadfinal}
	\end{center}
	\end{subfigure}
	~
	\begin{subfigure}[t]{0.48\textwidth}
		\begin{center}
			\includegraphics[scale = 0.305,trim=55 70 30 40, clip]{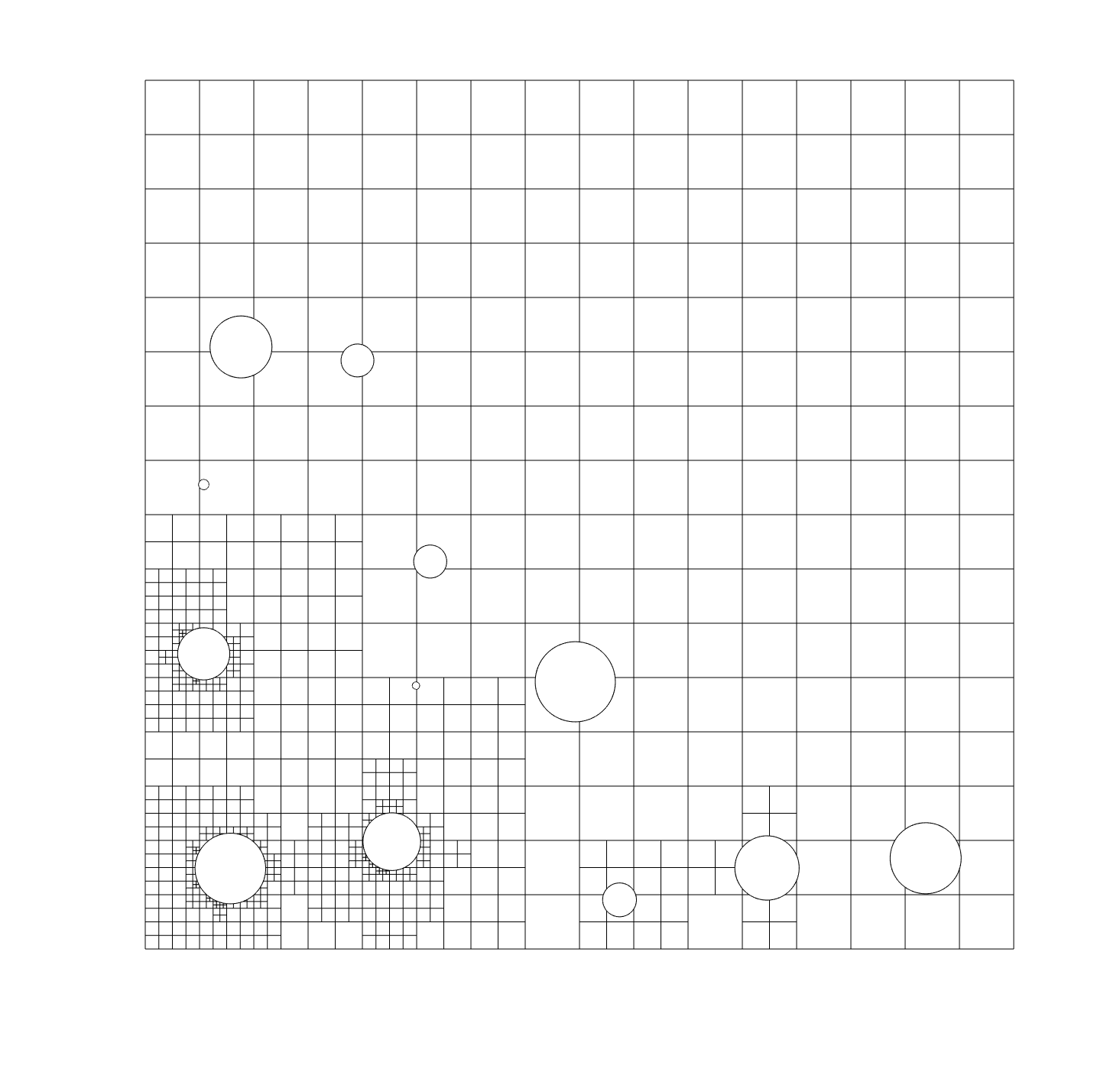}
			\caption{Mesh and geometry obtained with the combined refinement strategy at iteration $9$.} \label{fig:iter10meshfinal}
		\end{center}
	\end{subfigure}
	~
	\begin{subfigure}[t]{0.48\textwidth}
		\begin{center}
			\hspace{0.3cm}
			\includegraphics[scale = 0.31,trim=41 87 30 38, clip]{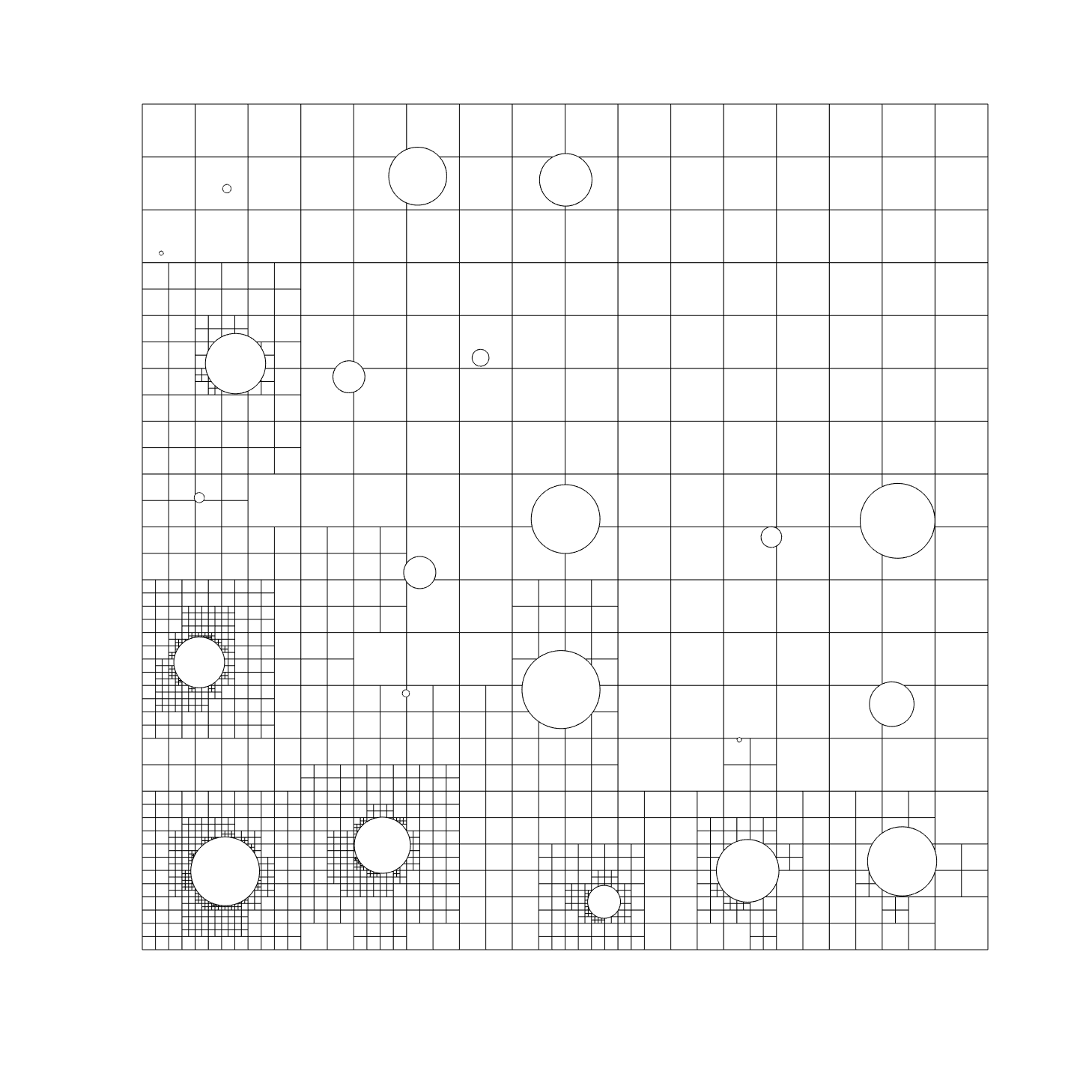} 
			\caption{Final mesh and geometry obtained with the combined refinement strategy.} \label{fig:finalmeshfinal}
		\end{center}
	\end{subfigure}
	\caption{Numerical test \ref{sso:multi} -- Considered geometry, initial mesh, and intermediate and final meshes and geometries obtained with adaptive strategies, with and without geometric refinement.} \label{fig:cvmfomeshes}
\end{figure}

\changes{As for the numerical experiments of Section~\ref{sec:adaptstrategytests}, we} first perform the adaptive strategy described in Section~\ref{sec:adaptive}, starting from the fully defeatured domain $\Omega_0$ and from a uniform mesh of $16\times 16$ elements in $\Omega_0$, as illustrated in Figure~\ref{fig:initialmeshfinal}. Then, we perform the same adaptive strategy but without geometric refinement, that is, with this second strategy, holes are never added to the geometrical model. \changes{To do so, we do not take into account the defeaturing contributions $\mathscr{E}_D^k\big(u_\mathrm d^h\big)$ in the MARK module (see Section~\ref{sec:mark}).} In this experiment, we use THB-splines of degree $p=3$, we consider $\alpha_N = 1$ and $\alpha_D = 4$, and we choose $\theta = 0.3$ as marking parameter. 
For the REFINE module precised for IGA in Section~\ref{ss:refineiga}, we impose the mesh to be $\mathcal T$-admissible of class $3$, and the elements are dyadically refined. Moreover, when a feature is marked for refinement, it is added to the geometrical model by trimming, as explained in Section~\ref{sec:gen}. 

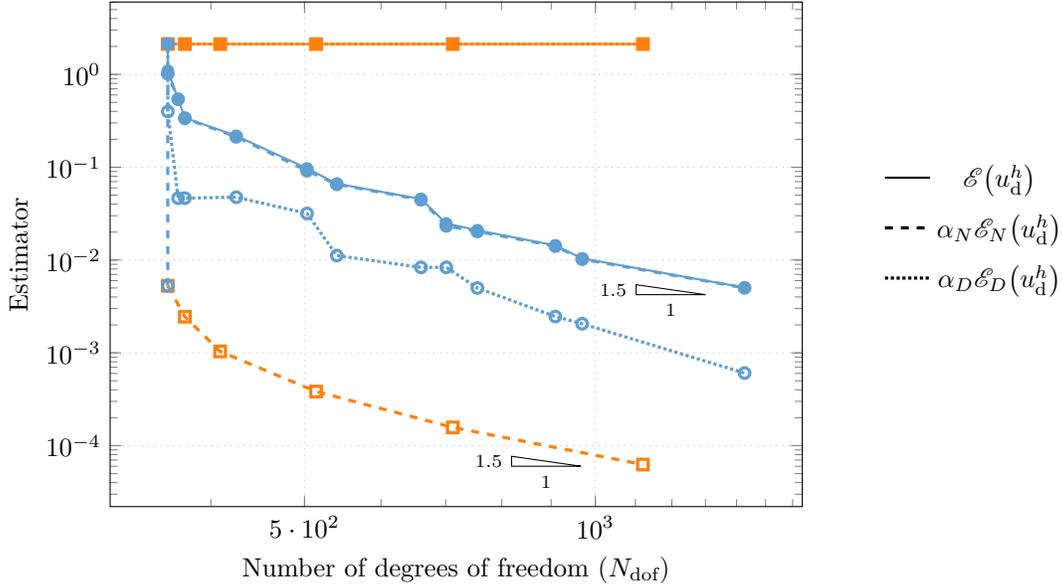
\begin{figure}
	\begin{center}
		\begin{tikzpicture}[scale=1]
		\begin{axis}[xmode=log, ymode=log, legend style={row sep=5pt, at={(1.25,0.7)}, legend columns=1, anchor=north, draw=none}, xlabel=Number of degrees of freedom ($N_\text{dof}$), ylabel=Estimator, width=0.65\textwidth, height = 0.5\textwidth, grid style=dotted,grid,
		xtick = {500,1000}, 
		minor xtick = {300, 400, 500, 600,700,800,900,1000, 1100, 1200, 1300, 1400, 1500, 1600},
		xticklabels={$5\cdot 10^2$, ${10}^3$}] 
		\addlegendimage{black, thick, mark options=solid}
		\addlegendentry{E}
		\addlegendimage{dashed,black,very thick,mark options=solid}
		\addlegendentry{E}
		\addlegendimage{densely dotted,black,very thick,mark options=solid}
		\addlegendentry{E}
		\addplot[mark=square*, c2, thick, mark options=solid] table [x=ndof, y=est2, col sep=comma] {data/final_simulation_no_holes2.csv};
		\addplot[mark=square, c2, densely dotted, very thick, mark options=solid] table [x=ndof, y=est_mod4, col sep=comma] {data/final_simulation_no_holes2.csv};
		\addplot[mark=square, c2, dashed, very thick, mark options=solid] table [x=ndof, y=est_num, col sep=comma] {data/final_simulation_no_holes2.csv};
		\addplot[mark=*, c4, thick, mark options=solid] table [x=ndof, y=est2, col sep=comma] {data/final_simulation2.csv};
		\addplot[mark=o, c4, densely dotted, very thick, mark options=solid] table [x=ndof, y=est_mod4, col sep=comma] {data/final_simulation2.csv};
		\addplot[mark=o, c4, dashed, very thick, mark options=solid] table [x=ndof, y=est_num, col sep=comma] {data/final_simulation2.csv};
		\logLogSlopeReverseTriangle{0.58}{0.1}{0.08}{1.5}{black}
		\logLogSlopeReverseTriangle{0.76}{0.1}{0.42}{1.5}{black}
		\legend{$\mathscr E\big(u_\mathrm d^h\big)$, $\alpha_N\mathscr E_N\big(u_\mathrm d^h\big)$, $\alpha_D\mathscr E_D\big(u_\mathrm d^h\big)$};
		\end{axis}
		\end{tikzpicture}
		\caption{Numerical test \ref{sso:multi} -- Convergence of the \changes{overall error estimator and of its numerical error component,} with respect to the number of degrees of freedom. In {\color{c4}blue circles}, we consider the adaptive strategy of Section~\ref{sec:adaptive} \changes{specialized to IGA,} in which features are iteratively added to the geometric model. Note that the dashed and solid blue curves are basically superposed. In {\color{c2}orange squares}, we only consider mesh refinements, \changes{that is, the adaptive process is only steered by the numerical contribution of the overall error estimator, and features are never added to the geometry.} In this case, note that the dotted and solid orange curves are basically superposed.} \label{fig:cvmfo}
	\end{center} 
\end{figure}

The mesh and geometry obtained at iteration $9$, and the final mesh and geometry obtained with both adaptive strategies when the total number of degrees of freedom exceeds $10^3$ are represented in Figure~\ref{fig:cvmfomeshes}. Results are reported in Figure~\ref{fig:cvmfo}. The blue lines with circles correspond to the adaptive strategy of Section~\ref{sec:adaptive}, and the sets of marked features at each iteration are the following: 
$\{1\}$, $\{2,6\}$, $\emptyset$, $\emptyset$, $\{3\}$, $\{4,8,11,16\}$, $\{5\}$, $\emptyset$, $\{7,12,17\}$, $\{10,13,22\}$, $\{15\}$, $\{9,14,21,23,26,27\}$. For instance, the error estimator is divided by $10$ with the addition of only $4$ out of the $27$ features in the geometrical model, and with a number of degrees of freedom increased by slightly more than a third. \changes{Note for instance that the small feature~$F^7$ is added a few iterations before the large feature $F^{25}$, even if $F^{25}$ is an order of magnitude larger than $F^7$. This confirms our intuition on the problem and this is coherent with the results obtained in \cite{paper1defeaturing}.} Moreover, the overall estimator and its numerical and defeaturing contributions converge as $N_\mathrm{dof}^{-\frac{3}{2}}=N_\mathrm{dof}^{-\frac{p}{2}}$, as expected. 
Furthermore, the final mesh, represented in Figure~\ref{fig:finalmeshfinal}, is refined towards the lower left corner of the domain, and the first selected features are also the ones closer to that corner. This is indeed expected as the exact solution has a high gradient around that corner. 

The orange lines with squares in Figure~\ref{fig:cvmfo} correspond to the results of the adaptive strategy without geometric refinement, i.e., when the defeaturing component of the error is not considered. We can observe that convergence is lost, because the defeaturing error contribution of the estimator is and remains very high, even if the numerical error contribution keeps converging as $N_\mathrm{dof}^{-\frac{p}{2}}$. The obtained final mesh is refined around the lower left corner, reflecting the high gradient of the solution in this area, but if we do not add any feature to the geometry, one cannot obtain a more accurate solution. This is reflected by the proposed discrete defeaturing error estimator, validating the developed theory.

%% file: appendix.tex
In this section, we state and give the (reference to the) proof of some results that are used throughout the paper. The symbol $\lesssim$ will be used to mean any inequality which does not depend on the mesh size nor on the size of the considered domains, but which can depend on their shape.

\begin{lemma}[Trace inequalities] \label{lemma:traceineq}
	\begin{itemize}
		\item Let $D\subset \mathbb R^n$ be a bounded Lipschitz domain, and let $h_D := \mathrm{diam}(D)$. Then for all $v\in H^1(D)$, 
		\begin{equation*} 
		\|v\|_{0,\partial D} \lesssim h_D^{-\frac{1}{2}} \|v\|_{0,D} + h_D^{\frac{1}{2}} \left\|\nabla v\right\|_{0,D}.
		\end{equation*}
		\item Let $\mathcal Q$ be a shape regular FE mesh. Then for all $K\in \mathcal Q$ and for all $v\in H^1(K)$, 
		\begin{equation*} 
		\|v\|_{0,\partial K} \lesssim h_K^{-\frac{1}{2}} \|v\|_{0,K} + h_K^{\frac{1}{2}} \left\|\nabla v\right\|_{0,K}.
		\end{equation*}
	\end{itemize}
\end{lemma}
\begin{proof}
	For the first part, use \cite[Theorem~1.5.1.10]{grisvard} and take $\varepsilon = h_D$. For the second part, see \cite[Section~1.4.3]{dipietroern}.
\end{proof}

\begin{lemma}[Friedrichs' inequality] \label{lemma:friedrichs}
	Let $D\subset \mathbb R^n$ be a bounded Lipschitz domain, 
	let $h_D := \mathrm{diam}(D)$, and let $\omega \subset \partial D$ such that $|\omega|>0$ and $|\omega|^\frac{1}{n-1}\simeq h_D$. Then for all $v\in H^1(D)$, 
	\begin{equation*} 
	\|v-\overline{v}^{\omega}\|_{0,D} \lesssim h_D \|\nabla v\|_{0,D}.
	\end{equation*}
\end{lemma}
\begin{proof}
	See \cite[Lemma~3.30]{ernguermond1}, together with a standard rescaling argument. 
\end{proof}